\documentclass[reqno,12pt]{amsart}

\usepackage[margin=1.23in, top=0.85in, bottom = 0.85in]{geometry}

\usepackage{ifthen}
\usepackage{amssymb, xspace, bm}

\usepackage[pagebackref=true, colorlinks=true, citecolor=blue]{hyperref}

\usepackage[off]{pdfsync}

\usepackage{mathtools}

\newcounter{HaveBBM} 

\ifthenelse{\value{HaveBBM}=1}{
    \usepackage{bbm}
    }{}

\newcounter{dtlForSubmission} 
\newcounter{dtlMarginComments} 
\newcounter{dtlSomeDetail} 
\newcounter{dtlFullDetails} 

\newcounter{DetailLevel} 

\newcommand{\DetailMarginNote}[1]{
    \ifthenelse{\value{DetailLevel}=\value{dtlMarginComments} \or \value{DetailLevel}>\value{dtlMarginComments}}
        {{\small #1}}{}
    }

\newcommand{\DetailSome}[1]{
    \ifthenelse{\value{DetailLevel}=\value{dtlSomeDetail} \or \value{DetailLevel}>\value{dtlSomeDetail}}
        {{\small \textbf{Detailed compile only}: #1}}{}
    }

\newcommand{\DetailFull}[1]{
    \ifthenelse{\value{DetailLevel}=\value{dtlFullDetails} \or \value{DetailLevel}>\value{dtlFullDetails}}
        {{\small \textbf{Detailed compile only}: #1}}{}
    }

\newcommand{\NotDetailSome}[1]{
    \ifthenelse{\value{DetailLevel}=\value{dtlSomeDetail} \or \value{DetailLevel}>\value{dtlSomeDetail}}
        {}{#1}
    }

\newcommand{\NotDetailFull}[1]{
    \ifthenelse{\value{DetailLevel}=\value{dtlFullDetails} \or \value{DetailLevel}>\value{dtlFullDetails}}
        {}{#1}
    }

\newcommand{\DetailSomeElse}[2]{
    \ifthenelse{\value{DetailLevel}=\value{dtlSomeDetail} \or \value{DetailLevel}>\value{dtlSomeDetail}}
        {{\small \textbf{Detailed compile only}: #1}}{#2}
    }

\newcommand{\DetailFullElse}[2]{
    \ifthenelse{\value{DetailLevel}=\value{dtlFullDetails} \or \value{DetailLevel}>\value{dtlFullDetails}}
        {{\small \textbf{Detailed compile only}: #1}}{#2}
    }

\newcommand{\DetailSomeInline}[1]{
    \ifthenelse{\value{DetailLevel}=\value{dtlSomeDetail} \or \value{DetailLevel}>\value{dtlSomeDetail}}
        {{\small #1}}{}
    }

\newcommand{\DetailFullInline}[1]{
    \ifthenelse{\value{DetailLevel}=\value{dtlFullDetails} \or \value{DetailLevel}>\value{dtlFullDetails}}
        {{\small #1}}{}
    }

\newcommand{\DetailSomeElseInline}[2]{
    \ifthenelse{\value{DetailLevel}=\value{dtlSomeDetail} \or \value{DetailLevel}>\value{dtlSomeDetail}}
        {{\small #1}}{#2}
    }

\newcommand{\DetailFullElseInline}[2]{
    \ifthenelse{\value{DetailLevel}=\value{dtlFullD	etails} \or \value{DetailLevel}>\value{dtlFullDetails}}
        {{\small #1}}{#2}
    }

\newcommand{\ExplainDetailLevel}{
    Detail level is
    \ifthenelse{\value{DetailLevel}=\value{dtlForSubmission}}
        {0: for submission}
        {\ifthenelse{\value{DetailLevel}=\value{dtlMarginComments}}
            {1: as for submission but with margin comments}
            {\ifthenelse{\value{DetailLevel}=\value{dtlSomeDetail}}
                {2: some proofs not intended for submission}
               {\ifthenelse{\value{DetailLevel}=\value{dtlFullDetails}}
                   {3: full details}
                   {invalid}
                }
            }
        }
    }

\newcounter{DoAdditionalConstraint} 

\newcommand{\AdditionalConstraint}[2]{
    \ifthenelse{\value{DoAdditionalConstraint}=1}
        {{\small #1}}{#2}
    }

\newtheorem{theorem}{Theorem}[section]
\newtheorem{prop}[theorem]{Proposition}
\newtheorem{lemma}[theorem]{Lemma}
\newtheorem{cor}[theorem]{Corollary}

\theoremstyle{definition}
\newtheorem{definition}[theorem]{Definition}

\newtheorem{remark}[theorem]{Remark}

\numberwithin{equation}{section}

\newcommand{\abs}[1]{\left\vert#1\right\vert}

\DeclarePairedDelimiter{\pabs}{\lvert}{\rvert}
\DeclarePairedDelimiter{\Largepabs}{\Big\lvert}{\Big\rvert}

\newcommand{\BoldTau}
    {\mbox{\boldmath $\tau$}}

\newcommand{\skipline}{\vspace{11pt}}

\newcommand{\BB}[1]{\ensuremath{\mathbb{#1}}}

\newcommand{\R}{\ensuremath{\BB{R}}}
\newcommand{\N}{\ensuremath{\BB{N}}}

\newcommand{\iny}{\ensuremath{\infty}}
\newcommand{\grad}{\ensuremath{\nabla}}
\ifthenelse{\value{HaveBBM}=1}{
    
    }{
    
    }
\DeclareMathOperator{\dv}{div} %
\DeclareMathOperator{\curl}{curl} %
\DeclareMathOperator{\supp}{supp} %
\newcommand{\prt}{\ensuremath{\partial}}
\newcommand{\brac}[1]{\ensuremath{\left[ #1 \right]}}
\newcommand{\pr}[1]{\ensuremath{\left( #1 \right) }}
\newcommand{\set}[1]{\ensuremath{\left\{ #1 \right\}}}

\newcommand{\norm}[1]{\ensuremath{\left\Vert #1 \right\Vert}}
\newcommand{\smallnorm}[1]{\ensuremath{\Vert #1 \Vert}}
\newcommand{\refA}[1]{Appendix~\ref{A:#1}}
\newcommand{\refS}[1]{Section~\ref{S:#1}}
\newcommand{\refSAnd}[2]{Sections~\ref{S:#1} and \ref{S:#2}}
\newcommand{\refSThrough}[2]{Sections~\ref{S:#1} through \ref{S:#2}}
\newcommand{\refT}[1]{Theorem~\ref{T:#1}}
\newcommand{\refTAnd}[2]{Theorems~\ref{T:#1} and \ref{T:#2}}

\newcommand{\refP}[1]{Proposition~\ref{P:#1}}
\newcommand{\refPAnd}[2]{Propositions~\ref{P:#1} and \ref{P:#2}}
\newcommand{\refPThrough}[2]{Propositions~\ref{P:#1} through \ref{P:#2}}
\newcommand{\refL}[1]{Lemma~\ref{L:#1}}
\newcommand{\refLAnd}[2]{Lemmas~\ref{L:#1} and \ref{L:#2}}

\newcommand{\refD}[1]{Definition~\ref{D:#1}}

\newcommand{\refC}[1]{Corollary~\ref{C:#1}}
\newcommand{\refE}[1]{(\ref{e:#1})}
\newcommand{\refEAnd}[2]{(\ref{e:#1}, \ref{e:#2})}
\newcommand{\refEAndAndAnd}[4]{(\ref{e:#1}, \ref{e:#2}, \ref{e:#3}, \ref{e:#4})}
\newcommand{\refEThrough}[2]{(\ref{e:#1}) through (\ref{e:#2})}
\newcommand{\refR}[1]{Remark~\ref{R:#1}}
\newcommand{\eps}{\ensuremath{\varepsilon}}
\newcommand{\Cal}[1]{\ensuremath{\mathcal{#1}}}
\newcommand{\al}{\ensuremath{\alpha}}

\newcommand{\pdx}[2]{\frac{\prt #1}{\prt #2}}
\newcommand{\pdxmixed}[3]{\frac{\prt^2 #1}{\prt #2 \prt #3}}
\newcommand{\pdxn}[3]{\frac{\prt^{#3} #1}{\prt #2^{#3}}}
\newcommand{\diff}[2]{\frac{ d#1}{d#2}}

\newcommand{\ol}{\overline}

\newcommand{\MOC}{modulus of continuity\xspace}
\newcommand{\CapMOC}{Modulus of continuity\xspace}

\newcommand{\ProofStep}[1]{\bigskip\noindent\textbf{#1}}

\iftrue
\newcommand{\PART}[1]{}
\else {
\newcommand{\PART}[1]{
\bigskip
\begin{center}
\textbf{#1}
\phantomsection   
\addcontentsline{toc}{chapter}{#1}
\end{center}
}
\fi

\begin{document}

\raggedbottom

\numberwithin{equation}{section}

\newcommand{\SideNote}[1]{
	    \ifthenelse{\value{DetailLevel}=\value{dtlMarginComments}
	    	\or \value{DetailLevel}>\value{dtlMarginComments}}
		{\pdfsyncstop\marginpar[\raggedleft\tiny #1]{\raggedright\tiny
		#1}\pdfsyncstart}
		{}
		}

\newcommand{\AuthorNote}[3]{\SideNote{\textcolor{blue}{(#1 #2):} #3}}

\newcommand{\JimNote}[2]{\AuthorNote{Jim}{#1}{#2}}

\newcommand{\HMNote}[2]{\AuthorNote{H \& M}{#1}{#2}}

\newcommand{\DavidNote}[2]{\AuthorNote{David}{#1}{#2}}

\newcommand{\MarginNote}[1]{\SideNote{#1}}

\newcommand{\ToDo}[1]{
	    \ifthenelse{\value{DetailLevel}=\value{dtlMarginComments}
	    	\or \value{DetailLevel}>\value{dtlMarginComments}}
		{\textbf{[#1]}}
		{}
		}


\newcommand{\Detail}[1]{
    \MarginNote{Detail}
    \skipline
    \hspace{+0.25in}\fbox{\parbox{4.25in}{\small #1}}
    \skipline
    }

\newcommand{\Holder}
    {H\"{o}lder }

\newcommand{\Holders}
    {H\"{o}lder's }

\newcommand{\wh}{\widehat}

\newcommand{\n}{\bm{n}}


\newcommand{\Ignore}[1] {}

%
%
\newcommand{\ds}{\displaystyle}
\newcommand{\olg}{\ol{\gamma}}
\newcommand{\GG}{G_\Omega}
\newcommand{\GGBC}{G_{\ol{B}^C}}
\newcommand{\KK}{K_\Omega}
\newcommand{\KKBC}{K_{\ol{B}^C}}
\newcommand{\JJ}{J_\Omega}
\newcommand{\JJBC}{J_{\ol{B}^C}}
\newcommand{\KKol}{\ol{K}_\Omega}
\newcommand{\KKolBC}{\ol{K}_{\ol{B}^C}}
\newcommand{\stardot}{\mathop{* \cdot}}

%
%

%
%

\author[Ambrose]{David M. Ambrose}
\address{Department of Mathematics, Drexel University, 3141 Chestnut Street, Philadelphia, PA 19104}
\email{ambrose@math.drexel.edu}

\author[Kelliher]{James P. Kelliher}
\address{Department of Mathematics, University of California, Riverside, 900 University Ave.,
Riverside, CA 92521}
\email{kelliher@math.ucr.edu}

\author[Lopes Filho]{Milton C. Lopes Filho}
\address{Instituto de Matem\'atica, Universidade Federal do Rio de Janeiro, Cidade Universit\'aria -- Ilha do Fund\~ao, Caixa Postal 68530, 21941-909 Rio de Janeiro,  RJ -- Brazil}

\curraddr{}
\email{mlopes@im.ufrj.br}

\author[Nussenzveig Lopes]{Helena J. Nussenzveig Lopes}
\address{Instituto de Matem\'atica, Universidade Federal do Rio de Janeiro, Cidade Universit\'aria -- Ilha do Fund\~ao, Caixa Postal 68530, 21941-909 Rio de Janeiro,  RJ -- Brazil}

\curraddr{} \email{hlopes@im.ufrj.br}



\title
    [Serfati solutions]
    {Serfati solutions to the 2D Euler equations on exterior domains}

\begin{abstract}
We prove existence and uniqueness of a weak solution to the incompressible 2D Euler
equations in the exterior of a bounded smooth obstacle when the initial data is a bounded divergence-free
velocity field having bounded scalar curl. This work completes and extends the ideas outlined by P. Serfati for the same problem in the whole-plane case.   With non-decaying vorticity, the Biot-Savart
integral does not converge, and thus velocity cannot be reconstructed from vorticity in a
straightforward way.  The key to circumventing this difficulty is the use of the Serfati identity,
which is based on the Biot-Savart integral, but holds in more general settings.
\end{abstract}

\date{(compiled on \today)}

\maketitle

\Ignore{ 
\begin{small}
    \begin{flushright}
        Compiled on \textit{\textbf{\today}}
    \end{flushright}
\end{small}
} 

\DetailMarginNote{
    \begin{small}
        \begin{flushright}
            Compiled on \textit{\textbf{\today}}

            \ExplainDetailLevel

        \end{flushright}
    \end{small}
    }

\newcommand{\BoxComment}[1]{
    \skipline
    \hspace{+0.25in}\fbox{\parbox{4.25in}{\small #1}}
    \skipline
    }

\vspace{-0.30in}

{\small \tableofcontents}

\newpage

%
%
\section{Introduction}\label{S:Intro}

The incompressible Euler equations describe the velocity field, $u$, and pressure, $p$, of a constant-density, inviscid fluid. The equations (without forcing) can be written in strong form as
\begin{align} \label{e:EulerEqsVelForm}
	\left\{
		\begin{array}{rl}
			\prt_t u + u \cdot \grad u + \grad p = 0
				& \text{in } \Omega, \\
			\dv u = 0
				& \text{in } \Omega, \\
			u \cdot \n = 0
				& \text{on } \prt \Omega, \\
			u(0) = u_0
				& \text{ in } \Omega.
		\end{array}
	\right.
\end{align}
Here, $\Omega$ is a domain with boundary (empty, if $\Omega = \R^2$) and $\n$ is the outward unit normal to the boundary. The initial velocity, $u_0$, and the solution, $(u, p)$, are assumed to lie in appropriate function spaces. If $\Omega$ is unbounded, some condition at infinity must be imposed.

In two dimensions, the classical well-posedness result for finite-energy \textit{weak} solutions with bounded initial vorticity (the scalar curl of the velocity) is that established by Yudovich in \cite{Y1963} (and extended by him in \cite{Y1995} to allow slightly unbounded vorticities). Yudovich's results are for a bounded domain, but his ideas were adapted to the full plane case, see \cite{Majda1986}. Vishik, in \cite{VishikBesov}, working in the full plane, established a slightly larger uniqueness class of unbounded vorticities.
Each of these full-plane results, however, requires the initial vorticity to have some decay at infinity. This assumption is not natural from the physical point of view, as full plane flow is an approximate model for flow far from boundaries, where no decay of distant vorticity
should be expected.

In 1995, Ph. Serfati stated and outlined proofs of existence and uniqueness of solutions for the incompressible 2D Euler equations in the full plane with each of the initial velocity and initial vorticity bounded \cite{Serfati1995A}. We call such velocity fields, \textit{Serfati velocity fields}. Once no decay of vorticity is assumed, uniform boundedness of vorticity no longer implies boundedness of velocity, so it makes sense to add this condition as an hypothesis.
Our purpose in the present paper is to present a complete proof of Serfati's original result, in the full plane, and extend this
argument to exterior domains.

Until Serfati's 1995 paper, all existence results in an unbounded domain, including the full plane, made key use of the Biot-Savart law to recover the velocity from the vorticity. This law can be expressed in the form,
\begin{align}\label{e:BS}
	\KK[\omega]
		:= \int_\Omega \KK(\cdot, y) \omega(y) \, dy.
\end{align}
Here, $\KK$ is the Biot-Savart kernel,
\begin{align}\label{e:KKDef}
	\KK(x, y) = \grad_x^\perp \GG(x, y),
\end{align}
where $\GG$ is the Green's function for the Dirichlet Laplacian on $\Omega$. Above, $\grad_x^{\perp} = (-\prt_{x_2},\prt_{x_1})$.

When $\Omega$ is the domain exterior to a single, connected, bounded, domain then, given a scalar field, $\omega$, $u = \KK[\omega]$ is the unique divergence-free vector field on $\Omega$, decaying at infinity, with $u \cdot \n = 0$ on $\prt \Omega$, whose scalar curl (vorticity) is $\omega$, and whose circulation about the boundary is $\displaystyle{-\int_{\Omega} \omega}$. (See \refS{SerfatiIdExt} for more details.)

Convergence of the Biot-Savart integral requires, however, membership of $\omega$ in an appropriate space; for instance, $\omega\in L^1 \cap L^\iny$ would be sufficient.
For $\omega$ only in $L^\iny$, the Biot-Savart integral fails to converge. This is the heart of the difficulty in working with Serfati solutions.

Serfati's key insight, which we adopt, is to use, in place of the Biot-Savart law, the identity,
\begin{align}\label{e:SerfatiId}
	\begin{split}
		&u^j(t, x) - (u^0)^j(x) =
				\int_\Omega a(x - y)
					\KK^j(x, y) (\omega(t, y)
						-  \omega^0(y)) \, dy \\ 
			&\qquad- \int_0^t \int_\Omega \grad_y \grad_y^\perp
					\brac{(1 - a(x - y)) \KK^j(x, y)}
			\cdot (u \otimes u)(s, y) \, dy \, ds,
	\end{split}
\end{align}
$j = 1$, $2$, for all $(t, x)$ in $[0, T] \times \Omega$. Here, $a$ is any radially symmetric, smooth, compactly supported cutoff function with $a = 1$ in a neighborhood of the origin. We call \refE{SerfatiId} the \textit{Serfati identity}. (Actually, Serfati never derives or even states this identity, but rather states inequalities that follow from it.)

To motivate our approach and understand how the Serfati identity enters the analysis, let us describe formally how one obtains a bound on the $L^\iny$-norm of the velocity at time $t > 0$ given that both the initial velocity and the initial vorticity are in $L^\iny$.
Assume that $u$ is a solution to the Euler equations (without forcing). Then the vorticity is transported by the velocity with $\prt_t \omega = - u \cdot \grad \omega.$  (This always holds formally for solutions of the 2D Euler equations.)

If $\Omega$ were a bounded and simply connected domain, then applying the Biot-Savart law gives
\begin{align}\label{e:uj1}
	\begin{split}
	u^j(t, x)
		&
		= \int_\Omega \KK^j(x, y) \omega(t, y) \, dy.
	\end{split}
\end{align}
Hence, $\norm{u(t)}_{L^\iny} \le C_1 \smallnorm{\omega^0}_{L^\iny}$, where $C_1 = \smallnorm{\KK(x, \cdot)}_{L^1(\Omega)}$, since vorticity is transported by the velocity.

When working with an unbounded domain, exterior to a single obstacle, however, $\KK(x, \cdot)$ is in $L^1_{loc}(\Omega)$ but not in $L^1(\Omega)$, so $C_1$ is infinite, and this simple estimate fails. One can obtain a more refined estimate, however, by using $\prt_t \omega = - u \cdot \grad \omega$ and integrating by parts to obtain
\begin{align}\label{e:uj2}
	u^j(t, x)
		&= (u^0)^j(x)
			- \int_0^t \int_\Omega \pr{u(s, y) \cdot \grad_y} \grad_y^\perp \KK^j(x, y)
				\cdot u(s, y) \, dy \, ds
\end{align}
for $j = 1, 2$. Unfortunately, this only leads to
\begin{align*}
	\norm{u(t)}_{L^\iny}
		\le \norm{u^0}_{L^\iny}
			+ C_2 \int_0^t \norm{u(s)}_{L^\iny}^2 \, ds,
\end{align*}
where $C_2 = \smallnorm{\grad \grad \KK(x, \cdot)}_{L^1(\Omega)}$. More important,
$C_2$ is also infinite because although $\grad \grad \KK(x, \cdot)$ decays fast enough to be integrable near infinity, it has a singularity like $\abs{y - x}^{-3}$ at $y = x$.

Serfati realized that one can reach a balance between the expressions in \refEAnd{uj1}{uj2} by splitting the Biot-Savart law into near-field and far-field (with respect to $x$) parts. He did this specifically for the full plane, where $\KK(x, y) = K(x - y)$, with
\begin{align}\label{e:KBSFullPlane}
	K(x)
		:= \frac{x^\perp}{2 \pi \abs{x}^2}.
\end{align}
This leads to \refE{SerfatiIdConv} (see \refS{SerfatiIdR2}). From \refE{SerfatiIdConv}, with $a_\eps(\cdot) = a(\cdot/\eps)$ in place of $a$, he obtains, in effect,
\begin{align}\label{e:C1C2}
	\begin{split}
	C_1
		&= C_1(\eps)
		 = \smallnorm{a_\eps(x - \cdot) K(x -\cdot)}_{L^1(\Omega)}
		 \le C \eps, \\
	C_2
		&= C_2(\eps)
		= \smallnorm{\grad \grad [(1 - a_\eps(x - \cdot) )K(x - \cdot)]}_{L^1(\Omega)}
		 \le C \eps^{-1},
	\end{split}
\end{align}
and thereby the inequality,
\begin{align}\label{e:uSerfatiBound}
	\norm{u(t)}_{L^\iny}
		&\le \norm{u^0}_{L^\iny}
			+ C \eps
			+ \frac{C}{\eps} \int_0^t \norm{u(s)}_{L_\iny}^2 \, ds.
\end{align}
Letting
$
	\eps
		= \eps(t)
		= \pr{\int_0^t \norm{u(s)}_{L_\iny}^2 \, ds}^{1/2},
$
one obtains a bound on the $L^\iny$-norm of $u(t)$, as we do in \refS{ExistenceR2}.

Of course, this gives only a formal, a priori bound on the $L^\iny$-norm of the velocity, but this bound is the key to the proof of existence. Serfati's identity is also key to the proof of uniqueness, where, however, only one fixed value of $\eps > 0$ is needed.

If one tries to directly extend Serfati's approach to an exterior domain, a number of technical difficulties arise, but the key problem is that the estimates in \refE{C1C2} no longer hold. Let us consider explicitly the case where $\Omega = \ol{B}^C$ is the exterior of the closed unit disk. With $K$ as in \refE{KBSFullPlane}, the Biot-Savart kernel for all of $\R^2$, it is classical that
\begin{align}\label{e:KK}
	\KKBC(x, y)
		= K(x - y) - K(x - y^*),
\end{align}
where $y^* = y/\abs{y}^2$.

The best that can be obtained is $C_1(\eps) \le C(\eps + \eps^2)$ and $C_2(\eps) \le C(1 + \eps^{-1})$. This is sufficient for uniqueness but not for existence. (With a slight variation on the estimate of $C_2$, local-in-time existence can be established, however.) The reason for the failure of the estimate in \refE{C1C2} boils down to the failure of $\KKBC(x, y)$ to decay as $\abs{y} \to \iny$. To remedy this, we add the function, $K(x)$, to \refE{KK}, giving the \textit{modified Biot-Savart kernel},
\begin{align*} 
	\JJBC(x, y) = \KKBC(x, y) + K(x),
\end{align*}
which decays as $\abs{y} \to \iny$.

More generally, for the exterior of a bounded, connected and simply connected obstacle, we define
\begin{align}\label{e:JJDef}
	\JJ(x, y) = \KK(x, y) + \KKol(x),
\end{align}
where $\KKol$ is the unique divergence-free vector field tangential to $\prt \Omega$ having circulation one and decaying at infinity. An explicit form for $\KKol$ is given in \refE{olK} (from which we can see that $\KKolBC = K$ on $\overline{B}^C)$.

If one uses $\JJ$ in place of $\KK$ in the Biot-Savart law \refE{BS}, one obtains a different solution operator for the system

\begin{equation} \label{ellipsys}
\left\{\begin{array}{ll}
\dv u = 0, & \mbox{ in }\Omega,\\
\curl u = \omega, & \mbox{ in }\Omega,\\
u \cdot \n = 0,& \mbox{ on } \prt\Omega.
\end{array}\right.
\end{equation}
In fact the modified Biot-Savart law, with $\JJ$ in place of $\KK$, gives the unique solution to the system \eqref{ellipsys} which has {\em vanishing circulation} around $\prt \Omega$. This modified Biot-Savart law was first introduced by C. C. Lin in \cite{CCLin41}, as the perpendicular gradient of the {\em hydrodynamic Green's function}. For more details on the hydrodynamic Biot-Savart law, see \cite{LopesANDLopes2010}.

Unlike $\KK(x, \cdot)$, the kernel $\JJ(x, \cdot)$ no longer vanishes on the boundary; nonetheless, it turns out that using $\JJ$ in place of $\KK$ only disrupts the Serfati identity slightly, leading to \refE{SerfatiIdJ} below.

The equivalent of \refE{C1C2} for $\JJ$ in place of $\KK$ holds (or holds almost; see \refP{JKBounds} below), and we use these bounds in the same manner as Serfati to obtain \refE{uSerfatiBound}. Now, however, we must also account for a boundary integral in \refE{SerfatiIdJ}, and this we cannot do simply by bounding the size of the Biot-Savart kernel. It is dealt with easily enough, however, by a simple estimate (see \refP{JBoundaryTermBound}).


\bigskip

This paper is organized as follows: We state our results in \refS{Results}. We give the proof of existence separately for the full plane in \refS{ExistenceR2} and for an exterior domain in \refS{ExistenceExt}. In each of these sections, we start by deriving the Serfati identity for the given type of domain then give the existence proofs. We rely on estimates on the Biot-Savart kernel developed later in \refS{BSEstimates}. We also rely on these estimates in \refS{UniquenessAndCont}, where we prove uniqueness in \refS{Uniqueness}, extending the argument to give a type of continuous dependence on initial data in \refS{ContinuousDependence}.

The Biot-Savart kernel estimates of \refS{BSEstimates} are developed first for the full plane in \refS{BSKernelFullPlane}, then for the exterior to the unit disk in \refS{BSKernelExteriorDisk}, and finally for the exterior to a single obstacle in \refS{BSKernelExteriorSingleObstacle}, each section building on the previous.

Examples of Serfati velocities are given in \refS{Examples}.

There are other proofs of existence and uniqueness in the full plane for Serfati initial velocity, most notably those of \cite{Taniuchi2004, TaniuchiEtAl2010}. The proofs, however, employ Littlewood-Paley decompositions for existence and paradifferential calculus for uniqueness, and so are not easily extended to other domains. We discuss these issues at some depth in \refS{ConcludingRemarks}.

In \refA{InitData}, we show how to prepare a sequence of initial velocities that are smooth with compactly supported vorticity and that converge in an appropriate sense to a given bounded initial velocity having bounded vorticity. (This approximate sequence is employed in \refSAnd{ProofExistenceR2}{ProofExistenceExt} to obtain existence of solutions.)

\section{Statement of results}\label{S:Results}

\noindent  The purpose of this section is to give precise statements of the main results in this work: existence, uniqueness, and a mild form of continuous dependence of solutions on initial data. We will treat two very different fluid domains---the full plane and domains exterior to a single obstacle. To be more precise, we will denote the fluid domain by $\Omega$, be it all of $\R^2$ or the exterior of a single connected and simply connected bounded domain with a $C^{\infty}$ boundary. In the latter case let $\bm{n}$ denote the unit exterior normal to $\Omega$ at the finite boundary $\Gamma$. (For notational convenience we set $\Gamma = \emptyset$ when considering full plane flow.) We let $\BoldTau$ denote the unit tangent vector, oriented so that
\begin{align*}
    \BoldTau
        = - \n^\perp
        := - (-n_2, n_1)
        = (n_2, - n_1).
\end{align*}

We begin with basic definitions concerning the type of velocity field we are interested in and the notion of weak solution of the Euler equations we will consider.

If $u$ is a vector field on $\Omega$, we write
\begin{align*}
	\omega(u)
		:= \curl u
		= \prt_1 u^2 - \prt_2 u^1
\end{align*}
for the scalar curl (vorticity) of $u$. Following a common convention, we sometimes write $\omega$ for $\omega(u)$ when $u$ is understood; that is, we use $\omega$ both as a function on vector fields and as the value of that function.

It is classical that, by taking the scalar curl of the two dimensional incompressible Euler equations \refE{EulerEqsVelForm}, we obtain the vorticity equation, or the vorticity formulation of the Euler equations:
\begin{align} \label{e:EulerEqsVortForm}
	\left\{
		\begin{array}{rl}
			\prt_t \omega + u \cdot \grad \omega = 0
				& \text{in } \Omega, \\
			\dv u = 0
				& \text{in } \Omega, \\
			\curl u = 0
				& \text{in } \Omega, \\
			u \cdot \n = 0
				& \text{on } \prt \Omega, \\
			\omega(0) = \omega_0 = \curl u_0
				& \text{ in } \Omega.
		\end{array}
	\right.
\end{align}

\begin{definition}\label{D:SerfatiVelocity}
	We say that a divergence-free vector field $u \in L^\iny(\Omega)$
	with $u \cdot \n = 0$ on $\Gamma$ and
	$\omega(u) \in L^\iny(\Omega)$ is a \textit{Serfati velocity}.
	We denote by $S = S(\Omega)$ the Banach space of all Serfati velocity fields
	with the norm,
	\begin{align*}
		\norm{u}_S
			= \norm{u}_{L^\iny} + \norm{\omega(u)}_{L^\iny}.
	\end{align*}
\end{definition}
\begin{remark}\label{R:Trace}
	Since $u \in L^2_{loc}$ and $u$ is divergence-free, the trace of
	its normal component, $u \cdot \n$, is well-defined and belongs
	to $H^{-1/2}(\Gamma)$ (see, for instance,
	Theorem I.1.2 of \cite{T2001}).
\end{remark}

\begin{definition}\label{D:ESol}
	Fix $T > 0$. Assume that
	$u \in L^\iny(0, T; S) \cap C([0, T] \times \Omega)$ and
	let $\omega = \omega(u)$.
	We say that $u$ is a \textit{Serfati solution} to the Euler equations
	without forcing and with initial velocity $u^0 = u|_{t = 0}$ in $S$
	if the following conditions hold:
	\begin{enumerate}
		\item
			The vorticity equation
			$\prt_t \omega + u \cdot \grad \omega = 0$ (see \refE{EulerEqsVortForm}$_1$) holds in the sense
			of distributions.
			
		\smallskip
			
		\item
			For any radially symmetric, smooth, compactly supported
			cutoff function $a$
			with $a = 1$ in a neighborhood of the origin the Serfati
			identity in \refE{SerfatiId} holds.
			
		\smallskip
		
		\item
			The vorticity $\omega$ is transported by the flow map for $u$
			(see \refR{FlowMap}).
	\end{enumerate}
	When $\Omega$ is the exterior of a single obstacle, we require, in addition,
	that the circulation of the velocity about the boundary be conserved over time.
\end{definition}

\begin{remark}\label{R:FlowMap}
	The flow map, mentioned in the definition above, is well-defined if $u \in
	 L^\iny(0, T; S) \cap C([0, T] \times \Omega)$. Indeed, such $u$ has a log-Lipschitz \MOC, $\mu$, in space, uniformly over $(0,T)$; see \refL{Morrey}.
	It follows that there exists a unique measure-preserving classical flow map,
	$X \colon (0,T) \times \Omega \to \Omega$, for $u$.
	Property (3) of \refD{ESol} then means that
	$\omega(t, X(t, x)) = \omega^0(x)$ for all $(t, x)$ in $(0,T) \times \Omega$.
\end{remark}

\Ignore{ 
\begin{remark}
    Identity \refE{SerfatiId} is implicit in Serfati's work, \cite{Serfati1995A};
    we will refer to it hereafter as the \textit{Serfati identity}.
\end{remark}
} 

Our main results are \refTAnd{Existence}{ContDep}, in which we establish the existence, uniqueness, and a limited form of continuous dependence on initial data for Serfati solutions. We begin with the statement of existence and uniqueness.

\begin{theorem}\label{T:Existence}
	Let $T > 0$. Assume that $u^0 \in S$. Then there exists  a unique,
	Serfati solution $u$ to the Euler equations as in \refD{ESol}.
	Moreover, there exists a unique flow map $X=X(t,\cdot)$ for $u$;
	that is, $X \colon [0,T]\times\Omega \to \Omega$ and
	\begin{alignat*}{2}
		\prt_t X(t, x) &= u(t, X(t,x)),
            &\quad& t \in (0,T), x\in \Omega \\
		X(0,x) &= x,  && x \in \Omega.
	\end{alignat*}
	The flow map is measure-preserving and
	$X(t, \cdot)\in C^{\beta(t)}$,
	where $\beta(t) = e^{-\al \abs{t}}$
	and $\al = C \norm{u}_{L^\iny(0, T; S)}$.
\end{theorem}

\begin{remark}\label{R:Pressure}
It is shown in \cite{K2013} that the solutions constructed in \refT{Existence} are also distributional solutions of the  velocity formulation of the Euler equations \refE{EulerEqsVelForm}. Moreover, there exists an associated pressure whose asymptotic behavior is $O(\log \abs{x})$ for large $\abs{x}$ and whose gradient is bounded (this is done in \cite{K2013} both for the full plane and an exterior domain).
\end{remark}

The following is a statement that Serfati solutions depend continuously, in $L^{\infty}$-norm, on the (Serfati) initial data. We will need additional notation to state the result.

For any $p \in [1, \iny]$, $L^p_{uloc}(\Omega)$ is the uniformly local $L^p$ space; that is, the space of all measurable functions whose norm,
\begin{align*}
	\norm{f}_{L^p_{uloc}(\Omega)}
		:= \sup_{U \subset \Omega, \abs{U} \le C_0} \norm{f}_{L^p(U)},
\end{align*}
is finite, where $C_0$ is an arbitrary fixed positive constant and $\abs{U}$ is the Lebesgue measure of $U$. For any $p \in [2, \iny]$, let
\begin{align}\label{e:Sp}
	S^p
		= \set{u \in (L^\iny(\Omega))^2 \colon \dv u = 0, \,
			\omega(u) \in L^p_{uloc}(\Omega), \,
			u \cdot \n = 0 \text{ on } \Gamma}.
\end{align}
Then $S^p$ is a Banach space under the norm $\norm{u}_{S^p} = \norm{u}_{L^\iny} + \norm{\omega(u)}_{L^p_{uloc}}$. (We require $p \ge 2$ so that $u \cdot \n$ is well-defined, as in \refR{Trace}.)

\begin{theorem}\label{T:ContDep}
Let $u_1, u_2$ be Serfati solutions to the Euler equations for a fixed $T > 0$
and let $p \in (2, \iny]$.
Let $u_1^0$, $u_2^0$ be the initial velocities with $u^0_1 - u^0_2 \in S^p$.
For all sufficiently small $s_0 = \smallnorm{u^0_1 - u^0_2}_{S^p}$ there exist $C > 0$ such that
\begin{align}\label{e:u1u2BoundExplicit}
	\begin{split}
	&\norm{u_1(t) - u_2(t)}_{L^\iny}
		\le C e^{Ct} s_0
			- C(1 + t) e^{Ct} (C s_0 t)^{e^{-Ct(1 + t)}} \log(C s_0 t)
	\end{split}
\end{align}
for all $t$ in $[0, T]$, where $C$ depends on
$T$ and $p$.
\end{theorem}

\begin{remark}
	The last term in \refE{u1u2BoundExplicit} goes to zero as $s_0 \to 0^+$
	since
	$
		\lim_{r \to 0^+} r^\al \log r = 0
	$
	for any $\al > 0$.
\end{remark}

\refT{ContDep} is not a statement that Serfati solutions depend continuously on initial data since we establish only that
the velocities, at future time, are close in $L^{\infty}$ if they are close at time zero in $S^p$, but we provide no information on the $S$ norm of their difference in terms of their initial $S$ norm.

In fact, we should not expect continuous dependence in the $L^\iny$ norm of the vorticity, and hence not in $S$. For instance, a small initial perturbation of a vortex patch will displace the contour and result in a large discrepancy between perturbed and unperturbed solutions, relative to the $L^{\infty}$-norm of vorticity, for any positive time.

%
%

\PART{Part I: Existence in the full plane}

\section{Existence in the full plane}\label{S:ExistenceR2}

\noindent Existence of weak solutions for the incompressible 2D Euler equations has been established under many different kinds of regularity assumptions. The proofs follow a standard strategy consisting in first generating a sequence of approximations,  then establishing enough  a priori estimates to show that the sequence is compact in an appropriate function space and, finally, passing to the limit in the weak form of the Euler equations. To obtain compactness, a priori estimates are needed for both velocity and vorticity. Whenever the function space is based on a rearrangement invariant space, the vorticity estimates are immediate, as future vorticity is simply a rearrangement of its initial values. One then establishes velocity estimates by integrating vorticity estimates using the Biot-Savart law, which relates vorticity to velocity through a Biot-Savart kernel (see \cite{LopesLopesTadmor2000} for details). For an unbounded fluid domain, this kernel has very mild decay at infinity. Hence, in order to ensure that the Biot-Savart law is well-defined it is necessary to impose decay of vorticity at infinity. It turns out that this is the only reason to impose decay of vorticity at infinity.

In his pioneering work, \cite{Serfati1995A}, Ph. Serfati got around the lack of decay of the Biot-Savart kernel, and managed to establish existence and uniqueness for 2D Euler, in the full plane, for bounded velocity fields whose vorticity is merely bounded and does not obey any decay assumptions at infinity. To do so, it was necessary to introduce an alternative to the Biot-Savart law, which we refer to as the Serfati identity. One of the main uses of the Serfati identity is to play the role of the Biot-Savart law in obtaining a priori estimates for the velocity given estimates on vorticity. Serfati's proofs were terse and incomplete, but brilliant.

The existence of weak solutions to the 2D Euler equations in the full plane with Serfati initial data is included in the work of Taniuchi \cite{Taniuchi2004}. His proof, however, makes extensive use of the group structure of $\R^2$, through the use of Littlewood-Paley theory, and thus does not extend to domains with boundary.

In this section we present a complete proof of existence of Serfati solutions in the case of the full plane, inspired by Serfati's work. We will see that, in contrast to Taniuchi's proof, Serfati's ideas extend naturally to other unbounded fluid domains.

%
%
\subsection{The Serfati identity in the full plane}\label{S:SerfatiIdR2}

As mentioned, the Serfati identity is an alternative to the Biot-Savart law
in that it provides a priori estimates on the velocity.

Let us begin by recalling the Biot-Savart law in the full plane. Let $\omega \in C_c^\iny(\R^2)$, where the subscript $c$ in $C^{\iny}_c$ means functions in $C^{\iny}$ with compact support. The Biot-Savart law relating a divergence-free velocity field $u$ to its vorticity $\omega$ is given by
\begin{align*}
    u(t, x)  = \int_\Omega K(x - y) \omega(t, y) \, dy,
\end{align*}
where
\begin{align*}
    K(x) = \frac{1}{2 \pi} \frac{x^\perp}{\abs{x}^2}
\end{align*}
is the Biot-Savart kernel and where, for $x = (x_1, x_2)$,
\begin{align*}
    x^\perp
        := (- x_2, x_1).
\end{align*}

In view of the mild decay of the kernel $K$ at infinity, together with the singularity at the origin, it is easy to check that
\begin{align}
K \in L^p_{loc}(\R^2), \quad \text{for all} 1\leq p < 2,\qquad K \in L^q(\R^2\setminus B_R(0)) \text{ for all } q > 2.
\end{align}
Hence, to ensure that the integral in the Biot-Savart law is finite it is necessary to have
\[\omega\in L^{p'}\cap L^{q'},\quad \mbox{ for some } p'>2 \mbox{ and } 1\leq q' < 2.\]
This is precisely the kind of hypothesis which is avoided by using the Serfati identity.

It is useful to introduce the notation $\stardot$ defined below:
\begin{align*}
    \begin{array}{ll}
    v \stardot w
 		= v^i * w^i
 	        &\mbox{if $v$ and $w$ are  vector fields}, \\
    A \stardot B
		= A^{ij} * B^{ij}
		    &\mbox{if $A$, $B$ are matrix-valued functions on $\R^2$},
    \end{array}
\end{align*}
where $*$ denotes convolution.
We have adopted the convention that repeated indices are implicitly summed.

Let $f$ be a scalar field and $v$ a vector field. Then, using the notation introduced above, we have
\begin{align}\label{e:CurlIBP}
        f * \curl v
            &= f * (\prt_1 v^2 - \prt_2 v^1)
            = \prt_1 f * v^2 - \prt_2 f * v^1
            = \grad^\perp f \stardot v
\end{align}
and
\begin{align}\label{e:gradperpdivIBP}
    \begin{split}
        \grad^\perp f \stardot \dv (v \otimes v)
            &= -\prt_2 f * \prt_j(v^1 v^j) + \prt_j \prt_1 f
                * (v^2 v^j) \\
            &= - \prt_j \prt_2 f * (v^1 v^j) + \prt_j \prt_1 f
                *(v^2 v^j) \\
            &= \grad \grad^\perp f \stardot (v \otimes v).
    \end{split}
\end{align}

\noindent \begin{prop}\label{P:SerfatiKeyR2}

Let $u$ be a $C^\iny$ classical solution to the Euler equations with initial vorticity, $\omega^0$, compactly supported. Then, for any radially symmetric function $a \in C_c^\iny(\R^2)$ such that $a = 1$ in a neighborhood of the origin, the following identity holds true:
\begin{align}\label{e:SerfatiIdConv}
    \begin{split}
        u^j(t&) - (u^0)^j
            = (a K^j) *(\omega(t) - \omega^0) \\
        &\qquad- \int_0^t \pr{\grad \grad^\perp \brac{(1 - a) K^j}}
        \stardot (u \otimes u)(s) \, ds,
            \quad j = 1, 2.
	\end{split}
\end{align}

\end{prop}

\begin{remark}
It is easy to check that \refE{SerfatiIdConv} corresponds exactly to the Serfati identity \refE{SerfatiId} when $\Omega = \R^2$.
\end{remark}

\begin{proof}
For classical solutions, the vorticity is transported by the flow, so since it is initially compactly supported it remains so for all time. This fact and the smoothness of the solution justify the calculations that follow.

For $j = 1, 2$, we have,
\begin{align*}
    \prt_t u^j
        &= \prt_t (K^j *\omega)
        = \prt_t (aK^j *\omega) + \prt_t ((1 - a)K^j *\omega).
\end{align*}
We integrate in time to get
\begin{align}\label{e:SerfatiIdStartR2}
    \begin{split}
    u(t, x)
        &= u^0(x) + \int_0^t \prt_s (K^j *\omega)(s, x) \, ds \\
        &= u^0(x)
            + \int_0^t \prt_s \brac{(a K^j) * \omega(s, x)} \, ds \\
            &\qquad\qquad\qquad
            + \int_0^t ((1 - a) K^j) * \prt_s \omega(s, x) \, ds \\
        &= u^0(x) + (a K^j)*(\omega(t) - \omega^0)(x) \\
            &\qquad\qquad\qquad
            + \int_0^t ((1 - a) K^j) * \prt_s \omega(s, x) \, ds.
    \end{split}
\end{align}
We now treat the final integrand. We have:
\begin{align*}
    ((1 - a) &K^j) * \prt_s \omega \\
        &=
            - ((1 - a) K^j) *
                (u \cdot \grad \omega) \\
        &= - ((1 - a) K^j) *\curl (u \cdot \grad u) \\
        &= - \grad^\perp ((1 - a) K^j)
            \stardot (u \cdot \grad u) \\
        &= - \grad^\perp((1 - a) K^j)
            \stardot (\dv u \otimes u) \\
        &= - \grad \grad^\perp ((1 - a) K^j)
            \stardot (u \otimes u),
\end{align*}
where we used the vorticity equation $\prt_s \omega + u \cdot \grad \omega = 0$, the identity $u \cdot \grad \omega = \curl (u \cdot \grad u)$, and \refEAnd{CurlIBP}{gradperpdivIBP}. Substituting this back into \refE{SerfatiIdStartR2} yields \refE{SerfatiIdConv}.

\end{proof}

The main advantage of the Serfati identity is that the localized kernel $aK$ is integrable, while $\nabla\nabla^{\perp}((1-a)K)$ is smooth and has behavior asymptotic to $|x|^{-3}$ for large $x$, which is also integrable in $\R^2$. Hence, if both $u \in L^{\iny}$ and $\omega\in L^{\iny}$, uniformly in time, then both terms on the right-hand-side of \refE{SerfatiIdConv} converge. In other words, the Serfati identity is well-defined for velocity fields $u \in L^{\iny}(S)$.

\begin{remark}
The Serfati identity \eqref{e:SerfatiId} or its full plane version \refE{SerfatiIdConv} is intended as a dynamic renormalization to substitute for the classical Biot-Savart law, in situations where the Biot-Savart law is not well-defined. One could instead
consider the following renormalization:
\begin{align} \label{e:ABS}
		u(t, x)
			= u^0(x) +
				\lim_{R \to \iny} (a_R \KK)*
				    (\omega(t) - \omega^0)(x),
\end{align}
where $a_R=a_R(z)\equiv a(z/R)$ and $a$ is as in property (2) of Definition \ref{D:ESol}. This identity is also well-defined for Serfati velocities, and it is easy to see that all Serfati solutions satisfy \eqref{e:ABS}. In fact, as is shown in \cite{K2013}, for the full plane, \refE{ABS} is interchangeable with \refE{SerfatiId} in property (2) of \refD{ESol}.
\end{remark}

%
%
\subsection{Proof of existence in the full plane}\label{S:ProofExistenceR2}

\noindent
The proof of existence proceeds in several steps. One begins by producing a sequence of smooth approximations and by establishing a priori estimates, in $S$, for this sequence. It is easy to
show that the approximate vorticities are uniformly bounded in $L^{\infty}$, but the uniform estimate
on the $L^{\infty}$-norm of the approximate velocities relies completely on the Serfati identity \refE{SerfatiIdConv}. This is the heart of the proof of existence. The subsequent steps are, by now, nearly standard, as set forth in Chapter 8 of \cite{MB2002}: we show that the velocities are uniformly log-Lipschitz, from which compactness of the flow maps follows; we choose a fixed converging subsequence of the approximate flow maps. Next, we define a {\it candidate} limit vorticity and show that, along this subsequence, the approximate vorticities converge to the candidate limit vorticity, uniformly in time, in $L^p_{loc}(\R^2)$ for all $1\leq p < \infty$. We then use the Serfati identity \refE{SerfatiIdConv} again to show that, along the same subsequence, the corresponding approximate velocities converge uniformly to a continuous velocity field. Next we show that the limit velocity and its associated  vorticity are solutions, in the sense of distributions, of the incompressible 2D Euler equations and, finally, that they satisfy the Serfati identity \refE{SerfatiIdConv}. We conclude the proof of existence by establishing additional regularity of the limit velocity---we show it is log-Lipschitz in space, uniformly in time. 

Our existence proof contains elements of the  proof, given in Theorem 8.1 of \cite{MB2002}, of the existence of solutions to the Euler equations with classical Yudovich initial vorticity ($L^1 \cap L^\iny$), and of Serfati's  proof of existence of (Serfati) solutions found in \cite{Serfati1995A}. We depart from both, however, in that we use the Serfati identity to show that our approximating sequence of velocities is Cauchy rather than establishing some time continuity of the velocities. (Serfati uses properties of the pressure to show the continuity, in time, of the approximate velocities.)

\begin{proof}[\textbf{Proof of existence in \refT{Existence} for the full plane}]
Let $u^0 \in S$ and assume that $u^0$ does not vanish identically; otherwise, there is nothing to prove.

\ProofStep{Step 1. Construct approximating sequence.}
We construct the sequence of approximations by generating a smooth sequence of vector fields which approximate the initial data and, afterwards, by exactly solving the Euler equations with the smooth data.

Let $(u_n^0)_{n=1}^\iny$ and $(\omega_n^0)_{n=1}^\iny$ be the approximating sequences to the initial velocity, $u^0$, and initial vorticity, $\omega^0$, given by \refP{InitData}. By hypothesis, $u^0$ is not identically zero, which means that $u^0_n$ does not vanish identically either. Let $u_n$ be the classical, smooth solution to the Euler equations with initial velocity $u_n^0$, and with initial vorticity, $\omega_n^0$. The existence and uniqueness of such solutions follows, for instance, from \cite{McGrath1967} and references therein. Finally, let $\omega_n = \curl u_n$.

\ProofStep{Step 2. Bound velocities in $L^\iny([0, T] \times \Omega)$.}
We begin with the a priori estimate,
\begin{align}\label{e:omegaunLInfBoundR2}
        \norm{\omega_n}_{L^\iny(\R \times \R^2)}
            \le \smallnorm{\omega_n^0}_{L^\iny},
\end{align}
on the vorticity, which can be deduced from the fact that the smooth vorticity is transported by a smooth, divergence-free vector field.
We also have, by construction,
\begin{align} \label{e:omegan0un0LInfBoundR2}
           \smallnorm{u_n^0}_{L^\iny}
            \le C \smallnorm{u^0}_{L^\iny},
     \qquad \smallnorm{\omega_n^0}_{L^\iny}
            \le C \smallnorm{\omega^0}_{L^\iny}.
\end{align}

Next, we will use the Serfati identity \refE{SerfatiIdConv} for the fluid domain $\R^2$, with $u_n$, $\omega_n$ in place of $u$, $\omega$.

Let $a$ be any smooth, compactly supported cutoff function that is equal to $1$ in a neighborhood of $0$ (see item (2) of \refD{ESol}). Fix  $\eps > 0$, to be  specified later. Set
\begin{align} \label{e:aeps}
a_\eps = a_\eps (x) = a\left(\frac{x}{\eps}\right).
\end{align}
From \refE{SerfatiIdConv}, for $u_n$, $\omega_n$, it follows that
\begin{align*}
	\abs{u_n(t, x)}
		\le &\abs{u_n^0(x)}
			+ \abs{(a_\eps K^j) *(\omega_n(t) - \omega_n^0)} \\
		&
			+ \int_0^t \abs{\pr{\grad \grad^\perp \brac{(1 - a_\eps) K^j}}
        \stardot (u_n \otimes u_n)(s)} \, ds.
\end{align*}

We will make use of the detailed estimates on the Biot-Savart kernel derived in \refS{BSKernelFullPlane}.

Applying Young's convolution inequality, followed by the localized estimates on the Biot-Savart kernel in $\R^2$ contained in \refP{JKBounds}, we conclude that
\begin{align}\label{e:BasicunBoundR2}
    \begin{split}
	\norm{u_n(t)}_{L^\iny}& \le
\smallnorm{u^0}_{L^\iny}  + (\smallnorm{\omega_n(t)}_{L^\iny}+
\smallnorm{\omega^0}_{L^\iny} )\smallnorm{a_\eps K}_{L^1} \\
		& \qquad
+ \int_0^t \smallnorm{\nabla\nabla^{\perp}[(1-a_\eps)K]}_{L^1} \norm{u_n(s)}_{L^\iny}^2 \, ds \\
        &\le C \smallnorm{u^0}_{L^\iny}
			+ C \eps\smallnorm{\omega^0}_{L^\iny}
			+ \frac{C}{\eps} \int_0^t \norm{u_n(s)}_{L^\iny}^2 \, ds,
    \end{split}
\end{align}
where we also used \refEAndAndAnd{omegaunLInfBoundR2}{omegan0un0LInfBoundR2}{aKBound}{D2KBound} in the last inequality.

Observe that we can choose $\eps > 0$ arbitrarily, even allowing it to depend on $t$, for each fixed $t$.  Let
\begin{align*}
	\eps
		= \eps(t)
		= \pr{\int_0^t \norm{u_n(s)}_{L_\iny}^2 \, ds}^{1/2}.
\end{align*}
We obtain
\begin{align*} 
	\norm{u_n(t)}_{L^\iny}
		&\le C
			+ C \pr{\int_0^t \norm{u_n(s)}_{L_\iny}^2 \, ds}^{1/2},
\end{align*}
so that
\begin{align*} 
	\norm{u_n(t)}_{L^\iny}^2
		&\le C
			+ C \int_0^t \norm{u_n(s)}_{L_\iny}^2 \, ds.
\end{align*}
We conclude from Gronwall's lemma that
\begin{align} \label{e:unifbdun}
	\norm{u_n(t)}_{L^\iny}
		\le C e^{Ct}.
\end{align}
Therefore, $u_n$ lies in $L^\iny([0, T] \times \Omega)$ for any $T > 0$, with a bound that is uniform in $n$. This, together with \refEAnd{omegaunLInfBoundR2}{omegan0un0LInfBoundR2}, yields
\begin{align} \label{step2est}
\|u_n(t)\|_S \leq C
\end{align}
for some $C = C(T, u_0) > 0$ and for all $0\leq t \leq T$.

\bigskip

Recall the definition of the space of log-Lipschitz functions $LL$ on a domain $U \subseteq \R^2$:
\begin{align} \label{e:LL}
 LL(U) = \left\{ f \in L^{\iny}(U) \;\Big|\; \sup_{x\neq y}
 \frac{|f(x)-f(y)|}{(1+\log^+|x-y|)|x-y|} <\infty \right\},
\end{align}
where $\log^+(z)=\max\{-\log z, 0\}$. This is a Banach space under the norm given by
\[\|f\|_{LL}:= \|f\|_{L^{\infty}} + \sup_{x\neq y} \frac{|f(x)-f(y)|}{(1+\log^+|x-y|)|x-y|}.\]

\ProofStep{Step 3. Log-Lipschitz bound on \MOC of $(u_n)$ uniform in $n$.}
We have
\[
	\|u_n(t)\|_{LL}
		\le C\|u^0\|_{S}.
\]
This follows immediately from \refL{Morrey} together with the a priori estimate \eqref{step2est} on $\|u_n\|_S$ .

\ProofStep{Step 4. Convergence of flow maps.} Associated to each (smooth) $u_n$ there is a unique (smooth) forward flow map, $X_n$, that is, a map $X_n:[0,T]\times\R^2 \to \R^2$ such that
\begin{align*}
    \left\{
    \begin{array}{rl}
        \partial_t X_n(t,x) = u_n(t,X_n(t,x))
            & \text{for } (t, x) \in (0, T) \times \R^2, \\
        X_n(0,x) =x
            & \text{for } x \in \R^2.
    \end{array}
    \right.
\end{align*}
As $\dv u_n = 0$, the flow maps are measure-preserving.

Let $X_n^{-1} (t,\cdot)$ denote the inverse flow map, also known as the {\em back-to-labels} map.

Let $L$ be a compact subset of $[0, T] \times \R^2$. The  log-Lipschitz \MOC of $u_n$, uniform in $n$ and $t$, implies the following estimates for the forward and for the inverse flow maps:
\begin{align*}
	\abs{X_n(t,x_1)-X_n(t,x_2)}
		&\le C \abs{x_1-x_2}^{e^{-\norm{u_n}_{LL} \abs{T}}}, \\
	\abs{X_n^{-1}(t,y_1)-X_n^{-1}(t,y_2)}
		&\le C \abs{y_1-y_2}^{e^{-\norm{u_n}_{LL} \abs{T}}}.
\end{align*}
We also have
\begin{align*}
    \abs{X_n(t_1, x) - X_n(t_2, x)}
        &\le \norm{u_n}_{L^\iny([0, T] \times \Omega)} \abs{t_1 - t_2}
        \le C \abs{t_1 - t_2}, \\
	\abs{X_n^{-1}(t_1, y) - X_n^{-1}(t_2, y)}
		&\le \norm{u_n}_{L^\iny([0, T] \times \Omega)} \abs{t_1 - t_2}^{e^{-\norm{u_n}_{LL} \abs{T}}}
		\le C \abs{t_1 - t_2}^{e^{-\norm{u_n}_{LL} \abs{T}}},
\end{align*}
see Lemma 8.2 in \cite{MB2002}.
These estimates, uniform in $n$ and $0\leq t \leq T$, yield, using the Arzela-Ascoli theorem, a subsequence of the forward flow map converging uniformly on compact subsets of $L$ to a limit forward flow map, $X$, which is measure-preserving and along which vorticity is transported, just as in Proposition 8.2 of \cite{MB2002}. A simple diagonalization argument gives a subsequence that converges uniformly on any compact subset $L$ of
$[0, T] \times \R^2$. We relabel this subsequence, $(X_n)$. Clearly, the limit flow map $X$ also satisfies the H\"older estimates above.

A similar argument can be used to establish that there also exists a subsequence, which may be taken to be the same, such that $X_n^{-1} (t,\cdot)$ is uniformly convergent, on compact subsets of $\R^2$. It is then simple to show that the limit of this subsequence is $X^{-1}(t,\cdot)$.

\ProofStep{Step 5. Convergence of vorticities}: Define, a.e. $t \in [0,T]$, $\omega(t, x) := \omega^0(X^{-1}(t, x))$. Then we have that $\omega_n \to \omega$ in $L^\iny(0, T; L^p_{loc}(\R^2))$; to see this we adapt the proof given for bounded initial vorticity on page 316 of \cite{MB2002}, that $\omega_n(t) \to \omega(t)$ in $L^1(\R^2)$.

Let $p$ be in $[1, \iny)$. For any compact set $L$ in $\R^2$,
\begin{align*}
	&\norm{\omega_n(t, \cdot) - \omega(t, \cdot)}_{L^p(L)}
		= \norm{\omega_n^0(X_n^{-1}(t, \cdot))- \omega^0(X^{-1}(t, \cdot))}_{L^p(L)} \\
		&\qquad
		\le \norm{\omega_n^0(X_n^{-1}(t, \cdot))
		    - \omega^0(X_n^{-1}(t, \cdot))}_{L^p(L)} \\
		&\qquad\qquad\qquad
			+ \norm{\omega^0(X_n^{-1}(t, \cdot)) - \omega^0(X^{-1}(t, \cdot))}_{L^p(L)} \\
		&\qquad
		= \norm{\omega_n^0 - \omega^0}_{L^p(X_n^{-1}(L))}
			+ \norm{\omega^0(X_n^{-1}(t, \cdot)) - \omega^0(X^{-1}(t, \cdot))}_{L^p(L)}.
\end{align*}
Because $u_n$ is uniformly bounded in $n$, there exists some $R > 0$ such that $X_n^{-1}(L) \subseteq B_R$ for all $n$. Thus,
\begin{align*}
	\norm{\omega_n^0 - \omega^0}_{L^p(X_n^{-1}(L))}
		\le \norm{\omega_n^0 - \omega^0}_{L^p(\ol{B_R})}
		\to 0
\end{align*}
by property (3) of \refP{InitData}.

\Ignore{ 
Now let $(f_k)$ be a sequence in $C_c(B_{2R})$ with $\norm{\omega_0 - f_k}_{L^p(B_R)} \le k^{-1}$.
Then
\begin{align*}
	&\norm{\omega^0(X_n(t, \cdot))) - \omega^0(X(t, \cdot)))}_{L^p(L)} \\
		&\qquad
		\le \norm{\omega^0(X_n(t, \cdot))) - f_k(X_n(t, \cdot)))}_{L^p(L)} \\
		&\qquad\qquad
			+ \norm{f_k(X_n(t, \cdot))) - f_k(X(t, \cdot)))}_{L^p(L)} \\
		&\qquad\qquad
			+ \norm{f_k(X(t, \cdot))) - \omega^0(X(t, \cdot)))}_{L^p(L)} \\
		&\qquad
		= \norm{f_k - \omega^0}_{L^p(X_n^{-1}(L))}
			+ \norm{f_k - \omega^0}_{L^p(X^{-1}(L))} \\
		&\qquad\qquad
			+ \norm{f_k(X_n(t, \cdot))) - f_k(X(t, \cdot)))}_{L^p(L)} \\
		&\qquad
		\le 2 \norm{f_k(t) - \omega^0}_{L^p(B_R)}
			+ \norm{f_k(X_n(t, \cdot))) - f_k(X(t, \cdot)))}_{L^p(L)}.
\end{align*}
Given $\eps > 0$, choose $k$ large enough that $\norm{f_k(t) - \omega^0(t)}_{L^p(B_R)} \le (2/3) \eps$. Since $f_k$ is continuous, $f_k(t, X_n(t, \cdot)) \to f_k(t, X(t, \cdot))$ uniformly on $L$ and so also in $L^p(L)$. Hence, there exists an integer $N$ such that
\begin{align*}
	\norm{f_k(t, X_n(t, \cdot))) - f_k(t, X(t, \cdot)))}_{L^p(L)}
		\le \delta/3
\end{align*}
for all $n \le N$.
} 

As for the second term, $\norm{\omega^0(X_n^{-1}(t, \cdot)) - \omega^0(X^{-1}(t, \cdot))}_{L^p(L)}$, it too vanishes as $n\to \infty$ since $\omega^0$ is in $L^p$ and $X_n^{-1}(t,\cdot) \to X^{-1}(t,\cdot)$, uniformly in compact subsets of $\R^2$. Here, we are using the fact that
translations are continuous in $L^p$.

Hence, $\omega_n (t,\cdot) \to \omega(t,\cdot)$ in $L^p_{loc}(\R^2)$. This entire argument can be made uniform over $[0, T]$ by replacing $L^p$-norms with $L^\iny(0, T; L^p)$-norms.

\ProofStep{Step 6. Velocities are Cauchy in $C([0, T] \times L)$}: We have established convergence of the flow maps (and its inverse maps) to a limiting flow map (and its inverse) and convergence of the vorticities to a limiting vorticity, which is transported by the limiting flow map. As shown in Step 3, we also have equicontinuity of $(u_n)$ in space. We now show convergence of the velocities, but we will do this in two steps. First, we will use the Serfati identity once more to show that the sequence, $(u_n)$, is Cauchy in $C([0, T] \times L)$, for any compact subset, $L$, of $\R^2$.

We note, in passing, that at the corresponding point in the proof of Theorem 8.1 of \cite{MB2002} one obtains equicontinuity in time as well, employing potential theory estimates that require the vorticity to decay at infinity. In \cite{Serfati1995A}, Serfati at this same point uses estimates on the pressure (see \refR{Pressure}) to obtain sufficient regularity on $\prt_t u_n = - u_n \cdot \grad u_n - \grad p_n$. In either case, Arzela-Ascoli gives a convergent subsequence in $C([0, T] \times \R^2)$.

Let $x$ belong to $L$ and let $L_\eps = L + B_{c \eps}(0)$, where $a$ is supported in $B_c(0)$.
From \refE{SerfatiIdConv}, for any \textit{fixed} $\eps > 0$,
\begin{align}\label{e:I1I2I3Bound}
	&\abs{u_n(t, x) - u_m(t, x)}
		\le \abs{u_n^0(x) - u_m^0(x)}  + I_1 + I_2 + I_3,
\end{align}
where
\begin{align*}
	I_1
		&= \abs{(a_\eps K^j) *(\omega_n(t) - \omega_m(t))}, \;
	I_2
		= \abs{(a_\eps K^j) *(\omega_n^0 - \omega_m^0)}, \\
	I_3
		&= \int_0^t \abs{\pr{\grad \grad^\perp \brac{(1 - a_\eps) K^j}}
        \stardot (u_n \otimes u_n - u_m \otimes u_m)(s)} \, ds.
\end{align*}

Fix $q$ in $(2, \iny)$ and let $p$ in $(1, 2)$ be the \Holder exponent conjugate to $q$. Then from \refP{Rearrangement} and Young's convolution inequality,
\begin{align*}
	I_1
		&\le C \smallnorm{a_\eps(x - \cdot) K(x - \cdot)}_{L^p(L_\eps)}
			\smallnorm{\omega_n(t) - \omega_m(t)}_{L^q(L_\eps)} \\
		&\le \frac{C \eps^{2 - p}}{2 - p}
			\smallnorm{\omega_n(t) - \omega_m(t)}_{L^q(L_\eps)}
\end{align*}
and, similarly,
\begin{align*}
	I_2
		&\le \frac{C \eps^{2 - p}}{2 - p}
			\smallnorm{\omega^0_n - \omega^0_m}_{L^q(L_\eps)},
\end{align*}
while
\begin{align*}
	I_3
		&\le \int_0^t \norm{\grad \grad((1 - a_\eps(x - \cdot))
			K(x, \cdot))}_{L^1(\R^2)} \\
		&\qquad\qquad\qquad
			\norm{(u_m \otimes u_m - u_n \otimes u_n)
			(s, \cdot)}_{L^\iny(\R^2)} \, ds \\
		&\le C \eps^{-1} \int_0^t
			\norm{(u_m - u_n) (s, \cdot)}_{L^\iny(\R^2)} \, ds
		\le C t \eps^{-1}.
\end{align*}
Here, we used \refE{D2KBound} and the identity,
$
	u_m \otimes u_m - u_n \otimes u_n
		= u_m \otimes (u_m - u_n) + u_n \otimes (u_m - u_n),
$
with the uniform bound on the sequence, $(u_k)$, in $L^\iny([0, T] \times \R^2)$.

Thus,
\begin{align*}
	&\abs{u_n(t, x) - u_m(t, x)}
		\le \abs{u_n^0(x) - u_m^0(x)} + C t \eps^{-1} \\
		&\qquad
			+ \frac{C \eps^{2 - p}}{2 - p}
				\brac{\norm{\omega_n(t, \cdot) -  \omega_m(t, \cdot)}_
				    {L^q(L_\eps)}
			+ \norm{\omega_n^0 -  \omega_m^0}_{L^q(L_\eps)}}.
\end{align*}
For concreteness, we choose $p = 3/2$, so that $q = 3$.
Taking the supremum over all $(t, x)$ in $[0, T] \times L$ gives
\begin{align*}
	&\norm{u_n - u_m}_{L^\iny([0, T] \times L)}
		\le \norm{u_n^0 - u_m^0}_{L^\iny(L)} + C t \eps^{-1} \\
		&\qquad
			+ C \eps^{\frac{1}{2}}
			    \brac{\smallnorm{\omega_n -  \omega_m}_{L^\iny([0, T];
			        L^3(L_\eps))}
			+ \norm{\omega_n^0 -  \omega_m^0}_{L^3(L_\eps)}}.
\end{align*}

Now, given any $\delta > 0$, let $\eps = 1/\delta$. Then choose $N$ large enough that
\begin{align*}
	\norm{\omega_n -  \omega_m}_{L^\iny([0, T]; L^3(L_\eps))}
		+ \norm{\omega_n^0 -  \omega_m^0}_{L^3(L_\eps)}
		< \delta
\end{align*}
and $\norm{u_n^0 - u_m^0}_{L^\iny(L)} < \delta$ for all $n, m > N$.
It follows that
\begin{align*}
	\norm{u_n - u_m}_{L^\iny([0, T] \times L)}
		&< \delta + C \delta + C \delta^{1/2}.
\end{align*}

This shows that the sequence, $(u_n)$, is Cauchy in $C([0, T] \times L)$ (without the need to take a further subsequence).

\ProofStep{Step 7. Convergence to a solution}: Since $(u_n)$ is Cauchy in $C([0, T] \times L)$, for any compact subset $L$ of $\R^2$, and divergence-free for each $n$, the sequence converges to some divergence-free vector field, $u$, lying in $C([0, T] \times L)$; hence, in fact, the convergence is uniform in any compact subset of $[0, T] \times \R^2$. We also have shown  that $\omega_n \to \omega$ in $L^\iny(0, T; L^p_{loc}(\R^2))$ for all $p$ in $[1, \iny)$. In particular, we have that $\omega_n \to \omega$ in $L^\iny(0, T; L^1_{loc}(\R^2))$ and $L^\iny(0, T; L^2_{loc}(\R^2))$, while $u_n \to u$ in $L^\iny(0, T; L^2_{loc}(\R^2))$.
It follows by linearity that $\omega = \curl u$.

Let $\Cal{D} = C_c^{\infty} ((0, T) \times \R^2)$ and let $((\cdot, \cdot))$ represent the duality pairing between the distributions in $\Cal{D}'$ and the test functions in $\Cal{D}$. To show that $\prt_t \omega + u \cdot \grad \omega = 0$ in $\Cal{D}'$ we must show that
\begin{align*}
((\omega, \prt_t \varphi)) + ((u \cdot \grad \varphi, \omega)) = 0
\end{align*}
for all $\varphi$ in $\Cal{D}$. This follows immediately from the convergences pointed out above.

That the Serfati identity \refE{SerfatiIdConv} 
holds for $u$ regardless of the choice of the cutoff function, $a$, follows from these same convergences and the observation that $(u_n)$ is bounded in $L^\iny$. Indeed, let $a$ be a cutoff function as in item (2) of \refD{ESol}. The terms on the left-hand-side of \refE{SerfatiIdConv}, valid for $u_n$, clearly converge to the corresponding terms with $u$ in place of $u_n$. The first term on the right-hand-side of \refE{SerfatiIdConv}, $(aK^j)\ast(\omega_n(t) - \omega_n^0)$, also converges to the corresponding term with $\omega$ instead of $\omega_n$, since
$aK \in L^q_{loc}$, $1 \leq q < 2$ (see \refP{Rearrangement}), while $\omega_n \to \omega$ in $L^\iny(0, T; L^p_{loc}(\R^2))$ for all $p$ in $[1, \iny)$. It remains to examine the last term on the right-hand-side of \refE{SerfatiIdConv},
$\ds{- \int_0^t \pr{\grad \grad^\perp \brac{(1 - a) K^j}}\stardot (u_n \otimes u_n)(s) \, ds}$.
Fix $x \in \mathbb{R}^2$. Let $M$ be such that the support of $a$ is in the ball of radius $M$. We split the integral in the convolution into an integral on $\{|x-y|\leq 2M\}$ and an integral on $\{|x-y| > 2M\}$. In the bounded set the integrals with $u_n$ converge to the corresponding integrals for $u$ simply because $u_n$ converges to $u$ in $C([0, T] \times L)$, for any compact subset $L$ of $\mathbb{R}^2$. In the unbounded set we note that
$\grad \grad^\perp \brac{(1 - a) K^j} = (1-a)\grad\grad^{\perp}K^j$, since the derivatives of $a$ vanish in $\mathbb{R}^2 \setminus B_{2M}$. In the course of proving \refP{JKBounds} we establish that
\[|(1-a)\grad\grad^{\perp}K^j|(y) \leq \frac{C}{|y|^3}.\]
This is enough to show that the $L^1$-norm of $(1-a)\grad\grad^{\perp}K^j$ on the set $\{y \;|\; |x-y| > 2M\}$ is bounded from above by $(2M)^{-1}$. Hence, since $u_n$ is bounded in $L^\iny$, the integral on the unbounded set can be made arbitrarily small. This establishes \refE{SerfatiIdConv} for $u$.

\ProofStep{Step 8. \CapMOC of the velocity}:
The limit velocity $u(t)$ has a log-Lipschitz \MOC; this follows either from \refL{Morrey} or directly from the convergence of $(u_n)$ with a uniform bound on the log-Lipschitz \MOC on compact subsets.
\end{proof}

%
%

\PART{Part II: Existence---the case of an exterior domain}

\section{Existence in an exterior domain}\label{S:ExistenceExt}

\noindent The proof of existence in an exterior domain closely parallels that for the whole plane; in this section, we report only on the differences between the proofs. The derivation of the Serfati identity requires the majority of the effort, as it now requires us to treat boundary integrals. We give its derivation in \refS{SerfatiIdExt}. The proof of existence itself requires modifications in only two steps of the whole-plane proof of  \refS{ProofExistenceR2}: in \refS{ProofExistenceExt} we supply the details.

The sequence of smooth approximating solutions in an exterior domain that we employ in our proof of existence are those constructed by Kikuchi in \cite{Kikuchi1983}. 

Throughout this section let $\Omega$ denote the domain exterior to a bounded, smooth, connected and simply connected obstacle.

\begin{theorem}\label{T:Kikuchi}[Kikuchi, \cite{Kikuchi1983}]
	Fix $T > 0$.
	Let $u^0 \in C^\iny(\Omega)$ with $\omega(u^0)$
	compactly supported (this is more regularity than Kikuchi
	requires).
	There exists a unique classical solution, $(u, p)$, to the Euler
	equations without forcing,
	having $u^0$ as initial velocity,
   such that the vorticity is transported by the flow map,
   the circulation of $u(t)$ about $\prt \Omega$
   is conserved over time, and $u(t, x) \to 0$
	as $\abs{x} \to \iny$. Moreover, $u \in C^1([0, T] \times \Omega))$
	and $\grad p \in C([0, T] \times \Omega))$.
\end{theorem}

%
%
\subsection{The Serfati identity in an exterior domain}\label{S:SerfatiIdExt}

In this subsection we show that the alternate Serfati identity in \refE{SerfatiIdJ} holds for any radially symmetric, smooth, compactly supported cutoff function $a$, with $a = 1$ in a neighborhood of the origin.

Recall the hydrodynamic Biot-Savart kernel $\JJ$ as defined in \refE{JJDef}, and the divergence-free vector field, tangential to $\prt \Omega$, having circulation one around $\prt \Omega$ and decaying at infinity, $\KKol$ as given in \refE{olK}.

\begin{prop}\label{P:SerfatiKeyExt}
	Let $u$ be a $C^\iny$ smooth solution to the Euler equations with
	initial vorticity $\omega^0$,
	compactly supported, as given by \refT{Kikuchi}. Let the function, $a$,
	be as in (2) of \refD{ESol}.
	Then the Serfati identity, \refE{SerfatiId}, holds, and we also have
	\begin{align}\label{e:SerfatiIdJ}
		\begin{split}
			&u^j(t, x) \\
				&\quad
				= (u^0)^j(x)
				+ \int_\Omega a(x - y)
						\JJ^j(x, y) (\omega(t, y) -  \omega^0(y)) \, dy \\
				&\qquad
				- \int_0^t \int_\Omega \pr{u(s, y) \cdot \grad_y} \grad_y^\perp
						\brac{(1 - a(x - y)) \JJ^j(x, y)} \\
				&\qquad\qquad\qquad\qquad\qquad\qquad\qquad\qquad
				\cdot u(s, y) \, dy \, ds \\
				&\qquad
				- \frac{\KKol^j(x)}{2} \int_0^t \int_\Gamma
				    \abs{u(y(\sigma))}^2 \,
					\grad a(x - y(\sigma))
					\cdot \BoldTau \, d\sigma \, ds,
		\end{split}
	\end{align}
	where
	$y = y(\sigma)$ is a parameterization by arc length of $\partial \Omega$.
\end{prop}
\begin{proof}
Denote the circulation of $u$ about $\prt \Omega$ by
\begin{align*}
	\Gamma(u) = \int_\Gamma u \cdot \BoldTau,
\end{align*}
and the mass of the corresponding vorticity $\omega=\omega(u)$ by
\begin{align*}
	m(\omega) = \int_{\Omega} \omega.
\end{align*}

For smooth solutions of the Euler equations in $\Omega$, both of these quantities are conserved.
Because $\omega^0$ is compactly supported and $\omega(u)$ is transported by the flow map, $\omega(u)$ remains compactly supported for all time. This fact and the smoothness of the solution justify the calculations that follow.

Set
\begin{align*} 
	\KK[\omega]
		= \int_\Omega \KK(x, y) \omega(y) \, dy,
	\quad
	\JJ[\omega]
		= \int_\Omega \JJ(x, y) \omega(y) \, dy,
\end{align*}
and note that both integrals converge, since $\omega$ is compactly supported.

Observe that
\begin{align*} 
\JJ[\omega]	= \KK[\omega] + m(\omega) \KKol.
\end{align*}

Since $u$ conserves circulation over time, $\KKol$ has unit circulation, and $\JJ$ has zero circulation we have
\begin{align*}
	u &	= \JJ[\omega] + \Gamma(u^0) \KKol(x) \\
     & = \KK[\omega] + [m(\omega^0) + \Gamma(u^0)] \KKol (x).
\end{align*}
Hence,
\begin{align}\label{e:prtuClassical}
	\prt_t u^j(x)
		&= \prt_ t \int_\Omega \JJ^j(x, y) \omega(t, y) \, dy
		= \prt_ t \int_\Omega \KK^j(x, y) \omega(t, y) \, dy,
\end{align}
	where we have used both the conservation of $m(\omega)$ and of circulation.

Starting with \refE{prtuClassical} and using the vorticity equation \refE{EulerEqsVortForm}, we have,
\begin{align} \label{e:SerfatiIdStartExtH}
    \begin{split}
        \prt_t u^j(t, x)
            &= \prt_t \int_\Omega a(x - y) \KK^j(x, y)
                \omega(t, y) \, dy \\
            &\qquad
                - \int_\Omega (1 - a(x - y)) \KK^j(x, y)
                    (u \cdot \grad \omega)(t, y) \, dy,
    \end{split}
\end{align}
	$j = 1, 2$. We rewrite the last term as before as
\begin{align} \label{e:SerfatiIDIBP1}
	\begin{split}
    - \int_\Omega &(1 - a(x - y)) \KK^j(x, y) (u \cdot \grad \omega)(t, y) \, dy \\
	 	&= - \int_\Omega (1 - a(x - y)) \KK^j(x, y) \curl (u \cdot \grad u)(t, y) \, dy \\
	 	&= \int_\Omega (1 - a(x - y)) \KK^j(x, y)
	 	    \dv \brac{(u \cdot \grad u)^\perp(t, y)} \, dy \\
	 	&= - \int_\Omega \brac{(u \cdot \grad u)^\perp(t, y)}
			\cdot \grad \brac{(1 - a(x - y)) \KK^j(x, y)} \, dy \\
	 	&= \int_\Omega (u \cdot \grad u)(t, y)
			\cdot \grad^\perp \brac{(1 - a(x - y)) \KK^j(x, y)} \, dy.
	\end{split}
\end{align}
The boundary integral above vanishes because $\KK(x, \cdot) = 0$ on the boundary.

Let $V$ be a vector field on $\Omega$ and recall the following identity:
\begin{align*}
	(u \cdot \grad) (V \cdot u) = [(u \cdot \grad) V] \cdot u + [(u \cdot \grad) u] \cdot V.
\end{align*}

Integrating on $\Omega$, we obtain
\begin{align}\label{e:SerfatiIDIBP2}
	\begin{split}
	\int_\Omega [(u \cdot \grad) u] \cdot V
		&= \int_\Omega (u \cdot \grad) (V \cdot u) - \int_\Omega [(u \cdot \grad) V] \cdot u \\
		&=  - \int_\Omega (u \cdot \grad V) \cdot u,
	\end{split}
\end{align}
the first integral vanishing in integrating by parts since $\dv u = 0$ and $u \cdot \n = 0$.

Using \refE{SerfatiIDIBP2} with
$V = \grad^\perp \brac{(1 - a(x - y)) \KK^j(x, y)}$, putting the resulting term back into \refE{SerfatiIdStartExtH}, and integrating in time yields \refE{SerfatiId}.

To obtain \refE{SerfatiIdJ}, we return to \refE{prtuClassical}, writing,
\begin{align*}
	\prt_t &u^j(x)
		=  \prt_t \int_\Omega a(x - y) \JJ^j(x, y) \omega(y) \, dy
			+ \int_\Omega (1 - a(x - y)) \JJ^j(x, y) \prt_t \omega(y) \, dy,
\end{align*}
$j = 1, 2$. Integrating the last term by parts as we did in \refE{SerfatiIDIBP1}, we now have the additional, boundary integral (using $\JJ(x, y) = \KKol(x)$ when $y$ is on $\prt \Omega$):
\begin{align}\label{e:AddBoundInt}
    \begin{split}
	 \int_\Omega &(1 - a(x - y)) \JJ^j(x, y) \prt_t \omega(y) \, dy \\
	 	&= \int_\Omega (u \cdot \grad u)(y)
			\cdot \grad^\perp \brac{(1 - a(x - y)) \JJ^j(x, y)} \, dy \\
		&\qquad
			+ (\KKol^j(x) \int_\Gamma [u(y(\sigma)) \cdot \grad u(y(\sigma))]^\perp \cdot \n \,
				(1 - a(x - y(\sigma) )) \, d\sigma.
	\end{split}
\end{align}

The first term on the right-hand side we integrate by parts once more, as we did in proving \refE{SerfatiId},
the vanishing of $u \cdot \n$ on the boundary again being used to eliminate the boundary term.
For the second term, which contains the boundary integral, we use the identity,
\begin{align*}
	[(u &\cdot \grad) u] \cdot \BoldTau
		= \brac{(u \cdot \n \, \partial_{\n} + u\cdot\BoldTau \, \partial_{\BoldTau}) u}
		        \cdot \BoldTau \\
		&= (u\cdot\BoldTau) \partial_{\BoldTau} (u \cdot \BoldTau)
		= u \partial_{\BoldTau} u
		= \frac{1}{2} \diff{}{\sigma} \abs{u(y(\sigma)}^2.
\end{align*}
To make sense of $\prt_{\n}$, we extended $\n$ into a tubular neighborhood of the boundary. Since $u \cdot \n = 0$, the term containing $\prt_{\n}$ then vanished.

Integrating the boundary integral in \refE{AddBoundInt} by parts gives
\begin{align*}
	\int_\Gamma &[u(y(\sigma)) \cdot \grad u(y(\sigma))]^\perp \cdot \n \,
			(1 - a(x - y(\sigma) )) \, d\sigma \\
		&= \frac{1}{2} \int_\Gamma \diff{}{\sigma} \abs{u(y(\sigma))}^2
			\, (1 - a(x - y(\sigma) )) \, d\sigma \\
		&= \frac{1}{2} \int_\Gamma \abs{u(y(\sigma))}^2 \, \diff{}{\sigma} a(x - y(\sigma) )) \, d\sigma \\
		&= - \frac{1}{2} \int_\Gamma \abs{u(y(\sigma))}^2 \, \grad a(x - y(\sigma))
				\cdot \diff{y(\sigma)}{\sigma} \, d\sigma.
\end{align*}
This yields \refE{SerfatiIdJ}, since $\diff{y(\sigma)}{\sigma} = \BoldTau$.
\end{proof}

To control the boundary term in \refE{SerfatiIdJ}, we need control not just on the size of the integrands, but cancellation due to the velocity field itself. This is easily obtained from the simple bound in \refP{JBoundaryTermBound}.

\Ignore{ 
\begin{proof}
	Applying the chain rule gives.
	\begin{align*}
		 \frac{1}{2} &\diff{}{\sigma} \abs{u(y(\sigma)}^2
		 	= \frac{1}{2} \diff{}{\sigma} \sum_i (u^i(y(\sigma)))^2
			= u^i(y(\sigma)) \diff{}{\sigma} u^i(y(\sigma)) \\
			&= u^i(y(\sigma)) \prt_j u^i(y(\sigma)) \pdx{y_j(\sigma)}{\sigma}
			= u^i(y(\sigma)) \prt_j u^i(y(\sigma)) \BoldTau(\sigma) \\
			&= u(y(\sigma)) \cdot (\grad u(y(\sigma)) \cdot \BoldTau(\sigma))
			= (u(y(\sigma)) \cdot \grad u(y(\sigma))) \cdot \BoldTau(\sigma).
	\end{align*}
	The last equality holds simply because $u$ is parallel to $\BoldTau$.
\end{proof}
} 

\begin{prop}\label{P:JBoundaryTermBound}
	Let $\eps > 0$ and set $a_{\eps}$ as in \refP{JKBounds}. Let $u$ be a
	continuous vector field on $\Omega$ which is tangent to the boundary. Then there exists $C>0$
	such that
	\begin{align*}
		\abs{\frac{\KKol^j(x)}{2} \int_\Gamma \abs{u(y(\sigma))}^2 \, \grad a_\eps(x - y(\sigma))
				\cdot \BoldTau \, d\sigma}
			\le \frac{C}{\eps} \norm{u}_{L^\iny}^2.
	\end{align*}
\end{prop}
\begin{proof}
	This follows from the bound, $$\abs{\grad a_\eps(x - y(\sigma))}
	= \abs{\eps^{-1} \grad a ((x - y(\sigma)) \eps^{-1})}
	\le C \eps^{-1},$$ of \refP{JKBounds} and \refE{olKBound}.
\end{proof}

%
%
\subsection{Proof of existence in an exterior domain}\label{S:ProofExistenceExt}

\begin{proof}[Proof of existence in \refT{Existence} for an exterior domain]
As mentioned previously, $\Omega$ denotes the domain exterior to a bounded, smooth, connected and simply connected obstacle.

As in our proof of existence for the full plane in \refS{ProofExistenceR2}, we approximate the initial data employing \refP{InitData} and construct smooth solutions to the Euler equations using \refT{Kikuchi}. The key bounds in \refE{omegaunLInfBoundR2}, then, continue to hold on $\Omega$:
\begin{align}\label{e:omegaunLInfBoundExt}
    \begin{split}
        \norm{\omega_n}_{L^\iny(\R \times \Omega)}
            \le \smallnorm{\omega_n^0}_{L^\iny}
            \le C \smallnorm{\omega^0}_{L^\iny},
        \quad
        \smallnorm{u_n^0}_{L^\iny}
            \le C \smallnorm{u^0}_{L^\iny}.
    \end{split}
\end{align}

The proof proceeds in the identical manner to that of \refS{ProofExistenceR2} with the exceptions of two steps in the proof, described below. It is important to observe, though, that the convergences obtained are for compact subsets of $\ol{\Omega}$ and $[0, T] \times \ol{\Omega}$.

As before, we denote the approximate solutions by $u_n$ and $\omega_n$.

\ProofStep{Bound velocities in $L^\iny([0, T] \times \Omega)$}:
Let $a$ be any cutoff function as in (2) of \refD{ESol}. Let $a_{\eps}$ be as in \refP{JKBounds}. 

From \refE{SerfatiIdJ}, substituting $u_n$ and $\omega_n$ for $u$ and $\omega$, we have
\begin{align*}
	\abs{u_n(t, x)}
		\le &\abs{u_n^0(x)}
			+ \abs{\int_\Omega a_\eps(x - y)  \JJ(x, y) (\omega_n(t, y)
					-  \omega_n^0(y)) \, dy} \\
		&
			+ \int_0^t \abs{\int_\Omega \abs{\grad_y \grad_y
			    \pr{(1 - a_\eps(\cdot - y)) \JJ(\cdot, y)}}
				\abs{u_n(s, y)}^2 \, dy} \, ds \\
		&
			+ \abs{\frac{\KKol^j(x)}{2} \int_0^t \int_\Gamma \abs{u_n(s, y(\sigma))}^2
			    \, \grad a_\eps(x - y(\sigma))
				\cdot \BoldTau \, d\sigma \, ds}.
\end{align*}

Applying \refPAnd{JKBounds}{JBoundaryTermBound} to \refE{SerfatiIdJ}, and using \refE{omegaunLInfBoundExt}, it follows from \Holders inequality that, for some constant $C > 0$, independent of $n$,
\begin{align}\label{e:BasicunBoundExt}
	\norm{u_n(t)}_{L^\iny}
		&\le C
			+ C \eps
			+ \frac{C}{\eps} \int_0^t \norm{u_n(s)}_{L_\iny}^2 \, ds
\end{align}
for all $\eps > C_0$, with $C_0$ as in \refP{JKBounds}.

Observe that we can choose $\eps > C_0$ arbitrarily, even allowing it to depend on time. Hence, we can
let
\begin{align*}
	\eps
		= \eps(t)
		= \max \set{C_0 + 1, \pr{\int_0^t \norm{u_n(s)}_{L_\iny}^2 \, ds}^{1/2}}.
\end{align*}
The function $\eps(t)$ is continuous and non-decreasing, with $\eps(0) = C_0 + 1$. Suppose that there exists a finite time, $T_n^*$, at which $\displaystyle{\int_0^{T_n^*}\norm{u_n(s)}_{L_\iny}^2 \, ds}= (C_0+1)^2$. Then it follows directly from \refE{BasicunBoundExt} that $u_n$ lies in $L^\iny([0, T_n^*]; L^\iny)$ with a norm bounded by $C(C_0 + 1)$. After that time,
$\eps(t) = \displaystyle{\int_0^t \norm{u_n(s)}_{L_\iny}^2 \, ds} > C_0 + 1$, and we obtain
\begin{align*} 
	\norm{u_n(t)}_{L^\iny}
		&\le C
			+ C \pr{\int_0^t \norm{u_n(s)}_{L_\iny}^2 \, ds}^{1/2},
\end{align*}
so that
\begin{align*} 
	\norm{u_n(t)}_{L^\iny}^2
		&\le C
			+ C \int_0^t \norm{u_n(s)}_{L_\iny}^2 \, ds.
\end{align*}
We conclude from Gronwall's lemma that
\begin{align*}
	\norm{u_n(t)}_{L^\iny}
		\le \max \set{C e^{Ct}, C_0}
		    = C e^{Ct}.
\end{align*}
Thus, $u_n$ lies in $L^\iny([0, T] \times \Omega)$ for any $T > 0$ with a bound that is uniform in $n$.

\ProofStep{Velocities are Cauchy in $C([0, T] \times L)$}: Let $L$ be a compact subset of $\ol{\Omega}$. The only change to the proof of this step in \refS{ProofExistenceR2} is that $\JJ$ is used in place of $K$ in the expressions for $I_1$, $I_2$, and $I_3$ in \refE{I1I2I3Bound}, which also includes the additional term,
\begin{align*}
	I_4
		&=
		\abs{\frac{\KKol^j(x)}{2} \int_0^t \int_\Gamma
		(\abs{u_n(y(\sigma))}^2 - \abs{u_m(y(\sigma))}^2)
				\, \grad a_\eps(x - y(\sigma))
				\cdot \BoldTau \, d\sigma \, ds}.
\end{align*}
\refP{JBoundaryTermBound} and the uniform bound on the sequence, $(u_k)$, in $L^\iny([0, T] \times \Omega)$ gives
\begin{align*}
	I_4
		\le \frac{Ct}{\eps}.
\end{align*}

The estimates on $I_1$, $I_2$, and $I_3$ are unchanged, though now they only hold for $\eps > C_0$. But this is of no matter, since we take $\eps$ to infinity.
\end{proof}

%
%
\PART{Part III: Uniqueness and continuous dependence on initial data}

\section{Uniqueness and continuous dependence on initial data}\label{S:UniquenessAndCont}

\noindent Our proof of uniqueness, which assumes that the Serfati identity holds, derives from that of Serfati in \cite{Serfati1995A} (who also assumes, implicitly, that the Serfati identity holds). We present the proof in \refS{Uniqueness}. The continuous dependence on initial data of \refT{ContDep} is a modification of our uniqueness proof, and is presented in \refS{ContinuousDependence}.

In this section, $\Omega$ can be either all of $\R^2$ or an exterior domain. In the proofs, we exploit a number of estimates that we derive later in \refS{BSEstimates}. The estimates are stated in terms of $K$ (see \refE{KBSFullPlane}) for the full plane and in terms of $\KK$ (see \refE{KKDef}) for an exterior domain. When $\Omega = \R^2$, we have $\KK(x, y) = K(x - y)$.

%
%
\subsection{Uniqueness}\label{S:Uniqueness}

\noindent Our proof of uniqueness in \refT{Existence} differs from that of Serfati's  proof in \cite{Serfati1995A} in two key respects. First, we bound, in effect, the quantity $h(t)$ defined in \refE{h}, whereas Serfati bounds the quantity $\int_0^t \abs{h'(s)} \, ds$, which is more difficult to deal with rigorously. Second, we also bound the terms involving the Biot-Savart law differently, via \refP{olgKBound}, to obtain a unified argument that applies both to the full plane and to an exterior domain.

The Serfati identity is used both in the proof of existence and of uniqueness. Following the proof of uniqueness, we compare and contrast the application of the Serfati identity in the two proofs.

\begin{proof}[\textbf{Proof of uniqueness in \refT{Existence}}]
Let $u_1$, $u_2$ be the velocities for two solutions sharing the same initial velocity. Each $u_j$ is log-Lipschitz and it can always be arranged so that $\mu \colon [0, \iny) \to [0, \iny)$ serves as a common, strictly increasing, bounded \MOC  for both $u_1$ and $u_2$ with
\begin{align}\label{e:mu}
	\begin{array}{rl}
	\mu(r) = - C r \log r
		&\text{ for } 0 < r \le e^{-1}, \\
	- r \log r \le \mu(r)
		&\text{ for } r \ge e^{-1}.
	\end{array}
\end{align}

We will assume that the cutoff function, $a$, of \refE{SerfatiId} is equal to $0$ outside of $B_{e^{-1}}$. The choice of $e^{-1}$ is convenient because of the estimates in \refP{hlogBound}. We will also assume that the cutoff function is such that $C_0$ of \refP{JKBounds} is less than 1; thus, the estimates in \refEThrough{D2JBound}{D2KKBound} hold for $\eps = 1$.

Let $X_j$ be the flow map for $u_j$, $j = 1, 2$.
We will establish uniqueness by showing that $X_1 = X_2$. Let $t$ lie in $[0, T]$. Our approach is to bound the quantity,
\begin{align}\label{e:M}
	M(t)
		&= \int_0^t P(s) \, ds,
\end{align}
where
\begin{align*}
	P(s)
		&= \norm{u_2(s, X_2(s, \cdot)) - u_1(s, X_1(s, \cdot))}_{L^\iny}.
\end{align*}
We do this by obtaining, through a long series of estimates, the inequality
\begin{align}\label{e:MBound}
	M(t)
		\le \int_0^t \nu(M(s)) \, ds,
\end{align}
where $\nu$ is the Osgood \MOC (see \refL{Osgood}) given explicitly in \refE{nu}. Applying \refL{Osgood} to \refE{MBound} gives $M \equiv 0$. Then letting
\begin{align}\label{e:h}
	h(t) = \norm{X_1(t, \cdot) - X_2(t, \cdot)}_{L^\iny},
\end{align}
it follows that
\begin{align}\label{e:Motivatehm}
	\begin{split}
	h(t)
		&= \norm{X_1(t, \cdot) - X_2(t, \cdot)}_{L^\iny} \\
		&= \norm{\int_0^t u_1(s, X_1(s, \cdot))
			- u_2(s, X_2(s, \cdot)) \, ds}_{L^\iny} \\
		&\le \int_0^t \norm{u_1(s, X_1(s, \cdot))
			- u_2(s, X_2(s, \cdot))}_{L^\iny} \, ds \\
		&= M(t).
	\end{split}
\end{align}
Hence, $X_1 \equiv X_2$ so that $u_1 \equiv u_2$, and uniqueness holds.

(It is easy to see that $h(t)$ and $M(t)$ are continuous and bounded, because of the boundedness of $u_1$ and $u_2$. Hence the the inequality in \refEAnd{MBound}{Motivatehm} and the inequalities that follow all contain finite quantities.)


We now commence to prove \refE{MBound}. We start by obtaining a bound on the quantity, $\abs{u_1(t, X_1(t, x)) - u_2(t, X_2(t, x))}$, to obtain a bound on $P(t)$, which we will transform to the bound on $M(t)$ in \refE{MBound}.

By the triangle inequality,
\begin{align}\label{e:BasicUniqTriang}
    \begin{split}
	\pabs{u_1(t, X_1(t, &x)) - u_2(t, X_2(t, x))} \\
		&\le \abs{u_2(t, X_1(t, x)) - u_2(t, X_2(t, x))} \\
		&\qquad
			+ \abs{u_1(t, X_1(t, x)) - u_2(t, X_1(t, x))} \\
		&=: A_1 + A_2.
	\end{split}
\end{align}

We easily bound $A_1$ by
\begin{align}\label{e:A1Bound}
	A_1
		\le \mu(\abs{X_1(t, x) - X_2(t, x)})
		\le \mu(h(t)).
\end{align}
We obtain a bound for $A_2$ by subtracting \refE{SerfatiId} for $u_2$ from \refE{SerfatiId} for $u_1$:
\begin{align}\label{e:BasicEDiff}
	\begin{split}
		A_2 &\le
			\abs{\int_\Omega a(X_1(t, x) - y) \KK(X_1(t, x), y)
				(\omega^1(t, y) -  \omega^2(t, y)) \, dy} \\
		&\qquad+ \int_0^t \int_\Omega
				\abs{\grad_y \grad_y \pr{(1 - a(X_1(t, x) - y))
				\KK(X_1(t, x), y)}} \\
		&\qquad\qquad\qquad\qquad\qquad
				\abs{u_1 \otimes u_1 - u_2 \otimes u_2}(s, y)
					\, dy \, ds \\
		&=: B_1 + B_2.
	\end{split}
\end{align}

Because $\ol{\gamma}(y) := a(X_1(t, x) - y)$ is Lipschitz-continuous and has finite-measure support with Lipschitz constant and measure independent of $t$ and $x$, we can apply \refP{olgKBound} to conclude that
\begin{align}\label{e:B1Bound}
	&B_1
		\le C \smallnorm{\omega^0}_{L^\iny} \mu(h(t))
\end{align}
for some constant $C$ depending only upon the cutoff function $a$.

For $B_2$, we have simply,
\begin{align}\label{e:B2Bound}
	\begin{split}
	B_2
		&\le \int_0^t \norm{\grad \grad((1 - a(X_1(t, x) - \cdot))
		    \KK(X_1(t, x), \cdot))}_{L^1} \\
		&\qquad\qquad\qquad
			\norm{(u_2 \otimes u_2 - u_1 \otimes u_1)
			(s, \cdot)}_{L^\iny} \, ds.
	\end{split}
\end{align}
The $L^1$-norm in the integrand above is finite and bounded uniformly in $x$ by \refP{JKBounds}. Using,
\begin{align*}
	u_2 \otimes u_2 - u_1 \otimes u_1
		= u_2 \otimes (u_2 - u_1) + u_1 \otimes (u_2 - u_1),
\end{align*}
because $u_j$ lies in $L^\iny([0, T] \times \Omega)$, we have
\begin{align*}
	B_2
		&\le C \int_0^t \norm{u_2(s) - u_1(s)}_{L^\iny} \, ds \\
		&= C \int_0^t \norm{u_2(s, X_1(s, \cdot))
		    - u_1(s, X_1(s, \cdot))}_{L^\iny} \, ds \\
		&\le C \int_0^t \norm{u_2(s, X_1(s, \cdot))
		    - u_2(s, X_2(s, \cdot))}_{L^\iny} \, ds \\
		&\qquad
			+ C \int_0^t \norm{u_2(s, X_2(s, \cdot))
			    - u_1(s, X_1(s, \cdot))}_{L^\iny} \, ds \\
		&\le C \int_0^t \mu(h(s)) \, ds
			+ C \int_0^t \norm{u_2(s, X_2(s, \cdot))
			    - u_1(s, X_1(s, \cdot))}_{L^\iny} \, ds.
\end{align*}
Here, we used $\abs{u_2(s, X_2(s, \cdot)) - u_2(s, X_1(s, \cdot))} \le \mu(\abs{X_2(s, \cdot) -X_1(s, \cdot)}) \le \mu(h(s))$.

What we have shown is that
\begin{align*}
	\pabs{u_1(t, &X_1(t, x)) - u_2(t, X_2(t, x))} \\
		&\le  C \int_0^t \mu(h(s)) \, ds + C \mu(h(t))
			+ C \int_0^t P(s) \, ds.
\end{align*}
Taking the supremum over all $x$ in $\R^2$ and using \refE{M}, we conclude that
\begin{align*}
	P(t)
		\le C \int_0^t \mu(h(s)) \, ds + C \mu(h(t))
			+ C M(t).
\end{align*}

But $h(t) \le M(t)$ by \refE{Motivatehm}, and $\mu$ is nondecreasing so $\mu(h(t)) \le \mu(M(t))$ and $\mu(h(s)) \le \mu(M(s))$. Thus,
\begin{align*}
	M'(t)
		= P(t)
		\le C \int_0^t \mu(M(s)) \, ds + C \mu(M(t))
			+ C M(t).
\end{align*}
Since $M$ is increasing, we can write
\begin{align*}
	M'(t) \le
		C(1 + t)  \mu(M(t))
			+ C M(t).
\end{align*}
For our purposes, it is easier to weaken this inequality slightly to
\begin{align}\label{e:PBound}
	M'(s) \le
		C(1 + t)  \mu(M(s)) + C M(s)
		= \nu(M(s))
\end{align}
for all $s$ in $(0, t)$,
where
\begin{align}\label{e:nu}
	\nu(r) = C \brac{(1 + t) \mu(r) + r}.
\end{align}
Near $r = 0$, $\mu(r) >> r$, so that $\nu$ is an Osgood \MOC.

In integral form, using $M(0) = 0$, \refE{PBound} becomes
\begin{align*}
	M(t) \le \int_0^t \nu(M(s)) \, ds.
\end{align*}
That $M \equiv 0$ follows from \refL{Osgood}, and since $h(t) \le M(t)$, $h \equiv 0$ as well, which proves uniqueness.
\end{proof}

As in the proof of existence (see \refS{ProofExistenceR2}), we used the Serfati identity above, though we used it quite differently in the two proofs.

We bounded the first term in the Serfati identity ($B_1$ in the proof above) by appealing to \refP{olgKBound}. As we will see in the proof of \refP{olgKBound} below, the bound is obtained by making changes of variables using both flow maps to return the vorticity to the initial time, taking advantage of the transport of vorticity by the flow map. The penalty is that the kernel $\KK$ is distorted, as is the cutoff function. These distortions, however, can be accommodated using the estimates developed on $\KK$ in \refS{BSEstimates}, and using the smoothness of the cutoff function.

In the proof of existence, we estimated the term corresponding to $B_1$ directly with no change of variables, though using the estimates developed on $\KK$ in \refS{BSEstimates}.

In both proofs, the estimate on the second term in the Serfati identity ($B_2$ in the proof above) was done directly. However, at the point in the proof of existence where we first used the Serfati identity, we had not yet established the boundedness of the velocity. Hence, our estimate was necessarily quadratic in the $L^\iny$-norm of the velocity. But using Serfati's insight, we turned this estimate into a linear one and so obtained an $L^\iny$-norm on the velocity.

\begin{prop}\label{P:olgKBound}
	Assume that $u_1, u_2$ are Serfati solutions to the Euler equations
	with vorticities, $\omega_1, \omega_2$ and initial
	vorticities, $\omega^0_1$ and $\omega^0_2$, lying in $S^p$ for
	$p \in (2, \iny]$, where $S^p$ is defined in \refE{Sp}.
	Let $\ol{\gamma}$ be any Lipschitz function on $\Omega$ having
	finite-measure support.
	Let $h = h(t)$ be as in \refE{h} and set $\mu$ to be as in \refE{mu}.
	For all $x$ in $\Omega$,
	\begin{align*}
		\Largepabs{\int_\Omega \olg(y) &\KK^j(x, y)
			(\omega_1(t, y) - \omega_2(t, y)) \, dy} \\
			&\le C \max \set{\smallnorm{\omega^0_1}_{L^\iny},
				\smallnorm{\omega^0_2}_{L^\iny}} \mu(h)
				+ C_p \norm{\omega^0_1 - \omega^0_2}_{L^p_{uloc}}.
	\end{align*}
	The constant, $C$, depends only on the Lipschitz constant and measure
	of the support of $\olg$, and $C_p$ depends only on $p$ and the
	measure of the support of $\olg$.
\end{prop}
\begin{proof}
	Assume first that $h < e^{-1}$.

	Since for Serfati solutions, $\omega_j$ is transported by the
	flow, $X_j$, associated to $u_j$,
	$j = 1, 2$, we have
	\begin{align*}
		\int_\Omega &\olg(y) \KK^j(x, y)
			(\omega_1(t, y) - \omega_2(t, y)) \, dy \\
			&= \int_\Omega \olg(y) \KK^j(x, y)
				\pr{\omega^0_1(X_1^{-1}(t, y))
					- \omega^0_2(X_2^{-1}(t, y))} \, dy.
	\end{align*}
	Alternately making the changes of variable $y = X_1(t, z)$
	and $y = X_2(t, z)$, this becomes, since $X_1$ and $X_2$ are measure-preserving,
	\begin{align*}
		\int_\Omega \olg(X_1(t, z))
				\KK^j(x, X_1(t, z)) \omega^0_1(z) \, dz
		- \int_\Omega \olg(X_2(t, z))
				\KK^j(x, X_2(t, z)) \omega^0_2(z) \, dz.
	\end{align*}	
	We can write this as
	\begin{align*}
		\int_\Omega \olg(y) \KK^j(x, y)
			(\omega_1(t, y) - \omega_2(t, y)) \, dy
			= I_1 + I_2 + I_3,
	\end{align*}
	where
	\begin{align*}
		I_1
			&= \int_\Omega \brac{\olg(X_1(t, z)) - \olg(X_2(t, z))}
				\KK^j(x, X_2(t, z))
				\omega^0_1(z) \, dz, \\
		I_2
			&= \int_\Omega \olg(X_1(t, z)) \brac{\KK^j(x, X_1(t, z))
				- \KK^j(x, X_2(t, z))}
				\omega^0_1(z) \, dz, \\
		I_3
			&= \int_\Omega \olg(X_2(t, z))
				\KK (x, X_2(t, z))
				(\omega_1^0(z) - \omega_2^0(z)) \, dz.
	\end{align*}
	
	Letting
	\begin{align*}
		U
			= \set{z \colon \olg(X_1(t, z)) \ne \olg(X_2(t, z))},
	\end{align*}
	we have
	\begin{align*}
		\abs{I_1}
			&\le \norm{\omega^0_1}_{L^\iny}
				\sup_{z \in \Omega} \abs{\olg(X_1(t, z)) - \olg(X_2(t, z))}
				\int_U \abs{\KK^j(x, X_2(t, z))} \, dz \\
			&\le C \norm{\omega^0_1}_{L^\iny} h
				\int_U \abs{\KK^j(x, X_2(t, z))} \, dz,
	\end{align*}
	where $C$ is the Lipschitz constant for $\olg$.
	But,
	\begin{align*}
		U
			\subseteq X_1(t, \supp \olg) \cup X_2(t, \supp \olg)
	\end{align*}
	has measure bounded in time, since $X_1$ and $X_2$ are
	measure-preserving, and
	\begin{align*}
		\int_U \abs{\KK^j(x, X_2(t, z))} \, dz
			&= \int_{X_2(t, U)} \abs{\KK^j(x, y)} \, dy
			\le C
	\end{align*}
	by \refP{Rearrangement} and
	using $\abs{X_2(t, U)} = \abs{U}$. Hence,
	\begin{align*}
		\abs{I_1} \le C \smallnorm{\omega^0_1}_{L^\iny} h.
	\end{align*}

	Applying \refP{hlogBound}, we can easily bound $I_2$ by
	\begin{align*}
		\abs{I_2}
			&\le \norm{\ol{\gamma}}_{L^\iny} \smallnorm{\omega^0_1}_{L^\iny}
				\smallnorm{\KK^j(x, X_1(t, z)) - \KK^j(x, X_2(t, z))}_
				    {L^1(X_1^{-1}(t, \supp \olg))} \\
			&\le - C \smallnorm{\omega^0_1}_{L^\iny} h \log h,
	\end{align*}
	noting that we used $h < e^{-1}$.
	
	For $I_3$, we have
	\begin{align*}
		\pabs{I_3}
			&\le \norm{\olg(X_2(t, z)) \KK (x, X_2(t, z))}_{L^{p'}_z}
				\norm{\omega_1^0 - \omega_2^0}_
				{L^p(\supp \olg \circ X_2(t, \cdot))} \\
			&= \norm{\olg(w) \KK (x, w)}_{L^{p'}_w}
				\norm{\omega_1^0 - \omega_2^0}_
				{L^p(\supp \olg \circ X_2(t, \cdot))} \\
			&\le \norm{\olg(w) \KK (x, w)}_{L^{p'}_w}
				\norm{\omega_1^0 - \omega_2^0}_
				{L^p_{uloc}(\Omega)} \\
			&\le C_p \norm{\omega_1^0 - \omega_2^0}_
				{L^p_{uloc}(\Omega)},
	\end{align*}
	where $1/p' + 1/p = 1$. In the final inequality, we used
	\refP{Rearrangement}.
		
	Combining the bounds for $I_1$, $I_2$, and $I_3$ gives
	\begin{align*}
		\Largepabs{\int_\Omega \olg(y) &\KK^j(x, y) (\omega_1(t, y) - \omega_2(t, y))
				\, dy} \\
			&\le - C \smallnorm{\omega^0_1}_{L^\iny} h \log h + C_p \norm{\omega_1^0 - \omega_2^0}_{L^p_{uloc}(\Omega)} \\
			&= C \smallnorm{\omega^0_1}_{L^\iny} \mu(h) + C_p \norm{\omega_1^0 - \omega_2^0}_{L^p_{uloc}(\Omega)}.
	\end{align*}
	
	For $h \ge e^{-1}$, we apply, as above, \refP{Rearrangement} to
	conclude that
	\begin{align*}
		\Largepabs{\int_\Omega \olg(y) &\KK^j(x, y) (\omega_1(t, y) - \omega_2(t, y))
				\, dy}
			\le C \max \set{\smallnorm{\omega^0_1}_{L^\iny},
				\smallnorm{\omega^0_2}_{L^\iny}} \\
			&= C \mu(e^{-1}) \max \set{\smallnorm{\omega^0_1}_{L^\iny},
				\smallnorm{\omega^0_2}_{L^\iny}}
			\le C \mu(h) \max \set{\smallnorm{\omega^0_1}_{L^\iny},
				\smallnorm{\omega^0_2}_{L^\iny}},
	\end{align*}
	and the proof is complete.
\end{proof}

%
%
\subsection{Continuous dependence on initial data}\label{S:ContinuousDependence}

\noindent
In this subsection, we modify slightly the proof of uniqueness in the previous section to obtain the limited continuity on initial data stated in \refT{ContDep}.

\begin{proof}[{Proof of \refT{ContDep}}] We follow the same steps as in the proof of uniqueness in \refS{Uniqueness}, and use the same definitions made in that proof. Now, however, $u_1$ and $u_2$ are the unique solutions for \textit{different} initial data. This leads to the bound,
\begin{align*}
	\pabs{u_1(t, &X_1(t, x)) - u_2(t, X_2(t, x))} \\
		&\le \abs{u_1^0(x) - u_2^0(x)} + A_1 + A_2
		\leq \abs{u_1^0(x) - u_2^0(x)} + A_1 + B_1 + B_2,
\end{align*}
where $A_1$, $A_2$, $B_1$, and $B_2$ are the same as in \refS{Uniqueness}.

We bound $A_1$ and $B_2$ exactly as in \refEAnd{A1Bound}{B2Bound}, for the initial data was not used in their derivations. As in the proof of uniqueness, we bound the term $B_1$ using \refP{olgKBound}, but now an additional term,
\begin{align*}
	C_p \smallnorm{\omega_1^0 - \omega_2^0}_{L^p_{loc}(\Omega)},
\end{align*}
appears because the initial vorticities differ.

The net effect is that the bound in \refE{PBound} becomes
\begin{align}\label{e:NetEffectM}
	\begin{split}
	M'(s) &\le
		\norm{u_1^0 - u_2^0}_{L^\iny}
			+ C_p \smallnorm{\omega_1^0
					- \omega_2^0}_{L^p_{loc}(\Omega)} \\
		&\qquad\qquad
			+ C(1 + t)  \mu(M(s)) + C M(s) \\
		&\le C s_0 + C(1 + t)  \mu(M(s)) + C M(s) \\
		&= C s_0 + \nu(M(s)),
	\end{split}
\end{align}
where $\nu$ is as in \refE{nu}. In integral form, this is
\begin{align*}
	M(t) \le C s_0 t + \int_0^t \nu(M(s)) \, ds,
\end{align*}
since still $M(0) = 0$.

\refC{OsgoodCor} tells us that $M(t) \le \Gamma(t)$, where $\Gamma(t)$ is defined by
\begin{align*}
	\int_{C s_0 t}^{\Gamma(t)} \frac{ds}{\nu(s)} = t.
\end{align*}
It follows from \refE{NetEffectM} that
\begin{align*}
	P(t)
		= M'(t)
		\le C s_0 + C(1 + t) \mu(\Gamma(t))
			+ C \int_0^t P(s) \, ds,
\end{align*}
so by Gronwall's lemma we conclude that
\begin{align*}
	P(t) \le C \brac{s_0 + (1 + t) \mu(\Gamma(t))} e^{Ct}.
\end{align*}

Since $P(t) = \norm{\prt_t (X_2 - X_1)}_{L^\iny}$, we have
\begin{align*}
	\abs{X_2 - X_1}(t, x)
		&= \abs{\int_0^t \prt_s (X_2 - X_1)(s, x)}
		\le \int_0^t \abs{\prt_s (X_2 - X_1)(s, x)} \\
		&\le \int_0^t P(s) \, ds
		= M(t)
		\le \Gamma(t).
\end{align*}
So, one obtains continuous dependence of the flow maps with respect to initial data.

We can turn this into continuous dependence of velocity, as
\begin{align}\label{e:u1u2Bound}
	\begin{split}
	&\norm{u_1(t) - u_2(t)}_{L^\iny}
		= \norm{u_1(t, X_1(t, \cdot))
			- u_2(t, X_1(t, \cdot))}_{L^\iny} \\
		&\qquad
		\le \norm{u_1(t, X_1(t, \cdot))
			- u_2(t, X_2(t, \cdot))}_{L^\iny} \\
		&\qquad\qquad\qquad
			+ \norm{u_2(t, X_2(t, \cdot))
				- u_2(t, X_1(t, \cdot))}_{L^\iny} \\
		&\qquad
		\le P(t) + \norm{\mu(\abs{X_2(t, \cdot)
			- X_1(t, \cdot)})}_{L^\iny} \\
		&\qquad
		\le C\brac{s_0 + (1 + t)  \mu(\Gamma(t))} e^{Ct} + \mu(\Gamma(t)).
	\end{split}
\end{align}

To be explicit, for a fixed $t$ and all sufficiently small $s_0$, we will have $\nu(s) = C[-(1 + t) s \log s + s].$ Calculating, we have
\begin{align*}
	t
		&= \int_{C s_0 t}^{\Gamma(t)} \frac{ds}{\nu(s)}
		= -C \int_{C s_0 t}^{\Gamma(t)}
			\frac{ds}{s((1 + t) \log s - 1)} \\
		&= - C \int_{\log(C s_0 t)}^{\log \Gamma(t)}
			\frac{dr}{(1 + t) r - 1} \\
		&= - \frac{C}{1 + t} \brac{\log ((1 + t)
			\log \Gamma(t) - 1)
				- \log ((1 + t) \log(C s_0 t) - 1)} \\
		&= \frac{C}{1 + t}\log \frac{(1 + t)
			\log(C s_0 t) - 1}{(1 + t) \log \Gamma(t) - 1}.
\end{align*}

Simplifying yields the following equation:
\begin{align*}
	\frac{(1 + t) \log \Gamma(t) - 1}{(1 + t)
		\log(C s_0 t) - 1} = e^{-C t(1 + t)},
\end{align*}
which leads to
\begin{align*}
	\log \Gamma(t)
		&= \frac{1}{1 + t} + e^{-Ct(1 + t)}
			\brac{\log(C s_0 t) - \frac{1}{1 + t}} \\
		&= C_t + e^{-Ct(1 + t)} \log(C s_0 t),
\end{align*}
where
\begin{align*}
	C_t = \frac{1 - e^{-Ct(1 + t)}}{1 + t},
\end{align*}
which we note is greater than $0$.
Thus,
\begin{align*}
	\Gamma(t)
		= e^{C_t} (C s_0 t)^{e^{-Ct(1 + t)}}.
\end{align*}

The following then holds:
\begin{align*}
	\mu(\Gamma(t))
		&= - C \Gamma(t) \log \Gamma(t) \\
		&= - e^{C_t} (C s_0 t)^{e^{-Ct(1 + t)}}
			\brac{C_t + e^{-Ct(1 + t)} \log(C s_0 t)}.
\end{align*}
Hence, from \refE{u1u2Bound},
\begin{align*}
	&\norm{u_1(t) - u_2(t)}_{L^\iny}
		\le C e^{Ct} s_0 + \mu(\Gamma(t)) \brac{1 + C(1 + t) e^{Ct}} \\
		&\qquad
		= C e^{Ct} s_0 - e^{C_t} (C s_0 t)^{e^{-Ct(1 + t)}}
			\brac{C_t + e^{-Ct(1 + t)} \log(C s_0 t)}
			\brac{1 + C(1 + t) e^{Ct}} \\
		&\qquad
		\le C e^{Ct} s_0
			- C(1 + t) e^{Ct} (C s_0 t)^{e^{-Ct(1 + t)}} \log(C s_0 t),
\end{align*}
which is \refE{u1u2BoundExplicit}. The final inequality was obtained by keeping only the dominant terms.
\end{proof}

The following is Osgood's lemma, as it appears in Lemma 5.2.1 of \cite{C1996}.
\begin{lemma}[Osgood's lemma]\label{L:Osgood}
    Let $L$ be a measurable nonnegative function and $\gamma$ a nonnegative
    locally integrable function, each defined on the interval $[t_0,
    t_1]$. Let $\mu \colon [0, \iny) \to [0, \iny)$ be a
    continuous nondecreasing
    function, with $\mu(0) = 0$ (hence, $\mu$ is a \MOC)
    and $\mu > 0$ on $(0, \iny)$.
    Let $a \ge 0$, and assume that for
    all $t$ in $[t_0, t_1]$,
    \begin{align}\label{e:LInequality}
        L(t) \le a + \int_{t_0}^t \gamma(s) \mu(L(s)) \, ds.
    \end{align}
    If $a > 0$, then for all $t$ in $[t_0, t_1]$,
    \[
        \int_a^{L(t)} \frac{ds}{\mu(s)}
            \le \int_{t_0}^t \gamma(s) \, ds.
    \]
    If $a = 0$ and $\mu$ is an \textit{Osgood} \MOC, meaning that
    \begin{align*}
		\int_0^1 \frac{ds}{\mu(s)} = \iny,
	\end{align*}
	then $L \equiv 0$.
\end{lemma}

\begin{cor}\label{C:OsgoodCor}
	Let $L$, $a$, $\gamma$, and $\mu$ be as in \refL{Osgood} with
	$[t_0, t_1] = [0, T]$ for some $T > 0$ and with $a > 0$.
	Define
	$\Gamma \colon [0, T] \to [0, \iny]$ by
    \begin{align*}
        \int_{at}^{\Gamma(t)} \frac{ds}{\mu(s)}
            = \int_0^t \gamma(s) \, ds.	
	\end{align*}
	Assume that
    \begin{align}\label{e:LCorInequality}
        L(t) \le a t + \int_0^t \gamma(s) \mu(L(s)) \, ds
    \end{align}
    for
    all $t$ in $[0, T]$. Then $L \le \Gamma$ on $[0, T]$.
\end{cor}
\begin{proof}
	Let $t \in [0, T]$.
	We have, for all $t'$ in $[0, t]$,
	\begin{align*}
		L(t') \le a t + \int_0^{t'} \gamma(s) \mu(L(s)) \, ds.
	\end{align*}
	It follows from \refL{Osgood} that for all $t'$ in $[0, t]$,
	\begin{align*}
		\int_{at}^{L(t')} \frac{ds}{\mu(s)}
            \le \int_0^{t'} \gamma(s) \, ds.
	\end{align*}
	In particular,
	\begin{align*}
		\int_{at}^{L(t)} \frac{ds}{\mu(s)}
            \le \int_0^t \gamma(s) \, ds.
	\end{align*}
	Since $\mu > 0$ on $(0, \iny)$  and $\gamma \ge 0$
	it follows that $L(t) \le \Gamma(t)$.
\end{proof}

%
%
\PART{Part IV: Estimates for the Biot-Savart kernel}

\section{Estimates for the Biot-Savart kernel}\label{S:BSEstimates}

\noindent
In this section, we derive several delicate estimates on the Biot-Savart kernel and its derivatives, estimates that we used in the proof of existence and uniqueness in \refSThrough{ExistenceR2}{UniquenessAndCont}. We follow the basic approach of employing a conformal map developed in \cite{FLI2003, ILLF}, but must extend the methods considerably to deal with higher derivatives. Because of the use of a conformal map this approach is specific to 2D domains. (The exterior of multiply connected domains could be treated as in \cite{ILLF}, at the expense of considerable extra complexity.)

In \refPThrough{JKBounds}{hlogBound} we state the estimates we used in \refSThrough{ExistenceR2}{UniquenessAndCont} in a manner that unifies, to the extent possible, the two cases of the full plane and an exterior domain. In the three subsections that follow, we give the proofs of these estimates first for the full plane in \refS{BSKernelFullPlane}, then for the exterior of the unit disk in \refS{BSKernelExteriorDisk}, and finally for a domain exterior to an obstacle---a general smooth, connected and simply connected, bounded domain with $C^\iny$ boundary---in \refS{BSKernelExteriorSingleObstacle}.

The estimates in the full plane are the simplest, as the Biot-Savart kernel, which has an explicit form, has the greatest degree of symmetry. For the exterior of the unit disk, the Biot-Savart kernel can also be written explicitly, but the presence of the boundary induces a type of distortion that complicates the estimates considerably. It is this case that will consume most of our efforts. The exterior of an obstacle can be treated by employing a conformal map provided by the Riemann mapping theorem. Because this conformal map is well-behaved we can transfer all of the key estimates for the exterior of the unit disk to apply to the exterior of the obstacle as well.

In the statement of the propositions that follow, $\Omega$ can be either the full plane or the domain exterior to a bounded simply connected domain having $C^\iny$ boundary. We recall the definitions of $K$ in \refE{KBSFullPlane}, $\KK$ in \refE{KKDef}, and $\JJ$ in \refE{JJDef}. When $\Omega = \R^2$, we have $\KK(x, y) = \JJ(x, y) = K(x - y)$. We state the bounds on $K$ separately because they are used in deriving the more general bounds on $\KK$. (Some of the inequalities stated explicitly for $K$ are slightly stronger than those for $\KK$ specialized to $\Omega = \R^2$.)

\begin{prop}\label{P:JKBounds}
	Let $a$ be a cutoff function as in (2) of \refD{ESol},
	smooth, radially symmetric, and equal to 1 in a\
	neighborhood of the origin.
	For $\eps > 0$ set $a_\eps(\cdot) = a(\cdot/\eps)$.
	
	There exists $C > 0$ such that, for all $x \in \Omega$
	and all $\eps > 0$,
	\begin{align}
		\norm{a_\eps(x - y) K(x - y)}_{L^1_y(\R^2)}
			&\le C \eps,
					\label{e:aKBound} \\
		\norm{\grad_y \grad_y
			((1 - a_\eps(x - y)) K(x - y))}_{L^1_y(\R^2)}
				&\le C \eps^{-1},
				 	\label{e:D2KBound} \\
        \norm{a_\eps(x - y) \JJ(x, y)}_{L^1_y(\Omega)}
            &\le C \eps, \label{e:aJBound} \\
		\norm{a_\eps(x - y) \KK(x, y)}_{L^1_y(\Omega)}
			&\le C (\eps + \eps^2).
								\label{e:aKKBound}
    \end{align}

	Moreover, there exists $C_0>0$ such that for all $\eps > C_0$,
    \begin{align}
        \norm{\grad_y \grad_y ((1 - a_\eps(x - y))
            	\JJ(x, y))}_{L^1_y(\Omega)}
			&\le C \eps^{-1},
		                \label{e:D2JBound} \\
            \norm{\grad_y a_\eps(x - y) \otimes
            	\grad_y \JJ(x, y)}_{L^1_y(\Omega)}
        	&\le C \eps^{-1},
								\label{e:D1JBound} \\
		\norm{\grad_y \grad_y ((1 - a_\eps(x - y)) \KK(x, y))}_{L^1_y(\Omega)}
			&\le C.
								\label{e:D2KKBound}
	\end{align}
\end{prop}

\begin{prop}\label{P:Rearrangement}
	Let $U \subseteq \Omega$ have measure $2 \pi R^2$ for some
	$R < \iny$.
	Then for any $p$ in $[1, 2)$,
	\begin{align}\label{e:RearrangementBounds}
		\begin{split}
		\smallnorm{K(x - \cdot)}_{L^p(U)}^p
			&\le \frac{R^{2 - p}}{2 - p}, \\
		\smallnorm{\KK(x, y)}_{L^p_y(U)}^p
			&\le C \frac{R^{2 - p}}{2 - p} + C R^2, \\
		\smallnorm{\JJ(x, y)}_{L^p_y(U)}^p
			&\le C \frac{R^{2 - p}}{2 - p}.
		\end{split}
	\end{align}
\end{prop}

\begin{prop}\label{P:hlogBound}
	Let $X_1$ and $X_2$ be measure-preserving homeomorphisms of $\Omega$. Let
	$\delta = \norm{X_1 - X_2}_{L^\iny}$ and suppose $\delta < e^{-1}$. Then,
	for any measurable subset $U \subset \Omega$, with finite measure,
	there exists $C>0$, depending only on $\Omega$ and
	the measure of $U$, such that
	\begin{align}\label{e:KX1X2Diff}
		\begin{split}
				\smallnorm{K(x - X_1(z)) - K(x - X_2(z))}_{L^1_z(U)}
					&\le - C \delta \log \delta, \\
			\smallnorm{\KK(x, X_1(z)) - \KK(x, X_2(z))}_{L^1_z(U)}
				&\le - C \delta \log \delta.
		\end{split}
	\end{align}
\end{prop}

%
%
\subsection{The Biot-Savart kernel in the full plane}\label{S:BSKernelFullPlane}

In this subsection we obtain the bounds in \refPThrough{JKBounds}{hlogBound} that apply specifically to the full plane. As we will see in \refS{BSKernelExteriorDisk}, the Biot-Savart kernel, $K$, for the full plane appears in the expressions for the Biot-Savart kernel, $\KK$, for the exterior of the unit disk. Not surprisingly, then, the estimates developed in this subsection are fundamental to the estimates in \refS{BSKernelExteriorDisk}.

\begin{proof}[\textbf{Proof of \refP{JKBounds} for the full plane}]
	We can easily prove \refE{aKBound} by integrating using polar coordinates centered at $x$:
	\begin{align*}
		\norm{a_\eps(x - y) K(x - y)}_{L^1_y(\R^2)}
			\le \frac{2 \pi}{2 \pi} \int_0^{C \eps} \frac{r \, dr}{r}
			= C \eps,
	\end{align*}
where $C$ is given in terms of the size of the support of $a$.

For \refE{D2KBound}, we need first to make several estimates. We begin by computing the first and second-order
derivatives of
\[-K^{\perp}(z) = N(z) \equiv \frac{z}{2\pi|z|^2}.\]
We have
\begin{align} \label{e:firstorderNj}
	\prt_{z_p} N^j(z) = \frac{\delta_{jp}}{2\pi|z|^2} - \frac{z_jz_p}{\pi|z|^4},
\end{align}
and
\begin{align} \label{e:secondorderNj}
	\prt_{z_m} &\prt_{z_p} N^j(z)
		=	-\frac{z_m\delta_{jp} +z_p\delta_{jm}+z_j\delta_{mp}}{\pi|z|^4} + 4\frac{z_jz_mz_p}{\pi|z|^6}.
\end{align}

It is clear, then, that there exists $C>0$ such that
\begin{align}\label{e:KjDerivEstsR2}
	\begin{split}
	\abs{\prt_{y_p} [K^j(x - y)]}
		&\le C \abs{x - y}^{-2}, \\
	\abs{\prt_{y_m}  \prt_{y_p} [K^j(x - y)]}
		&\le C \abs{x - y}^{-3}.
	\end{split}
\end{align}

We have,
\begin{align*}
	\grad_y &\grad_y ((1 - a_\eps(x - y)) K^j(x - y)) \\
		&= \grad_y \brac{((1 - a_\eps(x - y))) \grad_y K^j(x - y) - \grad_y a_\eps(x - y) K^j(x - y)} \\
		&= ((1 - a_\eps(x - y))) \grad_y \grad_y K^j(x - y)
			- 2 \grad_y a_\eps(x - y) \otimes \grad_y K^j(x - y) \\
		&\qquad
			- \grad_y \grad_y a_\eps(x - y) K^j(x - y).
\end{align*}

Suppose that $a$ is supported on $B_{c}$, the ball of radius $c > 0$ centered at the origin,
with $a \equiv 1$ on $B_{c'}$, the ball centered at the origin
with radius $c'$ saitisfying $0 < c' < c,$ and let
\begin{align}\label{e:Aeps}
	A_\eps(x) = \set{y \in \Omega \colon c' \eps < \abs{x - y} < c \eps}.
\end{align}
Then
\begin{align}\label{e:aepsBounds}
	\abs{\grad_y a_\eps(x - y)} \le C\eps^{-1} \text{ and }
		\abs{\grad_y \grad_y a_\eps(x - y)} \le C \eps^{-2},
\end{align}
with each function supported on $y$ in $A_\eps(x)$.

Continuing to estimate the term $\abs{\grad_y \grad_y ((1 - a_\eps(x - y)) K^i(x - y))}$, we write
\begin{align*}
	&\abs{\grad_y \grad_y ((1 - a_\eps(x - y)) K^i(x - y))}
		\le (f_1 + f_2 + f_3)(x, y),
\end{align*}
where
\[f_1 = f_1(x,y)= \abs{((1 - a_\eps(x - y))) \grad_y \grad_y K^j(x - y)},\]
\[f_2 = f_2(x,y)= 2 \abs{\grad_y a_\eps(x - y) \otimes \grad K^j(x - y)},\]
\[f_3 = f_3(x,y)= \abs{\grad_y \grad_y a_\eps(x - y) K^j(x - y)}.\]
Observe that $f_j \ge 0$, $j = 1, 2, 3$; $f_1(x, y)$ is supported on $\abs{x - y} \ge c'\eps$;
$f_2(x, y)$ and $f_3(x, y)$ are supported on $y \in A_\eps(x)$. Thus, combining the bounds we obtained we find
\begin{alignat*}{2}
	f_1(x, y)
		&\le \frac{C}{\abs{x - y}^3}
	&,\quad
	&f_2(x, y)
		\le \frac{C}{\eps \abs{x - y}^2}, \\
	f_3(x, y)
		&\le \frac{C}{\eps^2 \abs{x - y}}.
\end{alignat*}

The necessary estimates for $f_1$, $f_2$, $f_3$ can be easily derived:
\begin{align*}
	\norm{f_1(x, y)}_{L^1_y(\Omega)}
		&\le 2 \pi C \int_{c' \eps}^\iny \frac{r \, dr}{r^3}
		= \frac{C}{\eps}, \\
	\norm{f_2(x, y)}_{L^1_y(\Omega)}
		&\le \frac{2 \pi C}{\eps} \int_{c'\eps}^{c \eps} \frac{r \, dr}{r^2}
		= \frac{C}{\eps} \brac{\log(c \eps) - \log (c'\eps)}
		= \frac{C}{\eps}, \\
	\norm{f_3(x, y)}_{L^1_y(\Omega)}
		&\le \frac{2 \pi C}{\eps^2} \int_{c' \eps}^{c \eps} \frac{r \, dr}{r}
		= \frac{C}{\eps^2} \brac{c \eps - c'\eps}
		= \frac{C}{\eps}.
\end{align*}
Together these bounds yield \refE{D2KBound}, establishing the estimates for the full plane.
\end{proof}

\begin{proof}[\textbf{Proof of \refP{Rearrangement} for the full plane}]
	Since $|K(x - y)|$ is a strictly decreasing function of the distance from $x$, it follows that $\norm{K(x - \cdot)}_{L^p(U)}$ is maximized
	over all subsets $U\subset \R^2$ with $\abs{U} = 2 \pi R^2$ when $U = B_R(x)$, the ball of radius $R$ centered at $x$. Thus,
	\begin{align*}
		\smallnorm{K(x - \cdot)}_{L^p(U)}^p
			&\le \smallnorm{K(x - \cdot)}_{L^p(B_R(x))}^p
			= 2 \pi \int_0^R \frac{r \, dr}{(2 \pi)^p r^p} \\
			&= (2 \pi)^{1 - p}  \frac{R^{2 - p}}{2 - p},
	\end{align*}
	giving $\refE{RearrangementBounds}_1$.
\end{proof}

\refL{KKAbsBoundFullPlane} is used in our proof of \refP{hlogBound} for the full plane, below.

\begin{lemma}\label{L:KKAbsBoundFullPlane}
	For any $p, q \ge 1$ with $p^{-1} + q^{-1} = 1$,
	\begin{align*}
		\abs{K(x - y_1) - K(x - y_2)}
			\le \frac{2^{\frac{1}{p}} \abs{y_1 - y_2}^{\frac{1}{q}}}
				{2 \pi \min (\abs{x - y_1}, \abs{x - y_2})^{2 - \frac{1}{p}}}.
	\end{align*}
\end{lemma}
\begin{proof}

Before we begin, we mention the following identity, which we will use:
$$2\pi|K(z_{1})-K(z_{2})|=\frac{|z_{1}-z_{2}|}{|z_{1}||z_{2}|},$$
for any $z_{1}$ and $z_{2}.$  This identity may be verified by a direct calculation.

	Now, let $a = \abs{x - y_1}$, $b = \abs{x - y_2}$, and let $\theta$ be the angle between $x - y_1$
	and $x - y_2$. Then
	\begin{align*}
		2 \pi &\abs{K(x - y_1) - K(x - y_2)}
			= \frac{\abs{y_1 - y_2}}{\abs{x - y_1} \abs{x - y_2}}
			= \frac{\abs{y_1 - y_2}^{\frac{1}{p}} \abs{y_1 - y_2}^{\frac{1}{q}}}
				{\abs{x - y_1} \abs{x - y_2}} \\
			&= \frac{(a^2 + b^2 - 2 a b \cos \theta)^{\frac{1}{2p}}}{a b}
				\abs{y_1 - y_2}^{\frac{1}{q}} \\
			&= (a^{2 - 2p} b^{-2p} + a^{-2p} b^{2 - 2p}
				- 2 a^{1 - 2p} b^{1 - 2p} \cos \theta)^{\frac{1}{2p}}
				\abs{y_1 - y_2}^{\frac{1}{q}} \\
			&\le \pr{(ab)^{-2p} (a^2 + b^2 + 2 a b)}^{\frac{1}{2p}}
				\abs{y_1 - y_2}^{\frac{1}{q}}
			= \pr{\frac{(a + b)^2}{(ab)^2 (ab)^{2p - 2}}}^{\frac{1}{2p}}
				\abs{y_1 - y_2}^{\frac{1}{q}} \\
			&= \pr{\frac{(a^{-1} + b^{-1})^2}{(ab)^{2p - 2}}}^{\frac{1}{2p}}
				\abs{y_1 - y_2}^{\frac{1}{q}}
			\le \pr{\frac{(2 \min (a, b)^{-1})^2}{(\min(a, b)^2)^{2p - 2}}}^{\frac{1}{2p}}
				\abs{y_1 - y_2}^{\frac{1}{q}} \\
			&= \pr{\frac{4}{\min (a, b)^{4p - 2}}}^{\frac{1}{2p}}
				\abs{y_1 - y_2}^{\frac{1}{q}},
	\end{align*}
	from which the result follows.
\end{proof}

\begin{proof}[\textbf{Proof of \refP{hlogBound} for the full plane}]
    Set $ A = A(z) = K(x-X_1(z)) - K(x-X_2(z))$.
    It follows from \refL{KKAbsBoundFullPlane} that, for any $p$, $q > 1$,
    with $p^{-1} + q^{-1} = 1$,
    \begin{align*}
        \|A\|_{L^1_z(U)}
            &\le C \norm{\frac{\abs{X_1(z) - X_2(z)}^{\frac{1}{q}}}
                {\min(\abs{x - X_1(z)},
                    \abs{x - X_2(z)})^{2 - \frac{1}{p}}}}_{L^1_z(U)} \\
            &\le C \delta^{\frac{1}{q}}
                \sum_{j = 1}^2
                \norm{\frac{1}{\abs{x - X_j(z)}^{2 - \frac{1}{p}}}}_{L^1_z(U)}
            = C \delta^{\frac{1}{q}}
                \sum_{j = 1}^2
                \norm{\frac{1}{\abs{x - y}^{2 - \frac{1}{p}}}}_
                    {L^1_y(X_j(U))} \\
            &= C \delta^{\frac{1}{q}}
                \sum_{j = 1}^2
                \norm{K(x - y)}_{L_y^{2 - \frac{1}{p}}(X_j(U))}^
                {2 - \frac{1}{p}}.
    \end{align*}
    Let $R > 0$ be such that $|U| = 2 \pi R^2$ and apply $\refE{RearrangementBounds}_1$
    of \refP{Rearrangement} to obtain
    \begin{align*}
        \|A\|_{L^1_z(U)}
        \le Cp R^{\frac{1}{p}} \delta^{1 - \frac{1}{p}}
        \le Cp \max\{1,R\}\delta^{1 - \frac{1}{p}}.
    \end{align*}
    Whenever $\delta < e^{-1}$, this bound is minimized, relative to $p$,
    when $p = - \log \delta$, giving
    \begin{align*}
        \|A\|_{L^1_z(U)}
        \le C\max\{1,R\}(- \log \delta) \delta^{1 + \frac{1}{\log \delta}}
        = C\max\{1,R\} e (- \log \delta) \delta,
    \end{align*}
    which is $\refE{KX1X2Diff}_1$.
\end{proof}

%
%
\subsection{The Biot-Savart kernel exterior to the unit disk}\label{S:BSKernelExteriorDisk}

\noindent In this subsection we assume that the flow domain $\Omega$ is the domain exterior to the (closed) unit disk; that is,
\begin{align*}
	\Omega = \ol{B}^C \equiv \R^2\setminus \overline{B_{1}(0)}.
\end{align*}
Let $\KKBC=\KKBC(x, y) \equiv \grad_x^\perp \GGBC(x, y)$, where $\GGBC$ is the Green's function for this domain. With $K$ as in \refE{KBSFullPlane}, the Biot-Savart kernel for all of $\R^2$, we can write $\KKBC$ as in \refE{KK}:
\begin{align}\label{e:KKBC}
	\KK(x, y)
		= \KKBC(x,y) = K(x-y) - K(x-y^\ast),
\end{align}
with $y^{\ast} = y/|y|^2$.

Our next lemma gives us some limited control over how much $K(x - y^*)$ differs from $K(x - y)$.

\begin{lemma}\label{L:xyxystarBound}
	Let $x \in \R^2$ such that $|x|>1$. Then
	\begin{align}\label{e:xyxystarUpperBound}
		\frac{\abs{x - y}}{\abs{x - y^*}}
			\le \max \set{2, 2R}
			\le 2(1 + R)
	\end{align}
	for all $y$ in $\Omega$ such that $\abs{x - y} \le R$.
	Also,
	\begin{align}\label{e:xyStarBound}
		\frac{1}{\abs{x - y^*}}
			\le 2
	\end{align}
	for all $y\in \Omega$ with $\abs{x - y} \ge 1$.
\end{lemma}
\begin{proof}
	We assume
	without loss of generality that $x = (a, 0)$ lies along the positive $x$-axis. Let $y \in \Omega \cap B_R(x)$ with $\abs{y} = r$ and set $\theta$ be the angle between $y$ and $x$.
	Assume that $x \ne y$.
	Then, fixing $a$ and $r$, let
	\begin{align*}
		k(\theta)
			:= \frac{\abs{x - y}^2}{\abs{x - y^*}^2}
				= \frac{a^2 + r^2 - 2 a r \cos \theta}
					{a^2 + \frac{1}{r^2} - 2 \frac{a}{r} \cos \theta}.
	\end{align*}
	Direct calculations show that the only solutions to $k'(\theta) = 0$ are $\theta = 0$ and $\theta = \pi$, and
	that $k''(0) > 0$ while $k''(\pi) < 0$. Thus, $k$ is maximized when $\theta = \pi$.
	(The maximum may occur for $y$ on $\partial B_{1}(0)$.)
	We then write
	\begin{align*}
		k(\pi)
			&= \frac{a^2 + r^2 + 2 a r}{a^2 + r^{-2} + 2 a r^{-1}}
			= \pr{\frac{a + r}{a + r^{-1}}}^2.
	\end{align*}
	If $y$ lies along the negative real axis, then $a$ and $r$ must be less than $R$, so that
	$a + r \le 2R$. Then, since also $a + r^{-1} \ge 1$,
	we have that $k(\pi) \le 4 R^2$. If $y$ lies along the positive real axis then $r < a$,
	so $k(\pi) < (2a/a)^2 = 4$. This gives \refE{xyxystarUpperBound}.

	Similarly, letting $m(\theta) = \abs{x - y^*}^2 = a^2 + r^{-2} - 2 a r^{-1}\cos \theta$
	for fixed $a$ and $r$, we have $m'(\theta)
	= 2 a \sin(\theta)/r$ and $m''(\theta) = 2 a \cos(\theta)/r$. Thus, the minimum of $m(\theta)$ occurs
	at $\theta = 0$, where
	$
		m(0)
			= a^2 + r^{-2} - 2 a r^{-1}
			= (a - r^{-1})^2.
	$
	But if $\abs{x - y} = M \ge 1$ then $r = a + M$, so that
	\begin{align*}
		m(0)^{\frac{1}{2}}
			&= a - \frac{1}{a + M}
			\ge 1 - \frac{1}{a + 1}
			= \frac{a}{a + 1}.
	\end{align*}
	Thus,
	\begin{align*}
		\frac{1}{\abs{x - y^*}}
			\le \frac{1}{m(0)^{\frac{1}{2}}}
			= 1 + \frac{1}{a}
			\le 2,
	\end{align*}
	since $\abs{a} \ge 1$. This is \refE{xyStarBound}.
\end{proof}

\begin{lemma}\label{L:xyBound}
	For all $R \ge 2$,
	\begin{align*}
		\inf \set{\abs{x} \abs{y} \colon x, y \in \Omega, \abs{x - y} = R}
			= R - 1.
	\end{align*}
\end{lemma}
\begin{proof}
Begin by observing that, using Lagrange multipliers,
\[\min_{\{|x-y|^2=R^2,\,x,y\in \Omega\}}\{|x|^2|y|^2\} \]
is attained when either $x$ and $y$ are linearly dependent or when one of $x$ or $y$ is on the boundary $\partial \Omega$. In the latter
case, assuming without loss of generality that $|y|=1$, we have
\[|x||y|=|x| \geq |x-y| - |y| = R - 1,\]
as desired.
Otherwise, if $x$ and $y$ are linearly dependent then $x = \beta y$, $\beta \in \R$ and the result follows easily from $|x-y|=R$.
 \end{proof}

In the proof of existence we make use of a modified Biot-Savart kernel, the hydrodynamic Biot-Savart kernel \refE{JJDef}. In the case of the exterior of the unit disk this kernel is given by:
\begin{equation} \label{hydroBSdiskcase}
	\JJ(x, y)
		= \JJBC(x,y) \equiv \KKBC(x,y) + K(x).
\end{equation}

\begin{lemma}\label{L:JBound}
	Let
	\begin{align}\label{e:L}
		L(x, y)
			= K(x - y^*) - K(x).
	\end{align}
	There exists a constant $C>0$ such that, for all $x$, $y$ in $\ol{B}^C$,
	we have
	\begin{align}
		\abs{\JJ(x, y)}
			\le \frac{C}{\abs{x - y}},       \label{e:JJBound} \\
		\abs{L(x, y)}
			\le \frac{C}{\abs{x - y}}.       \label{e:LBound}
	\end{align}
\end{lemma}
\begin{proof}
	We have,
	\begin{align*}
		\abs{L(x, y)}
			&= \frac{1}{2 \pi} \abs{
				\frac{x^\perp}{\abs{x}^2}
					- \frac{x^\perp - (y^*)^\perp}{\abs{x - y^*}^2}}
			= \frac{1}{2 \pi} \frac{\abs{y^*}}{\abs{x} \abs{x - y^*}} \\
			&= \frac{1}{2 \pi} \frac{1}{\abs{y} \abs{x} \abs{x - y^*}}
			\le \frac{C(1 + \abs{x - y})}{\max \set{1, \abs{x - y} - 1}  \abs{x - y}} \\
			&= \frac{C(1 + s)}{\max \set{1, s- 1}  \abs{x - y}},
	\end{align*}
	where $s = \abs{x - y}$.
	In the last inequality, we used \refLAnd{xyxystarBound}{xyBound}.
	
	Let $g(s) = (1 + s)/\max \set{1, s- 1}$. When $s \le 2$, $g(s) \le 1 + s \le 3$, and when
	$s > 2$,
	\begin{align*}
		g(s)
			= \frac{1 + s^{-1}}{1 - s^{-1}}
			< \frac{2}{\frac{1}{2}}
			= 4.
	\end{align*}
	Hence, $\abs{L(x, y)} \le C/\abs{x - y}$. But $\JJ(x, y) = K(x - y) + L(x, y)$ and
	$\abs{K(x - y)} = C/\abs{x - y}$, so the same inequality
	applies to $J$.
\end{proof}

\begin{proof}[\textbf{Proof of \refP{JKBounds}} for $\ol{B}^C$]
	Due to \refL{JBound}, \refE{aJBound} follows directly from \refE{aKBound}.
	
	To establish \refE{D2JBound}, we need only establish it with $L$ of \refE{L}
	in place of $\JJ$, for then we can add that bound to \refE{D2KBound}.

	We have,
\begin{align*}
	\prt_{y_n} &\prt_{y_j} (K^i(x - y^*))
		\equiv \prt_n \prt_j (K^i(x - y^*))
		= -\prt_n (\prt_{y^*_k} K^i(x - y^*) \prt_j y^*_k) \\
		&= -\prt_{y^*_k} K^i(x - y^*) \prt_n \prt_j y^*_k
			- \prt_n \prt_{y^*_k} K^i(x - y^*) \prt_j y^*_k \\
		&= -\prt_{y^*_k} K^i(x - y^*) \prt_n \prt_j y^*_k
			+ \prt_{y^*_m}  \prt_{y^*_k} K^i(x - y^*) \prt_n y^*_m \prt_j y^*_k.
\end{align*}
But,
\begin{align*}
	\prt_j y^*_k
		= \prt_j \frac{y_k}{\abs{y}^2}
		= - 2 \frac{y_k}{\abs{y}^3} \prt_j \abs{y} + \frac{\delta_{jk}}{\abs{y}^2}
		= - 2 \frac{y_j y_k}{\abs{y}^4} + \frac{\delta_{jk}}{\abs{y}^2},
\end{align*}
so
\begin{align*}
	\prt_n &\prt_j y^*_k
		= \prt_n \pr{- 2 \frac{y_j y_k}{\abs{y}^4} + \frac{\delta_{jk}}{\abs{y}^2}} \\
		&= 8 \frac{y_j y_k}{\abs{y}^5} \prt_n \abs{y}
			- 2 \frac{\prt_n y_j y_k}{\abs{y}^4}
			- 2 \frac{y_j \prt_n y_k}{\abs{y}^4}
			- 2 \frac{\delta_{jk}}{\abs{y}^3} \prt_n \abs{y} \\
		&= 8 \frac{y_j y_k y_n}{\abs{y}^6}
			- 2 \frac{\delta_{jn} y_k}{\abs{y}^4}
			- 2 \frac{\delta_{nk} y_j}{\abs{y}^4}
			- 2 \frac{\delta_{jk} y_n}{\abs{y}^4}.
\end{align*}

Thus,
\begin{align*}
	\abs{\prt_n \prt_j y^*_k} \le C \abs{y}^{-3}, \quad
	\abs{\prt_n y^*_m \prt_j y^*_k} \le C \abs{y}^{-4}.
\end{align*}
Hence,
\begin{align*}
	\abs{\prt_{y_n} \prt_{y_j} (K^i(x - y^*))}
		\le C \frac{\abs{\prt_{y^*_k} K^i(x - y^*)}}{\abs{y}^3}
			+ \frac{\abs{\prt_{y^*_m}  \prt_{y^*_k} K^i(x - y^*)}}{\abs{y}^4}.
\end{align*}

Clearly, from \refEAnd{firstorderNj}{secondorderNj}
we obtain that
\begin{align*}
	\abs{\prt_{y^*_k} K^i(x - y^*)}
		&\le C \abs{x - y^*}^{-2}, \\
	\abs{\prt_{y^*_m}  \prt_{y^*_k} K^i(x - y^*)}
		&\le C \abs{x - y^*}^{-3},
\end{align*}
so that
\begin{align}\label{e:KjDerivEstsOmega}
	\begin{split}
	\abs{\prt_{y_n} (K^i(x - y^*))}
		&\le \frac{C}{\abs{x - y^*}^2 \abs{y}^2}, \\
	\abs{\prt_{y_n} \prt_{y_j} (K^i(x - y^*))}
		&\le \frac{C}{\abs{x - y^*}^2 \abs{y}^3}
			+ \frac{C}{\abs{x - y^*}^3 \abs{y}^4}.
	\end{split}
\end{align}
\Ignore{ 
and hence
\begin{align}\label{e:gradygradyKBound}
	\abs{\grad_y \grad_y (K(x - y^*))}
		\le \frac{C}{\abs{x - y^*}^2 \abs{y}^3}
			+ \frac{C}{\abs{x - y^*}^3 \abs{y}^4}.
\end{align}
} 

Then,
\begin{align*}
	\grad_y &\grad_y ((1 - a_\eps(x - y)) L^i(x, y)) \\
		&= \grad_y \brac{((1 - a_\eps(x - y))) \grad_y L^i(x, y) - \grad_y a_\eps(x - y) L^i(x, y)} \\
		&= ((1 - a_\eps(x - y))) \grad_y \grad_y K^i(x - y^*)
			- 2 \grad_y a_\eps(x - y) \otimes \grad K^i(x - y^*) \\
		&\qquad
			- \grad_y \grad_y a_\eps(x - y) L^i(x - y^*).
\end{align*}
It is only the one, final, term in which $L$ appears in place of $K$.

Thus,
\begin{align}\label{e:gradgradaKi}
	\abs{\grad_y \grad_y ((1 - a_\eps(x - y)) L^i(x, y))}
		\le (f_1 + f_2 + f_3)(x, y),
\end{align}
where
\begin{align*}
	f_1 &= f_1(x,y) = \abs{((1 - a_\eps(x - y)))
	    \grad_y \grad_y K^i(x - y^{*})}, \\
	f_2 &= f_2(x,y)= 2 \abs{\grad_y a_\eps(x - y) \otimes \grad_y K^i(x - y^{*})}, \\
	f_3 &= f_3(x,y)= \abs{\grad_y \grad_y a_\eps(x - y) L^i(x,y)}.
\end{align*}

Define $A_\eps$ as in \refE{Aeps}.
Observe, then, that $f_1$ is supported on $\abs{x - y} \ge C_1\eps$, while
$f_2(x, y)$ and $f_3(x, y)$ are supported on $y \in A_{\eps}(x)$. Because of $\refE{KjDerivEstsOmega}_2$, it is natural to decompose $f_1$ as $f_1 = f_{1, 1} + f_{1, 2}$ in such a way that
\begin{alignat*}{2}
	\abs{f_{1, 1}(x, y)}
		&\le \frac{C}{\abs{x - y^*}^2 \abs{y}^3}
	&,\quad
	&\abs{f_{1, 2}(x, y)}
		\le \frac{C}{\abs{x - y^*}^3 \abs{y}^4}.
\end{alignat*}
From \refE{aepsBounds},  \refE{LBound}, and $\refE{KjDerivEstsOmega}_1$, we obtain
\begin{alignat*}{2}
	\abs{f_2(x, y)}
		&\le \frac{C}{\eps \abs{x - y^*}^2 \abs{y}^2}
	&,\quad
	&\abs{f_3(x, y)}
		\le \frac{C}{\eps^2 \abs{x - y}}.
\end{alignat*}

Set
$F_{1, j} = \norm{f_{1, j}(x, \cdot)}_{L^1}$, $j = 1, 2$ and
$
	F_j(x)
		= \norm{f_j(x, \cdot)}_{L^1},
$
$j = 2, 3$.

The bound on $F_3$ is very simple and applies without restriction on $\eps > 0$:
\begin{align*}
	F_3(x)
		&\le \frac{C}{\eps^2} \int_{c' \eps}^{c \eps} \frac{r \, dr}{r}
		= C \frac{(c - c') \eps}{\eps^2}
		\le \frac{C}{\eps}.
\end{align*}

To bound $F_{1, 1}$, $F_{1, 2}$, and $F_2$, fix $x$ in $\Omega$ and let
\begin{align*}
	U_\eps(x) = \set{y \in \Omega \colon \abs{x - y} > \eps}.
\end{align*}
Without loss of generality assume that $c' = 1$ (see \refE{Aeps}). With this choice of $c'$, we can set $C_0 = 2$ (at the end of the proof we will reduce this to $C_0 = 1$). Our goal now is to show that \refEThrough{D2JBound}{D2KKBound} hold for all $\eps > C_0$.

Assume first that $\eps > C_0 \abs{x}$.

For all $y$ in $U_\eps(x)$, we have $\abs{y} > \eps - \abs{x} > (C_0 - 1) \abs{x} \ge C_0 - 1$ and hence $\abs{x - y^*} > \abs{x} - (C_0 - 1)^{-1} \ge 1 - (C_0 - 1)^{-1}=: \al > 0$. Thus,
\begin{align*}
	F_{1, 1}(x)
		&= \int_{U_\eps(x)} f_1(x, y) \, dy
		\le \int_{U_\eps(x)}  \frac{C}{\abs{x - y^*}^2 \abs{y}^3} \, dy
		\le \frac{C}{\al^2} \int_{U_\eps(x)}  \frac{1}{\abs{y}^3} \, dy \\
		&\le C \int_{\abs{y} > \eps - \abs{x}} \frac{1}{\abs{y}^3} \, dy
		= C \int_{\eps - \abs{x}}^\iny \frac{r \, dr}{r^3}
		= \frac{C}{\eps - \abs{x}}.
\end{align*}
But, $\abs{x} < \eps/C_0$ so $\eps - \abs{x} > \eps(1 - C_0^{-1})$. Hence,
\begin{align*}
	F_{1, 1}(x)
		\le \frac{1}{1 - C_0^{-1}} \frac{1}{\eps}
		= \frac{C}{\eps}.
\end{align*}

A similar estimate for $F_{1, 2}$ gives $F_{1, 2}(x) \le C \eps^{-2}$.
\Ignore{ 
For $F_2$, the integral over $\Omega$ is replaced by
\begin{align*}
	C \int_{c \eps - \abs{x}}^\iny \frac{r \, dr}{r^4}
		= \frac{C}{(c \eps - \abs{x})^2},
\end{align*}
which gives $F_2(x) \le C \eps^{-2} < C \eps^{-1}$ when $\eps > C_0 \abs{x}$.
} 

To estimate $F_2$, first observe that for all $y$ in $A_\eps(x)$, $C_1 \eps < \eps - \abs{x} < \abs{y} < c \eps + \abs{x} < C_2 \eps$, where $C_1 = 1 - C_0^{-1}$ and $C_2 = c + C_0^{-1}$. (The values of $C_1$ and $C_2$ come from our assumption that $\eps > C_0 \abs{x}$.) Then, since $f_2(x, \cdot)$ is supported in $A_\eps(x) \subseteq U_\eps(x)$, and $\abs{x - y^*} > \al$ for all $y$ in $U_\eps(x)$, as we observed above, we have
\begin{align*}
	F_2(x)
		&\le \frac{C}{\eps} \int_{A_\eps(x)} \frac{dy}{\abs{y}^2}
		\le \frac{C}{\eps} \int_{C_1 \eps}^{C_2 \eps} \frac{r \, dr}{r^2}
		= C \frac{\log (C_2 \eps) - \log(C_1 \eps)}{\eps}
		= \frac{C}{\eps}.
\end{align*}
Together, these bounds give \refEAnd{D2JBound}{D1JBound} when $\eps > C_0 \abs{x}$.

Now assume that $\abs{x} > 2$ and that $\eps > 0$. Then $\abs{x - y^*} \ge \abs{x} - 1 \ge \frac{1}{2} \abs{x}$,
so we can simply estimate,
\begin{align*}
	F_{1, 1}(x)
		&\le \frac{C}{\abs{x}^2} \int_1^\iny \frac{r \, dr}{r^3}
		= \frac{C}{\abs{x}^2}
		\le \frac{C}{\abs{x}}, \\
	F_{1, 2}(x)
		&\le \frac{C}{\abs{x}^3} \int_1^\iny \frac{r \, dr}{r^4}
		= \frac{C}{\abs{x}^3}
		\le \frac{C}{\abs{x}}.
\end{align*}

For $F_2$, since $A_\eps(x)$ is contained in the annulus centered at the origin of inner radius $1$ and outer radius $c \eps + \abs{x}$, we have
\begin{align*}
	F_2(x)
		&\le \frac{C}{\eps \abs{x}^2} \int_1^{c \eps + \abs{x}}
			\frac{r \, dr}{r^2}
		= \frac{C \log(c \eps + \abs{x})}{\eps \abs{x}^2}
		\le \frac{C \log(2 \max \set{c \eps, \abs{x}})}{\eps \abs{x}^2} \\
		&\le C \max \set{\frac{\log(2 c \eps)}{\eps} \frac{1}{\abs{x}^2},
			\frac{\log(2 \abs{x})}{\abs{x}} \frac{1}{\eps \abs{x}}
			} \\
		&\le C \max \set{\frac{1}{\abs{x}^2},
			\frac{1}{\eps \abs{x}}
			}
		\le C \max \set{\frac{1}{\abs{x}},
			\frac{1}{\eps}
			}.
\end{align*}

Now if $\eps \le C_0 \abs{x}$ then $\abs{x}^{-1} \le C \eps^{-1}$, and these bounds, along with the earlier bound for $\eps > C_0 \abs{x}$, give \refEAnd{D2JBound}{D1JBound} for all $\eps > 0$ when $\abs{x} > 2$.

On the other hand, if $1 < \abs{x} < 2$ then the restriction that $\eps > C_0 \abs{x}$ is satisfied if $\eps > 2 C_0$. Relabeling $2 C_0$ to be $C_0$, this gives the stated result  for all $x$ in $\Omega$ when $\eps > C_0$.

The bounds in \refEAnd{aKKBound}{D2KKBound} now follow immediately from \refEAnd{aJBound}{D2JBound} and the observations that
	\begin{align*}
		\norm{a_\eps(x - y) K(x)}_{L^1_y(\Omega)}
			\le \frac{1}{2 \pi \abs{x}} \norm{a_\eps(x - y)}_{L^1_y(\Omega)}
			\le C\eps^2,
	\end{align*}
	since $\abs{x} \ge 1$, and
	\begin{align*}
		&\norm{\grad_y \grad_y ((1 - a_\eps(x - y)) K(x))}_{L^1_y(\Omega)} \\
			&\qquad
			= \norm{\grad_y \grad_y ((1 - a_\eps(x - y)) \otimes K(x)}_{L^1_y(\Omega)} \\
			&\qquad
			\le \frac{C}{\eps^2} \pr{\int_{\supp a_\eps} 1} \abs{K(x)}
			\le C.
	\end{align*}
\end{proof}

\begin{proof}[\textbf{Proof of \refP{Rearrangement}} for $\ol{B}^C$]
	From \refE{JJBound}, $\abs{\JJ(x, y))} \le C \abs{K(x - y)}$.
	Hence, the bound on $\JJ(x, y)$ follows from
	$\refE{RearrangementBounds}_1$, which we proved in
	\refS{BSKernelFullPlane}.
	The bound on $\KK$ follows from the bound on
	$\JJ$ combined with
	\refE{JJDef} and the boundedness of $\KKol$ on $\Omega$.
\end{proof}

\begin{proof}[\textbf{Proof of \refP{hlogBound}} for $\ol{B}^C$]
    We start by using the expression for $\KK$ in \refE{KKBC} to split
    the left-hand side of $\refE{KX1X2Diff}_2$ into two terms using
    the triangle inequality:
    \begin{align}\label{e:KKlogdeltaInitBound}
        \begin{split}
        &\smallnorm{\KK(x, X_1(z)) - \KK(x, X_2(z))}_{L^1_z(U)} \\
            &\qquad
            \le \smallnorm{K(x - X_1(z)) - K(x - X_2(z))}_{L^1_z(U)} \\
            &\qquad\qquad
                + \smallnorm{K(x - (X_1(z))^*)
                    - K(x - (X_2(z))^*)}_{L^1_z(U)} \\
            &\qquad
            \le - C \delta \log \delta
                + \smallnorm{K(x - (X_1(z))^*)
                    - K(x - (X_2(z))^*)}_{L^1_z(U)}.
        \end{split}
    \end{align}
    In the last inequality, we bounded the first of the two $L^1$ norms
    using $\refE{KX1X2Diff}_1$.

    We now bound the the remaining $L^1$ norm in \refE{KKlogdeltaInitBound}.

    We first observe that for all $x, y \in \Omega$,
    \begin{align*}
        \abs{x^* - y^*}
            &= \frac{\abs{x - y}}{\abs{x} \abs{y}}
            \le \abs{x - y},
    \end{align*}
    so
    \begin{align*}
        \norm{(X_1)^* - (X_2)^*}_{L^\iny}
            \le \norm{X_1 - X_2}_{L^\iny}
            = \delta.
    \end{align*}
    It also follows from \refL{xyxystarBound} that for all
    $x, y \in \Omega$,
    \begin{align*}
        \frac{1}{\abs{x - y^*}}
            &\le \frac{2}{\abs{x - y}} \pr{ 1 + \abs{x - y}}
            \le \frac{2}{\abs{x - y}} + 2.
    \end{align*}

    With these two observations, we now proceed as in the proof
    of \refP{hlogBound} for the full plane in \refS{BSKernelFullPlane},
    setting $A = A(z) = K(x - (X_1(z))^*) - K(x - (X_2(z))^*)$.
    It follows from \refL{KKAbsBoundFullPlane} that, for any $p$, $q > 1$,
    with $p^{-1} + q^{-1} = 1$,
    \begin{align*}
        \norm{A}_{L^1_z(U)}
            &\le \norm{\frac{\abs{(X_1(z))^* - (X_2(z))^*}^{\frac{1}{q}}}
                {\min(\abs{x - (X_1(z))^*},
                    \abs{x - (X_2(z))^*})^{2 - \frac{1}{p}}}}_{L^1_z(U)} \\
            &\le \delta^{\frac{1}{q}}
                \sum_{j = 1}^2
                \norm{\frac{1}{\abs{x - (X_j(z))^*}^{2 - \frac{1}{p}}}}_
                    {L^1_z(U)} \\
            &= \delta^{1 - \frac{1}{p}}
                \sum_{j = 1}^2
                    \int_U \frac{dz}{\abs{x - (X_j(z))^*}^
                        {2 - \frac{1}{p}}} \\
            &\le \delta^{1 - \frac{1}{p}}
                \sum_{j = 1}^2
                    \int_U 2^{2 - \frac{1}{p}} 2^{2 - \frac{1}{p}}
                        \pr{\frac{1}{\abs{x - X_j(z)}^
                        {2 - \frac{1}{p}}} + 1} \, dz \\
            &= 4^{2 - \frac{2}{p}} \delta^{1 - \frac{1}{p}}
                \sum_{j = 1}^2
                    \pr{\int_U
                        \frac{dz}{\abs{x - X_j(z)}^
                        {2 - \frac{1}{p}}} + \abs{U}}.
    \end{align*}
    Let $R > 0$ be such that $|U| = 2 \pi R^2$ and apply $\refE{RearrangementBounds}_1$
    of \refP{Rearrangement} to obtain
    \begin{align*}
        \norm{A}_{L^1_z(U)}
            &\le 4^{2 - \frac{2}{p}} p R^{\frac{1}{p}}
                \delta^{1 - \frac{1}{p}}
                + 4^{2 - \frac{2}{p}} 2 \pi R^2 \delta^{1 - \frac{1}{p}}
                        \\
            &\le C \max \set{1,R^2} \delta^{1 - \frac{1}{p}} (p + 1) \\
            &= C \delta^{1 - \frac{1}{p}} (p + 1),
    \end{align*}
    where $C$ depends only on the measure of $U$.
    For $\delta < e^{-1}$, we set $p = - \log \delta$, giving
    \begin{align*}
        \norm{A}_{L^1_z(U)}
            &\le C  \delta^{1 + \frac{1}{\log \delta}} (-\log \delta + 1) \\
            &= C e \delta (- \log \delta + 1) \\
            &\le - C \delta \log \delta.
    \end{align*}

    We have now bounded both $L^1$-norms in \refE{KKlogdeltaInitBound}
    by $- C \delta \log \delta$.
    Combining the two bounds gives $\refE{KX1X2Diff}_2$.
\end{proof}

%
%
\subsection{The Biot-Savart kernel exterior to a single obstacle}\label{S:BSKernelExteriorSingleObstacle}

\noindent In this subsection, we assume that $\Omega$ is the domain exterior to a bounded simply connected domain having $C^\iny$ boundary. Denote by $B$ the open ball of radius one centered at the origin. We assume without loss of generality that $B \subseteq \Omega^C$ (else a translation and dilation would make it so). As in \cite{ILLF, FLI2003}, we have a $C^\iny$-diffeomorphism (biholomorphishm when treated as a map from and to domains in the complex plane or Riemann sphere), $T \colon \Omega \to \R^2 \setminus \ol{B}$, that extends smoothly to the boundary; see also \cite{BellKrantz1987}.
By (2.3) of \cite{FLI2003}, $DT$ and $DT^{-1}$ are both bounded above so, as observed in \cite{ILLF}, $T$ is bi-Lipschitz.

We then have
\begin{align}\label{e:KJOmega}
	\begin{split}
	\KK(x, y)
		&= \KKBC(T(x), T(y)) DT(x), \\
	\JJ(x, y)
		&= \JJBC(T(x), T(y)) DT(x)
		= \KK(x, y)
				+ \KKol(x),
	\end{split}
\end{align}
where
\begin{align}\label{e:olK}
	\KKol(x) = K(T(x)) DT(x).
\end{align}
Then because $T$ is Lipschitz and $K$ is bounded on $\Omega$,
\begin{align}\label{e:olKBound}
	\norm{\KKol}_{L^\iny} \le C.
\end{align}

We will also need an estimate on $D^2 T$.
	
\begin{lemma}\label{L:D2TyAtInf}
	For some constant, $C_1$,
	\begin{align*}
		\abs{D^2 T(y)} \le C_1 \abs{y}^{-3}.
	\end{align*}
\end{lemma}
\begin{proof}
Viewing $z$ as a complex variable, it is established in (2.1, 2.2) of \cite{FLI2003} that $T(z) = \beta z + h(z)$ for some nonzero real constant, $\beta$,  and bounded function, $h$, holomorphic on $\Omega$ (as a subset of the Riemann sphere) with $h'(z) = O(z^{-2})$ as $\abs{z} \to \iny$. It follows that $h$ is analytic at the point at infinity in the Riemann sphere and that $h''(z) = O(\abs{z}^{-3})$. If $h(z) = h(x_1, x_2) = u(x_1, x_2) + i v(x_1, x_2)$ then
\begin{align*}
	T''(z)
		&= h''(z)
		= \pdxn{u}{x_1}{2} + i \pdxn{v}{x_2}{2}
		= \pdxmixed{v}{x_1}{x_2} - i \pdxmixed{u}{x_1}{x_2}.
\end{align*}
Therefore,
$
	D^2 T(y) = O(\abs{y}^{-3})
$
as $\abs{y} \to \iny$. The results follows, then, since $T$ is $C^\iny$ up to the boundary.
\end{proof}

In \refS{BSKernelExteriorDisk}, we obtained bounds on $\JJ$ and $\KK$ for $\Omega$ the exterior of a unit disk. We now show that these bounds continue to hold for an exterior domain.

\begin{proof}[\textbf{Proof of \refPThrough{JKBounds}{hlogBound} for an exterior domain.}]
	Both \refE{aJBound} and \refE{aKKBound} of \refP{JKBounds}
	for an exterior domain follow by
	making the change of variables,
	$z = T(y)$, using the boundedness of $DT$ on $\Omega$, and the fact
	that the estimates in \refP{JKBounds} for the unit disk are uniform
	in $x$.
	The proofs of \refPAnd{Rearrangement}{hlogBound} for an exterior domain
	as well as the bound on the hydrodynamic Biot-Savart kernel,
	\refE{JJBound} of \refL{JBound}, require only the boundedness of $DT$.
	Also, \refE{D2KKBound} follows from
	\refE{D2JBound} (which we establish below)
	using the same argument as used for the
	exterior of the unit disk along with the boundedness of $DT$.
	
	It remains to prove \refE{D2JBound} and \refE{D1JBound}.
	
	Now,
	\begin{align*}
		\grad_y &\grad_y ((1 - a_\eps(x - y)) \JJ^i(x, y)) \\
			&= \grad_y \brac{((1 - a_\eps(x - y))) \grad_y \JJ^i(x, y) - \grad_y a_\eps(x - y) \JJ^i(x, y)} \\
			&= \Lambda_1 + \Lambda_2 + \Lambda_3,
	\end{align*}
	where
	\begin{align*}
			\Lambda_1
				&= ((1 - a_\eps(x - y))) \grad_y \grad_y \JJ^i(x, y), \\
			\Lambda_2
				&= - 2 \grad_y a_\eps(x - y) \otimes \grad_y \JJ^i(x, y), \\
			\Lambda_3
				&= - \grad_y \grad_y a_\eps(x - y) \JJ^i(x, y).
	\end{align*}
	By virtue of \refE{aepsBounds} and \refE{JJBound} we can bound
	the $L^1$ norm of $\Lambda_3$ as we did for
	the function, $f_3$, in the proof of \refP{JKBounds} for $\ol{B}^C$
	to conclude that
	$
		\norm{\Lambda_3(x, y)}_{L^1_y(\Omega)}
			\le C \eps^{-1}.
	$
	
	To treat $\Lambda_1$ and $\Lambda_2$ we first introduce some notation
	for differentials. Since $\JJ$ is a function
	of two variables, we will write $D_2$ to mean the differential
	with respect to the second variable.
	(So far we have been following the convention common in fluid mechanics
	of writing $\grad$ in place
	of $D$, even for a $2$-tensor.)
	
	For $\Lambda_1$ and $\Lambda_2$, we calculate,
	\begin{align*}
		\grad_y \JJ^i(x, y)
			&= \grad_y (\JJBC(T(x), T(y)) DT(x)) \\
			&= D_2 \JJBC(T(x), T(y)) DT(x) DT(y), \\
		\grad_y \grad_y \JJ^i(x, y)
			&= D_2^2 \JJBC(T(x), T(y)) DT(x) (DT(y))^2 \\
				&\qquad + D_2 \JJBC(T(x), T(y)) DT(x) D^2 T(y).
	\end{align*}
	
	Making the change of variables, $z = T(y)$, the annulus,
	$A_\eps$, of \refE{Aeps} becomes
	$
		B_\eps := \set{z \in \Omega \colon c'
			\eps < \abs{x - T^{-1}(z)} < c \eps}.
	$
	Because $T$ is bi-Lipschitz, it follows easily
	that $B_\eps$ is contained in the annulus, $A'_\eps$, centered at $u$ of inner radius,
	$c'/\norm{DT}_{L^\iny}$, and outer radius, $c \norm{DT^{-1}}_{L^\iny}$.
	
	Thus, the common support of $\Lambda_2$ and $\Lambda_3$ is distorted by making the
	change of variable, $z = T(y)$,
	and its center is moved, but the bounds in \refE{aepsBounds} still apply.
	This allows us to conclude that $\Lambda_2$ and the term,
	$$
		((1 - a_\eps(x - y))) D_2^2 \JJBC(T(x), T(y)) DT(x) (DT(y))^2,
	$$ are bounded in the $L^1$ norm by $C \eps^{-1}$ for all $\eps > C_0$.
	
	What remains is to bound the $L^1$ norm of
	\begin{align*}
		\Lambda_4 = ((1 - a_\eps(x - y)) D_2 \JJBC(T(x), T(y)) DT(x) D^2 T(y).
	\end{align*}
	Using $\refE{KjDerivEstsR2}_1$, $\refE{KjDerivEstsOmega}_1$,
	the bi-Lipschitzness of $T$, and \refL{D2TyAtInf},
	\begin{align*}
		\abs{\Lambda_4(x, y))}
			\le \brac{\frac{C}{\abs{x - y}^2} + \frac{C}{\abs{T(x) - T(y)^*}^2
			    \abs{y}^2}} \frac{1}{\abs{y}^3}.
	\end{align*}
	
	Now, on $U_\eps(x) := B_{c' \eps}^C \cap \Omega$, the support of
	$(1 - a_\eps(x - \cdot))$,
	$\abs{x - y}^{-2} < (c' \eps)^{-2}$, so
	\begin{align*}
		\int_{U_\eps(x)} &\frac{C}{\abs{x - y}^2 \abs{y}^3} \, dy
			\le \frac{C}{c' \eps^2} \int_{\Omega} \frac{dy}{\abs{y}^3}
			\le \frac{C}{\eps^2}
			\le \frac{C}{C_0 \eps}
			= \frac{C}{\eps}
	\end{align*}
	for all $\eps > C_0$. The integral above was finite since $\Omega$ does not include the origin.
	
	Making the change of variables, $z = T(y)$, we have
	\begin{align*}
		\frac{C}{\abs{T(x) - T(y)^*}^2 \abs{y}^5}
			= \frac{C}{\abs{T(x) - z^*}^2 \abs{T^{-1}(z)}^5}
			\le \frac{C}{\abs{T(x) - z^*}^2 \abs{z}^5},
	\end{align*}
	since $T$ is bi-Lipschitz. Note that $T(x)$ lies in $\ol{B}^C$,
	and in calculating the $L^1$-norm, $z$ is integrated
	over $\ol{B}^C$, while the change of variables has unit
	Jacobian determinant. Hence, we can bound
	the $L^1$-norm of this term just as we did $F_{1, 1}$ or $F_{1, 2}$ in the proof
	of \refP{JKBounds} for $\ol{B}^C$.
	This leads to
	\begin{align*}
		\int_{U_\eps(x)} \frac{C}{\abs{T(x) - T(y)^*}^2 \abs{y}^5} \, dy
			< \frac{C}{\eps}
	\end{align*}
	for all $\eps > C_0$, so that also,
	\begin{align*}
		\norm{\Lambda_4(x, y)}_{L^1_y(\Omega)}
			< \frac{C}{\eps}
	\end{align*}	
	for all $\eps > C_0$.
		
	Combining these bounds,
	we obtain \refEAnd{D2JBound}{D1JBound} for an exterior domain.
\end{proof}

%
%
\section{Examples of Serfati vorticities}\label{S:Examples}
\noindent
It is natural to ask which bounded vorticities in the plane, or in an exterior domain, are the curl of some bounded
velocity, or, in other words, to characterize vorticities which give rise to Serfati velocities; we call these {\em Serfati vorticities}. This turns out to be a surprisingly subtle issue, which will be addressed in \cite{BL2}. We discuss it briefly here for the sake of completeness. Let us start with some observations:

\begin{itemize}

\item
Any Yudovich velocity (velocity having bounded and integrable vorticity) is Serfati in the whole plane or exterior domain.

\item

Periodic vorticities, with integral zero on the period, are Serfati vorticities.

\item

Any linear combination of Serfati velocities is Serfati; that is, $S$ is a vector space. In particular, adding a bounded, compactly supported, function to a periodic vorticity whose integral vanishes on the period gives rise to a Serfati vorticity. 

\item
Take
\[
	u(x) = \frac{x^{\perp}}{|x|}
		\text{ on } \Omega = \set{x \in \R^2 \colon \abs{x} > 1}.
\]
Then $\omega(u)=\curl (u) (x) = \abs{x}^{-1}$, which is bounded but does not decay fast enough to belong to $L^p(\Omega)$ for any $p \le 2$. Hence $\omega$  does not decay fast enough for the Biot-Savart law (in the exterior of the unit disk) to converge. Nonetheless, $u$ is bounded with bounded vorticity and, hence, Serfati. Treated as a stationary solution to the Euler equations, the corresponding pressure
satisfies $\grad p = \wh{r}/r$ so that $p = \log r$, in accordance with \refR{Pressure}. This example also gives rise, by composition with a conformal map, to an example in the exterior of a general, smooth, connected domain conformally equivalent to the disk.

\item
To any vorticity that is the characteristic function of an infinite strip in $\R^2$ there corresponds a  Serfati velocity. Note that this vorticity does not decay at infinity.

For example, suppose the strip is $\set{(x_1, x_2) \colon 0 < x_2 < 1}$. Then the velocity $u$ can be chosen to vanish on $x_2 \ge 1$, equal $(1 - x_2) \wh{\bm{i}}$ on the strip, and equal $\wh{\bm{i}}$ below the strip. Treated as a stationary solution to the Euler equations, the gradient of the corresponding pressure is zero, so again $p$ is in accordance with \refR{Pressure}.

\item
If $\omega$ is Serfati in $\R^2$ and is supported away from $\Omega^C$, then $\omega$ corresponds to a Serfati velocity in $\Omega$. To see this, cut off the stream function for $\omega$ so that the resulting velocity field, $u$, is tangent to $\prt \Omega$; in fact, $u$ vanishes on $\prt \Omega$.

\item
Consider the strip $\mathsf{S} = \{(x_1,x_2) \in \mathbb{R}^2 \;|\;2 < x_2 < 3\}$ (of course one can consider an arbitrary strip with arbitrary inclination).
Then

\[u(x) =    \left\{
\begin{array}{rl}
	(1, 0)     &       \mbox{ if } x_2 > 3, \\

	(x_2 - 2, 0)       &      \mbox{ if } 2 < x_2 < 3, \\

	(0, 0)    &      \mbox{ if } x_2 < 2
\end{array}
\right.\]
is a Serfati velocity in the exterior of the unit disk. Indeed, it is divergence-free, tangent to the boundary of the disk, and its curl is $\omega = - \chi_{_{\mathsf{S}}}$, hence bounded and non-decaying at infinity. Similar constructions hold for arbitrary $\Omega$  as long as the strip is placed at a distance away from $\Omega^C$. This gives rise to a family of examples---just vary the size
of the strip, the constant flow outside the strip, and the linear interpolation.

\end{itemize}

On the other hand, consider the following very simple solution to the Euler equations:
\begin{align*}
	u(t, x) = (t, 0), \quad p(t, x) = -x_1
\end{align*}
for all $(t, x)$ in $\R \times \R^2$. (This is a special case of an example in \cite{JunKato2003}.)
Then $u$ lies in $C([0, T]; S)$ for all $T > 0$ and $\grad p = (-1, 0)$ lies in $L^\iny(\R \times \R^2)$. Nonetheless, $u$ is not a weak solution as defined in \refD{ESol}. To see this, observe first that $u^0 = 0$ and the vorticity, $\omega$, of $u$ vanishes on $\R \times \R^2$. This leaves only the term,
\begin{align*}
	\int_0^t &\int_\Omega \pr{(s, 0) \cdot \grad_y} \grad_y^\perp
								\brac{(1 - a(x - y)) \KK^j(x, y)}
						\cdot (s, 0) \, dy \, ds \\
	&=
		\int_0^t (-s^2) \int_\Omega \prt_{y_1} \prt_{y_2}
								\brac{(1 - a(x - y)) \KK^j(x, y)}
								\, dy \, ds
	= 0,
\end{align*}
on the right-hand side of \refE{SerfatiId}.
Hence, for $j = 1$, the right-hand side of \refE{SerfatiId} is zero while the left-hand side is $t$, so the Serfati identity is not satisfied. Hence, requiring that the Serfati identity hold selects certain  solutions to the Euler equations whose velocity lies in the Serfati space.

We observe also that while $\grad p$ is bounded, $p$ is not sublinear. The pressure does not satisfy property (2.20) of \cite{TaniuchiEtAl2010} that is imposed to ensure uniqueness of solutions to the Euler equations for Serfati velocities in the whole plane. (This example is discussed further in \cite{K2013}, where it is shown that sublinear growth of the pressure is equivalent to the Serfati identity and that, specifically in the full plane, these two equivalent conditions reflect the solution being expressed in an inertial reference frame.)

Finally, it is proved in \cite{BL2} that vorticities that are identically (nonzero) constants {\it are not} Serfati, since the associated velocities grow linearly at infinity. Any vortex patch whose support contains disks of arbitrary radius is also not Serfati. The vortex patch consisting of a semi-infinite strip such as, for example, the characteristic function of the set $\{(0,\infty) \times (0,1\}) \subset \mathbb{R}^2$ is not a Serfati vorticity.  As it turns out, the Serfati condition, bounded vorticity with bounded velocity, is not particularly natural from a mathematical point-of-view. A result which applies to bounded vorticities, but allows for velocities
growing linearly at infinity would be desirable, but our argument does not extend in this way.

%
%
\section{Comparison with other approaches in the full plane}\label{S:ConcludingRemarks}

\noindent Working \textit{exclusively in the full plane} there are three other existence proofs \cite{Serfati1995A, Taniuchi2004, ChenMiaoZheng2013} and two other uniqueness proofs \cite{Serfati1995A, TaniuchiEtAl2010} in the literature, and a closely related vanishing viscosity argument \cite{Cozzi2009, Cozzi2010, Cozzi2013}. The approach in \cite{Serfati1995A}, which is the closest to our own, we have discussed already, but here we say a few words concerning \cite{Taniuchi2004, TaniuchiEtAl2010, Cozzi2009, Cozzi2010, Cozzi2013, ChenMiaoZheng2013} and some related open problems.

In \cite{Taniuchi2004}, Taniuchi establishes existence of solutions for initial velocity in $S$. Actually, he does so for slightly more general initial velocity in which the vorticity can be ``slightly unbounded,'' a local version, with nondecaying initial data, of Yudovich's space defined in \cite{Y1995}, but we will discuss his argument, and that of \cite{TaniuchiEtAl2010}, only as it relates to initial data in $S$. Taniuchi employs a sequence of smooth solutions with velocities in $S$ proven to exist in another 1995 paper of Serfati \cite{Serfati1995B}. Key to Taniuchi's argument is the identity for these smooth solutions from \cite{Serfati1995B},
\begin{align}\label{e:Serfatigradp}
	\grad p
		&= \frac{1}{2 \pi} (\grad (a \log \abs{\cdot})) * \prt_i \prt_j u^i u^j
			+ \frac{1}{2 \pi} \pr{\prt_i \prt_j \grad (1 - a) \log \abs{\cdot}} * u^i u^j,
\end{align}
where $a$ lies in $C_c^\iny(\R^2)$ with $a = 1$ near the origin.

\newcommand{\LP}{\ensuremath{\Cal{P}}} %

Taniuchi uses \refE{Serfatigradp} to prove that his smooth solutions are mild solutions to the Euler equations in the sense that
\begin{align}\label{e:uMild}
	u(t)
		&= u(s) - \int_s^t \LP (u \cdot \grad u)(\tau) \, d \tau,
\end{align}
where $\LP$ is the Helmholtz operator on $\R^2$, defined in terms of Riesz transforms. Making a Littlewood-Paley decomposition, Taniuchi uses \refE{uMild} in somewhat the same way that we, following Serfati, use \refE{SerfatiId} to obtain a uniform bound on the $L^\iny$-norm of the approximating velocities.
Taniuchi
establishes a  uniform-in-viscosity bound using the vorticity equation for the Navier-Stokes equations, including the case of zero viscosity, to show that $\omega(t)$ remains bounded in $L^\iny$. (We use the transport of vorticity by the flow map for the approximate solutions to show this.) Using these uniform bounds, he ultimately obtains convergence of a subsequence to a solution of the Euler equations having sublinear growth of the pressure at infinity. This solution, however, is not shown to satisfy the Serfati identity, \refE{SerfatiId} (it is shown in \cite{K2013}, however, that it does.)

In \cite{TaniuchiEtAl2010}, the authors establish a type of continuous dependence on initial data (including uniqueness as an important special case) for solutions to the Euler equations lying in $S$. They start with the solutions to the Euler equations constructed in \cite{Taniuchi2004}, first showing that the pressure satisfies
\begin{align*}
	p = \sum_{i, j = 1}^2 R_j R_k (u^j u^k),
\end{align*}
with $p$ lying in BMO, where the $R_j$ are Riesz transforms. (That this might be the key to uniqueness is suggested by the result of  \cite{GKKM2001, JunKato2003} on uniqueness of unbounded solutions to the Navier-Stokes equations.) This identity, along with the estimates established in \cite{Taniuchi2004}, is sufficient for the authors to apply an adaptation of the fundamental uniqueness argument of Vishik in \cite{VishikBesov} to prove uniqueness (and continuous dependence on initial data) assuming that pressure grows sublinearly at infinity.

Vishik's uniqueness argument, like ours or Serfati's, does not employ an energy argument. Vishik employs in a critical way the $B^0_{\iny, 1}$-norm (and ultimately a borderline Besov space norm he defines) of the difference, $w$, between velocities. We, on the other hand, employ the $L^\iny$-norm of the flow map associated with $w$ (and so also $w$ itself). Since the $B^0_{\iny, 1}$-norm of $w$ is defined in terms of the $L^\iny$-norms of the Littlewood-Paley operators applied to $w$, these are perhaps not so far apart in spirit, though the proofs are radically different.

Properties of the flow map are used in \cite{TaniuchiEtAl2010} only for a smoothed version of the velocity field (suppressing high frequencies using a Littlewood-Paley operator)
and no vorticity is assumed to be transported by the flow. This brings up the question of whether it is possible to establish the existence of a flow map for the solutions constructed in \cite{Serfati1995A, TaniuchiEtAl2010}. That this is so is proven indirectly in \cite{K2013}, by showing that the solutions we constructed in \refS{ExistenceR2} have sublinear growth of the pressure. Since Taniuchi's uniqueness proof only relies upon this fact, Taniuchi's solutions are the same as our own, which were constructed so that the vorticity is transported by the flow map.

The recent paper \cite{ChenMiaoZheng2013} also works in larger spaces than $S$ (ones much like those of \cite{Taniuchi2004, TaniuchiEtAl2010}) and employs paradifferential calculus. Like \cite{Taniuchi2004}, their existence argument uses the smooth non-decaying solutions constructed by Serfati in \cite{Serfati1995B}, though the proof differs from that of \cite{Taniuchi2004}. Also like \cite{TaniuchiEtAl2010} (of which the authors appear unaware) their proof of uniqueness is based on Vishik's uniqueness argument. The statement of uniqueness in \cite{ChenMiaoZheng2013} is missing, however, a condition, such as sublinear growth of the pressure, that is required for uniqueness to hold. (Vishik's argument suffices in the setting of \cite{VishikBesov} because the velocity decays at infinity, which enforces sublinear pressure growth.)

There is, in effect, another proof of uniqueness for Serfati initial velocity in \cite{Cozzi2009}, where the short-time vanishing viscosity limit of solutions to the Navier-Stokes equations to a solution to the Euler equations is proved. (The short-time result in \cite{Cozzi2009} is improved to arbitrarily large finite time in \cite{Cozzi2010}, but with the additional assumption that the initial velocity is in $L^2$. This last assumption is subsequently dropped in \cite{Cozzi2013}.) The uniqueness of the solutions to the Euler equations then follow since the solutions to the Navier-Stokes equations in this setting were shown to be unique in \cite{GKM1999}. Cozzi's approach departs significantly both from our approach and that of Vishik's as employed in \cite{TaniuchiEtAl2010}. Letting $w$ be the difference between the Navier-Stokes and Euler solutions, she uses the mild formulation of the solutions to control the low frequencies of $w$,  the boundedness of vorticity to control the high frequencies, and controls the middle frequencies by reducing the problem to proving the vanishing viscosity limit in the  homogenous Besov space, $\dot{B}^0_{\iny, \iny}$
It is easier to obtain the vanishing viscosity limit in this space because
Calderon-Zygmund operators are bounded on
$\dot{B}^0_{\iny, \iny}$
but not on $L^{\infty}$.

We stress that none of the approaches to existence or uniqueness in \cite{VishikBesov, Taniuchi2004, TaniuchiEtAl2010, ChenMiaoZheng2013, Cozzi2009, Cozzi2010, Cozzi2013} is adaptable to an exterior domain because of their use of Littlewood-Paley theory and paradifferential calculus.

\appendix

%
%
\section{Approximating the initial data}\label{A:InitData}

\noindent In our proof of existence of weak solutions to the Euler equations in \refSAnd{ExistenceR2}{ExistenceExt} we employed a sequence of smooth compactly supported initial velocities that converged to a given initial velocity in the Serfati space, $S$, of \refD{SerfatiVelocity}. In this appendix, we detail the construction of that sequence.

In essence, our approach is very simple and entirely standard: apply a cutoff function to the mollified stream function for the velocity and let the support of the cutoff function increase to fill all of $\Omega$. For the full plane, this is, in fact, all that is required.

For the exterior of a single obstacle, however, there are two technical hurdles which require some work to overcome; namely, the low regularity of the space, $S$, and the presence of a boundary. In the absence of a boundary, one could simply employ convolution to smooth the stream function. Both these issues are dealt with in \refL{ApproxStream}, where we construct a sequence of smooth stream functions converging to the stream function for $u$. We then cut off this sequence in \refP{InitData} to construct our approximate sequence of initial velocities.

\begin{lemma}\label{L:ApproxStream}
Let $u \in S$. There exist $\psi \in C^1(\overline{\Omega})$ and $(\psi_n)_{n=1}^{\infty} \in C_c^{\infty}(\overline{\Omega})$  such that the following hold:
\begin{enumerate}
	\item
		$u = \nabla^{\perp}\psi$;

	\item
		$\norm{\nabla^\perp \psi_n}_{S}$ is uniformly bounded with respect to $n$;

	\item
		there exists $C>0$ such that $|\psi_n(x)|\leq C|x|$ for all $n\in\N$;

	\item
		$\psi_n = 0$ on $\partial\Omega$;

	\item
		for any $p$ in $[1, \iny)$, $\Delta \psi_n \to \Delta \psi$ in $L^p_{loc}(\Omega)$ as $n \to \iny$;
		
	\item
		$\grad \psi_n \to \grad \psi$ in $L^\iny(\Omega)$.
\end{enumerate}
\end{lemma}

\begin{proof}
	It follows from \refL{Morrey}
	that $u$ is uniformly continuous (in fact, log-Lipschitz). Because of this,
	the stream function, $\psi$, for $u$,
	which satisfies $u = \grad^\perp \psi$, $\omega = - \Delta \psi$, $\psi = 0$ on $\Gamma$,
	lies in $C^1(\Omega)$.
	We can explicitly construct this stream function, as follows. Let $x \in \Omega$. Let $T \colon \Omega
	\to \R^2 \setminus B$ be the conformal map defined in
	\refS{BSKernelExteriorSingleObstacle}. (Here, $B = B_{1}(0)$ is the unit disk centered at the
	origin.)
	Let $\gamma_x$ be the curve whose image, under $T$, is the ray joining $T(x)/|T(x)|$ and $T(x)$.
	Then set
\begin{equation} \label{psi}
	\psi
		=\psi(x)
		= - \int_{\gamma_x} u^{\perp}\circ T^{-1} DT^{-1}
			\cdot d\mathbf{s}.
\end{equation}
Since $\dv{u}=0$ it follows that $u = \nabla^{\perp}\psi$. Also, as $u$ is bounded, and we also know that $DT^{-1}$ is bounded and $T$ grows  at most linearly, we find that $\psi(x)$ grows at most linearly.

We now begin the construction of the approximate stream functions, $\psi_n$. First fix $\bar{x} \in \Gamma$.
Let $U$ be a M\"{o}bius transformation that takes the unit circle to the real axis, with the unit disk mapping to the lower half-plane, and with $U(\bar{x}) = (0, 0)$. Let $\Phi = U \circ T$. Observe that $\Phi$ is a bi-holomorphism of $\Omega$ to the upper half-plane that extends smoothly up to the boundary.

Now let $r=r_{\bar{x}} > 0$ be small enough that $\Phi(B_{r}(\bar{x})\cap \Gamma)$ is an interval around $(y_1,y_2)=(0,0)$ on the $y_1$-axis and $\Phi(B_{r}(\bar{x})\cap \Omega)$ is an open set contained in $y_2>0$. Let $s>0$ be such that
$B_{s}(0) \subset \Phi(B_{r}(\bar{x}))$. We introduce $\phi=\phi(y)$, $y \in B_{s}(0)$, as
\begin{equation}\label{oddext}
	\phi(y)
		=
		\left\{
		\begin{array}{ll}
			\psi(\Phi^{-1}(y_1,y_2)), & \mbox{ } y_2 \geq 0,\\
			- \psi(\Phi^{-1}(y_1,-y_2)), & \mbox{ otherwise}.
		\end{array}
		\right.
\end{equation}
Let $\eta \in C_c^{\infty}(\R^2)$ taking values in $[0, 1]$ have total mass $1$ and be supported in the unit ball. Assume that $\eta(y_1, -y_2) = \eta(y_1, y_2)$ and set, for each $\eps > 0$, $\eta_{\eps}(y) = \eps^{-2}\eta(\eps^{-1} y)$. Define $\phi_{\eps} = \eta_{\eps} \ast \phi$ on $B_{s/2}(0)$, for $\eps < s/2$; because $\eta_\eps$ is supported on $B_\eps(0)$, this convolution is well-defined. Finally, let
\begin{align*}
	\psi_{\eps}^{\bar{x}}
		= \psi_{\eps}^{\bar{x}} (x) = \phi_{\eps}(\Phi(x))
\end{align*}
for $x \in \Phi^{-1}(B_{s/2}(0))$.

Because $\Phi$ is a bi-holomorhpishm,
\begin{align*}
	\Delta \psi_{\eps}^{\bar{x}}
		&= \Delta \pr{(\eta_\eps * \phi)(\Phi(x))}
		= \abs{\Phi'(x)}^2 (\Delta (\eta_\eps * \phi))(\Phi(x)) \\
		&= \abs{\Phi'(x)}^2 (\eta_\eps * \Delta \phi)(\Phi(x)) \\
		&= \abs{\Phi'(x)}^2 (\eta_\eps * (\abs{(\Phi^{-1})'}^2 \Delta \psi)(\Phi(x)).
\end{align*}
Thus,
\begin{align*} 
	\norm{\Delta \psi_{\eps}^{\bar{x}}}_{L^\iny}
		\le C \norm{\Delta \psi}_{L^\iny}
		\le C.
\end{align*}

This construction defines a smooth approximation of $\psi$ locally, near the boundary. Moreover, the approximation vanishes on the portion of the boundary where it is defined, due to the fact that we performed an odd extension followed by even mollification. The next step is to use the compactness of the boundary $\Gamma$ to introduce approximations everywhere on the boundary. Denote $U^{\bar{x}} \equiv \Phi^{-1}(B_{s/2}(0))$, and
recall that $\Phi$ is also $\bar{x}$-dependent. Since $\bar{x} \in U^{\bar{x}}$ it follows that
\[\Gamma \subset \cup_{\bar{x} \in \Gamma} U^{\bar{x}}\]
is a cover of $\Gamma$ by open sets. By compactness of $\Gamma$ we can find a finite subcover $U^1, U^2, \ldots, U^N$, $\eps_0 > 0$, and corresponding approximations $\psi^1_{\eps}, \psi^2_{\eps}, \ldots, \psi^N_{\eps}$, $\eps < \eps_0$, such that each approximation $\psi^i_{\eps}$ vanishes on $U^i \cap \Gamma$ and is smooth in $U^i$.

\Ignore{ 
Now,
$
	D^\al \psi^i_\eps
		&= \eta_\eps * D^\al \phi
$
for any multi-index, $\al$.
It then follows
} 
It then follows easily that
\begin{align}\label{e:VariouspsiBounds}
	&\norm{\nabla\psi^i_{\eps}}_{L^{\infty}}
		\le C, \quad
	\norm{\Delta \psi_{\eps}^i}_{L^\iny}
		\le C 
\end{align}
and that
\begin{align*}
	\Delta \psi^i_\eps \to \Delta \psi \text{ in } L^p(U^i)
		\text{ and }
		\grad \psi^i_\eps \to \grad \psi
		\text{ in } L^\iny(U^i).
\end{align*}

Let $U^0$ be such that
\begin{align*}
	\Omega \subset \cup_{i=0}^N U^i, \quad
	U^0 \cap \Gamma = \emptyset.
\end{align*}
Consider a partition of unity $\rho^i$, $i=0,1,\ldots,N$, so that
\begin{align*}
	\sum_{i=0}^N \rho^i = 1; \quad
	0 \leq \rho^i \leq 1; \quad
	\mbox{supp }\rho^i \subset U^i; \quad
	\rho^i \in C^{\infty}(U^i).
\end{align*}

We introduce, finally,
\begin{align*}
	\psi_{\eps}
		= \psi_{\eps}(x) \equiv \sum_{i=1}^N \rho^i \psi_{\eps}^i(x) + \rho^0(x)(\eta_{\eps}\ast\psi)(x),
\end{align*}
noting that the last convolution is defined for small $\eps$ since $\rho^0$ is supported in $U^0$.

Consider $\eps = 1/n$, $n \in \N$. By construction all the desired properties hold for the corresponding $\psi_n$. In particular, we note that
\begin{align*}
	\Delta \psi_n
		&= \sum_{i=1}^N \Delta
			\rho^i \psi_n^i
			+ \rho^i \Delta \psi_n^i
			+ \grad \rho^i \cdot \grad \psi_n^i \\
		& + \Delta \rho^0 \eta_n * \psi
			+ \rho^0 \eta_n * \Delta \psi
			+ \grad \rho^0 \cdot \eta_n * \grad \psi,
\end{align*}
in light of \refE{VariouspsiBounds}, shows that $\grad^\perp \psi_n$ is bounded in $S$, since $\curl \grad^\perp = -\Delta$.
\end{proof}

In \refP{InitData}, we cut off the sequence of smooth stream functions constructed in \refL{ApproxStream} to construct our approximate sequence of initial velocities.

\begin{prop}\label{P:InitData}
	Let $u$ lie in $S$ with $\omega = \curl u$. There exists a sequence, $(u_n)_{n = 1}^\iny$,
	of approximations to $u$ with the properties that:
	\begin{enumerate}
		\item
			$u_n = \KK[\omega_n]$ and lies in $C_c^\iny(\Omega)$, where
			$\omega_n = \curl u_n$ lies in $C_c^\iny(\Omega)$;
		\item
			$u_n \to u$ uniformly on any compact subset, $L$, of $\Omega$;
		\item
			for any $p$ in $[1, \iny)$,
			$\omega_n \to \omega$ in $L^p(L)$ for any compact subset, $L$, of $\Omega$
			at a rate that depends only on $p$ and $L$;
		\item
			$u_n$ is bounded in $S$ uniformly in $n$.
	\end{enumerate}
	\end{prop}
\begin{proof}
	Let $\psi$ and $(\psi_n)$ be as given by \refL{ApproxStream}.

	\Ignore{ 
	Let $\psi$ be a stream function associated with $u$ and let $\psi_n$ be a mollified
	form of $\psi$ with the property that $\psi_n = 0$ on $\prt \Omega$,
	$\psi_n$ lies in $C^\iny(\Omega)$, $\grad \psi_n$ and $\Delta \psi_n$ are
	uniformly bounded in $L^\iny(\Omega)$, and we have
	that
	$
		\psi_n \to \psi \text{ in } W^{2, p}_{loc}(\Omega)
	$
	and $C^1_{loc}(\Omega)$. \textbf{Milton will supply the tool to do this.}
	Also, $\abs{\psi(x)} \le C \abs{x}$ and $\abs{\psi_n(x)} \le C \abs{x}$  for
	some $C > 0$.
	} 

	\newcommand{\Bdot}{\dot{B}}

	Suppose that $\ol{\Omega^C} \subseteq B_{a}(0)$, $a > 0$, and let
	$h$ be a cutoff function equal to 1 on $B_{a}(0)$ and equal to zero outside of
	$B_{2a}(0)$.
	Define $\phi_n \colon \Omega \to [0, 1]$ for $n > 1$ by
	\begin{align*}
		\phi_n(x)
			= h(a x/ n).
	\end{align*}
	Then, defining $\Bdot_n := B_n(0) \cap \Omega$, $\phi_n$ is supported on
	$\Bdot_{2n}$ and is equal to $1$ on $\Bdot_n$.
	
	Let
	\begin{align*}
		\ol{\psi}_n = \phi_n \psi_n,
			\quad u_n = \grad^\perp \ol{\psi}_n,
			\quad \omega_n = \curl u_n = \Delta \ol{\psi}_n
	\end{align*}
	and note that $\omega_n, u_n \in C_c^\iny(\Omega)$ with
	$u_n = \KK[\omega_n]$, giving (1).
	
	Let $L$ be a compact subset of $\Omega$. Then
	\begin{align*}
		\norm{u_n - u}_{L^\iny(L)}
			&= \smallnorm{\phi_n \grad^\perp \psi_n
				+ \psi_n \grad^\perp \phi_n
				- \grad^\perp \psi}_{L^\iny(L)}.
	\end{align*}
	For all sufficiently large $n$, $\phi_n = 1$ on $L$ so
	\begin{align*}
		\norm{u_n - u}_{L^\iny(L)}
			&\le \smallnorm{\grad \psi_n - \grad \psi}_{L^\iny(L)}
			\to 0
	\end{align*}
	because of (6) of \refL{ApproxStream}. This gives (2).
		
	For (3), we calculate,
	\begin{align*}
		\norm{\omega_n - \omega}_{L^p(L)}
			= \norm{\psi_n \Delta \phi_n + \phi_n \Delta \psi_n
				+ 2 \grad \phi_n \cdot \grad \psi_n - \Delta \psi}_{L^p(L)}.
	\end{align*}
	For all sufficiently large $n$, $\phi_n = 1$ on $L$ so
	\begin{align*}
		\norm{\omega_n - \omega}_{L^p(L)}
			&\le \norm{\psi_n}_{L^p(L)} \norm{\Delta \phi_n}_{L^\iny(L)}
			+ 2 \norm{\grad \phi_n}_{L^\iny(L)} \norm{u_n}_{L^p(L)} \\
			&\qquad
			+ \norm{\Delta \psi_n - \Delta \psi}_{L^p(L)} \\
			&= \norm{\Delta \psi_n - \Delta \psi}_{L^p(L)} \to 0
	\end{align*}
	by (5) of \refL{ApproxStream}. This gives (3).
	
	For (4), we have
	\begin{align}\label{e:unLInfBound}
		\begin{split}
		&\norm{u_n}_{L^\iny(\Omega)} \\
			&\qquad
			\le \smallnorm{\phi_n \grad^\perp \psi_n
				+ \psi_n \grad^\perp \phi_n}_{L^\iny(\Omega)} \\
			&\qquad
			\le \smallnorm{\phi_n}_{L^\iny(\Omega)}
					\smallnorm{\grad^\perp \psi_n}_{L^\iny(\Omega)}
				+ \norm{\psi_n}_{L^\iny(\Bdot_{2n})}
					\smallnorm{\grad^\perp \phi_n}_{L^\iny(\Omega)} \\
			&\qquad
			\le C + C n n^{-1}
			= C.
		\end{split}
	\end{align}
	Here, we used
	$
		\norm{\grad \phi_n}_{L^\iny(\Omega)}
			\le C n^{-1}.
	$
	Also,
	\begin{align*}
		\norm{\omega_n}_{L^\iny(\Omega)}
			&= \norm{\psi_n \Delta \phi_n + \phi_n \Delta \psi_n
				+ 2 \grad \phi_n \cdot \grad \psi_n}_{L^\iny(\Omega)} \\
			 &\le
			 \norm{\psi_n}_{L^\iny(\Bdot_{2n})}
			 	\norm{\Delta \phi_n}_{L^\iny(\Omega)}
			 + \norm{\phi_n}_{L^\iny(\Omega)}
			 	\norm{\Delta \psi_n}_{L^\iny(\Omega)} \\
			 &\qquad
			+ 2 \norm{\grad \phi_n}_{L^\iny(\Omega)}
				\norm{u_n}_{L^\iny(\Omega)} \\
			&\le C n n^{-2} + C + C n^{-1}
			\le C.
	\end{align*}
	Together with \refE{unLInfBound}, this yields (4).
\end{proof}

Recall the definition of the log-Lipschitz space $LL(\Omega)$ in\refE{LL}.

\begin{lemma}\label{L:Morrey}
	Suppose $u \in S$. Then $u \in LL$ with $\norm{u}_{LL} \le C \norm{u}_S$.
	\Ignore{ 
	Suppose $u$ lies in $S$ and fix $p_0$ in $(2, \iny)$.
	There exists $\delta > 0$ and a constant, $C_{p_0}$, such that
	\begin{align*} 
		\abs{u(x + y) - u(x)}
			\le - C_{p_0} \norm{u}_S \abs{y} \log \abs{y}
	\end{align*}
	for all for all $x, y$ in $\Omega$ with $\abs{y} < \delta$.
	} 
\end{lemma}
\begin{proof}
Let $\Cal{E}$ be the extension operator from $\Omega$ to $\R^2$ defined by Stein in Theorem 5' p. 181 of \cite{S1970}. This operator has the property that it continuously extends functions on all Sobolev spaces on $\Omega$ to the corresponding space on $\R^2$. Let $\psi$ be a stream function for $u$ and extend $\psi$ using $\Cal{E}$ to all of $\R^2$, also calling the extended stream function $\psi$. (If $\Omega = \R^2$ we need not perform this extension.)

Let $\phi$ be a smooth cutoff function supported in $B_{2}(0)$ with $\phi \equiv 1$ on $B_{1}(0)$ and let $\phi_x(\cdot) := \phi(\cdot - x)$. Let $\ol{u} = \grad^\perp(\phi_x \psi)$ and let $\ol{\omega} = \curl \ol{u}$.

Applying Morrey's inequality gives, for any $\abs{y} < 1$ and $p \ge p_0$,
\begin{align*}
	\abs{u(x + y) - u(y)}
		= \abs{\ol{u}(x + y) - \ol{u}(y)}
		\le C_{p_0} \norm{\grad \ol{u}}_{L^p(\R^2)}
			\abs{y}^{1 - \frac{2}{p}}.
\end{align*}
Because $\ol{\omega}$ is compactly supported, $\ol{u} = K * \ol{\omega}$. Thus, we can apply the Calderon-Zygmund inequality to obtain
\begin{align*}
	&\abs{u(x + y) - u(y)}
		\le C \inf_{p \ge p_0} \set{p \norm{\ol{\omega}}_{L^p(\R^2)} \abs{y}^{1 - \frac{2}{p}}} \\
		&\qquad
		= C \inf_{p \ge p_0} \set{p \norm{\ol{\omega}}_{L^p(B_2(x))} \abs{y}^{1 - \frac{2}{p}}}
		\le C \norm{\ol{\omega}}_{L^\iny(\R^2)} \inf_{p \ge p_0} \set{p \abs{y}^{1 - \frac{2}{p}}} \\
		&\qquad
		= - C \norm{\ol{\omega}}_{L^\iny(\R^2)} \abs{y} \log \abs{y}
\end{align*}
for all sufficiently small $y$.

But,
\begin{align*}
	\norm{\ol{\omega}}_{L^\iny(\R^2)}
		&= \smallnorm{\phi_x \omega - \grad^\perp \phi_x \cdot u}_{L^\iny(B_{2}(x))} \\
		&\le \smallnorm{\omega}_{L^\iny(B_{2}(x))}
			+ C \smallnorm{u}_{L^\iny(B_{2}(x))}
		\le C \norm{u}_S.
\end{align*}
\end{proof}

%
%
\section*{Acknowledgements}

We gratefully acknowledge support from the following funding agencies: DMA was supported
by the National Science Foundation through grants DMS-1008387 and DMS-1016267.
JPK was supported by the National Science Foundation through grants
DMS-1212141 and DMS-1009545. MCLF's research was supported by CNPq grants \# 200434/2011-0 and \# 303089/2010-5. HJNL was supported by CAPES grant  BEX 6649/10-6, CNPq grant \# 306331/2010-1, and FAPERJ grant E-26/103.197/2012.

We thank Elaine Cozzi and Yasushi Taniuchi for useful conversations. We also thank Alexis Vasseur for bringing the example in \refS{Examples} of a non-Serfati solution to our attention.


\def\cprime{$'$} \def\polhk#1{\setbox0=\hbox{#1}{\ooalign{\hidewidth
  \lower1.5ex\hbox{`}\hidewidth\crcr\unhbox0}}}

\bibliographystyle{plain}

\end{document}